\documentclass{article}
\usepackage{amsfonts,amsmath,amsthm, amssymb,latexsym}%,mathrsfs,array}
\usepackage[obeyspaces,hyphens,spaces]{url}
\usepackage{comment}
\usepackage[backref=page]{hyperref}
\usepackage{multirow}
\usepackage{geometry}
\numberwithin{table}{section}
\usepackage{imakeidx}
\makeindex[title={Index of notation},columnsep=2pt]

\def\Z{{\mathbb Z}}
\def\nm{{\mathrm {nm}}}
\def\ind{{\mathrm {ind}}}
\def\rad{{\mathrm {rad}}}

\def\cA{{\mathcal A}}
\def\cB{{\mathcal B}}
\def\cC{{\mathcal C}}
\def\cD{{\mathcal D}}
\def\cE{{\mathcal E}}
\def\cM{{\mathcal M}}
\def\cH{{\mathcal H}}

\def\GL{{\rm GL}}

\def\Stab{{\rm Stab}}

\def\Sym{{\rm Sym}}
\def\Gal{{\rm Gal}}

\def\O{{\mathcal O}}
\newcommand{\cO}{{\mathcal O}}

\def\CS{{\mathcal S}}
\def\P{{\mathbb P}}

\def\Aut{{\rm Aut}}

\def\irr{{\rm irr}}

\def\Res{{\rm Res}}
\def\Vol{{\rm Vol}}
\def\R{{\mathbb R}}
\def\F{{\mathbb F}}

\renewcommand{\L}{\mathcal{L}}

\def\FF{{\mathcal F}}

\def\Q{{\mathbb Q}}

\def\H{{\mathcal H}}

\def\U{{\mathcal U}}

\def\Z{{\mathbb Z}}

\def\P{{\mathbb P}}
\def\F{{\mathbb F}}
\def\Q{{\mathbb Q}}
\def\C{{\mathbb C}}
\def\H{{\mathcal H}}
\def\W{{\mathcal W}}
\def\Y{{\mathcal Y}}

\def\fz1{{F_{\Z,1}}}

\def\max{{\rm max}}

\def\fp{\mathfrak{p}}

\DeclareMathOperator {\sgn} {\mathrm{sgn}}

\newcommand{\besttheta}{\frac14 - \delta} % best known bound for L(1/2+it,K)

\newtheorem{theorem}{Theorem}[section]
\newtheorem{corollary}[theorem]{Corollary}

\newtheorem{lemma}[theorem]{Lemma}
\newtheorem{remark}[theorem]{Remark}
\newtheorem*{remark*}{Remark}
\newtheorem*{remarks*}{Remarks}
\newtheorem{proposition}[theorem]{Proposition}
\newtheorem{definition}[theorem]{Definition}
\theoremstyle{definition}
\newtheorem{example}[theorem]{Example}
\newtheorem*{example*}{Example}
\newtheorem*{examples*}{Examples}

\newtheorem{thm}{Theorem}

\usepackage{xcolor,color,url,lmodern}

\title{There are infinitely many non-Galois cubic fields whose
  Dedekind zeta functions have negative central value}

\title{Central values of zeta functions of non-Galois cubic fields}

\author{Arul Shankar\footnote{University of Toronto, Toronto,
    ashankar@math.toronto.edu}, Anders S{\"o}dergren\footnote{Chalmers
    University of Technology and the University of Gothenburg,
    Gothenburg, andesod@chalmers.se} and Nicolas Templier\footnote{Cornell University, Ithaca, npt27@cornell.edu}}
\date{}

\begin{document}
\maketitle

\abstract{The Dedekind zeta functions of infinitely many non-Galois cubic fields
  have negative central values.}

\tableofcontents
\section{Introduction}

Let $K$ be a number field of degree $n$, and denote its Dedekind zeta
function by $\zeta_K$. It was known to Riemann that $\zeta_\Q(\frac12)
= -1.46... < 0$.  Hecke proved that $\zeta_K(s)$ has a meromorphic
continuation with a simple pole at $s=1$ and root number  \(+1\). The generalized Riemann
Hypothesis claims that all the nontrivial zeros lie on the line $\Re
s=1/2$, which would imply that $\zeta_K(s)$ takes only negative real
values in the open interval $s\in (1/2,1)$ by the intermediate value
theorem.  This leads to the question of the \emph{possible vanishing}
of $\zeta_K(s)$ at the central point $s=1/2$.  The answer was
given by Armitage~\cite{Armitage} who showed that a certain number
field $K$ of degree $48$ constructed by Serre~\cite[\S9]{Serre}
satisfies $\zeta_K(\frac 12)=0$, and also by Fr\"ohlich~\cite{Fr} who
constructed infinitely many quaternion fields $K$ of degree $8$ such
that $\zeta_K(\frac 12)=0$. In each of these examples, $\zeta_K(s)$ factors
into Artin $L$-functions some of which have root number $-1$. Such an
$L$-function is forced to vanish at $s=1/2$ which in turn forces
$\zeta_K(\tfrac 12)=0$. 

Conversely, which conditions on $K$ can warrant that $\zeta_K(\tfrac12)$ is
non-vanishing? A conjecture of Serre \cite[Conjecture 8.24.1(2)]{Goss} claims that if $\rho$ is an irreducible  representation of  \(\Gal(M/\Q)\) for a finite Galois extension  \(M\) of $\Q$, then the Artin $L$-function $L(s,\rho)$ vanishes at the central point  \(s=1/2\) if and only if  \(\rho\) is self-dual and the root number is $-1$.  An {\it $S_n$-number field} $K$ is a degree-$n$
extension of $\Q$ such that the normal closure  \(M\) of $K$ has Galois group
$S_n$ over $\Q$.  For such a field $K$, $\zeta_K(s)$ factors as the product of
$\zeta_\Q(s)$ and an Artin $L$-function  \(L(s,\rho_K)\) which is irreducible because  \(\rho_K\) is the standard \((n-1)\)-dimensional representation of  \(S_n\), and whose root number
is $+1$ because the root numbers of both  \(\zeta_K\) and  \(\zeta_\Q\) are  \(+1\).  This conjecture of Serre (in conjunction with GRH) would thus imply that $\zeta_K(\tfrac 12)<0$
for every $S_n$-number field $K$.

In the case $n=2$, a classical result of Jutila~\cite{Jutila}
establishes that $\zeta_K(\tfrac12)$ is non-vanishing for infinitely many quadratic
number fields $K$. This was later improved in a landmark result of
Soundararajan~\cite{Sound} to a positive proportion of such fields
when ordered by discriminant, with this proportion rising to at least
$87.5\%$ in some families. In this article, we study the case
$n=3$. Our main result is as follows.

\begin{thm}\label{thm1}
The Dedekind zeta functions of infinitely many $S_3$-fields
  have negative central values.
\end{thm}

We will in fact prove a stronger version of Theorem \ref{thm1}, in
which we restrict ourselves to cubic fields satisfying any finite set
of local specifications. To state this result precisely, we introduce
the following notation. Let $\Sigma=(\Sigma_v)$ be a {\it finite set
of cubic local specifications}. That is, for each place $v$ of $\Q$,
$\Sigma_v$ is a non-empty set of \'etale cubic extensions of $\Q_v$,
such that for large enough primes $p$, $\Sigma_p$ contains all
\'etale cubic extensions of
$\Q_p$. We let $\FF_\Sigma$ denote the set of cubic fields
$K$ such that $K\otimes\Q_v\in\Sigma_v$ for each $v$. Then we have the
following result.
\index{$\Sigma=(\Sigma_v)$, finite set of local specifications}
\index{$\FF_\Sigma$, family of cubic fields prescribed by $\Sigma$}

\begin{thm}\label{thmnv1}
Let $\Sigma$ be a finite set of local specifications. Then there are
infinitely many $S_3$-fields in $\FF_\Sigma$ with negative
central value.
\end{thm}

Define $\FF_\Sigma(X)$ to be the set of fields $K\in\FF_\Sigma$ with
$|\Delta(K)|<X$. The foundational work of
Davenport--Heilbronn \cite{DH} determined asymptotics
$|\FF_\Sigma(X)|\sim \alpha_\Sigma \cdot X$ with an explicit constant
$\alpha_\Sigma >0$.

We prove quantitative versions of our main theorems, where we give
lower bounds for the \emph{logarithmic density} $\delta_\Sigma(X)$
 of the set of fields
arising in Theorem \ref{thmnv1} with bounded discriminant:
\begin{equation}
\delta_\Sigma(X):=\log\bigl|
\{
K\in \FF_\Sigma(X),\ \zeta_K(\tfrac12)<0
\}
 \bigr|/\log X.
\end{equation}
Our
\index{$\delta_\Sigma(X)$, logarithmic
density of fields $K\in \FF_\Sigma(X)$ with $\zeta_K(\tfrac12)<0$ }
 next result implies that the number of cubic $S_3$-fields whose
Dedekind zeta function is negative at the central point has
logarithmic density $\geq 0.67$:

\begin{thm}\label{thmnv2}
For any finite set $\Sigma$ of local specifications,
\begin{equation*}
  \liminf_{X\to\infty} \delta_\Sigma(X)\geq \frac{64}{95} =
  0.67368\ldots ;\quad\quad
  \limsup_{X\to\infty} \delta_\Sigma(X)\geq \frac{97}{128} =
  0.75781\ldots
\end{equation*}
\end{thm}

Note that Theorem \ref{thmnv1}
is an immediate consequence of Theorem \ref{thmnv2} since we may
add a specification $\Sigma_p$ at an additional prime $p$ that
forces all cubic fields $K\in \FF_\Sigma$ to be non-Galois.
Alternatively, we may observe that the number of Galois cubic fields $K$,
with discriminant less than $X$, is known to be asymptotic to
$cX^{\frac 12}$ by work of Cohn \cite{Cohn}, where $c$ is an explicit
constant. Hence, Theorem \ref{thmnv2} implies  that most cubic fields $K\in \FF_\Sigma(X)$ with
$\zeta_K(\frac12)<0$ must be non-Galois.

The above numerical values are established from
\begin{equation*}
  \liminf_{X\to\infty} \delta_\Sigma(X)\geq \frac{2}{3-4\delta};\quad\quad\quad
  \limsup_{X\to\infty} \delta_\Sigma(X)\geq \frac34 + \delta,
\end{equation*}
where $\delta=\frac{1}{128}$ is the current record subconvexity exponent due to 
Blomer--Khan~\cite{Blomer-Khan}, which implies
\[
 |\zeta_K(\tfrac 12)|\ll_\epsilon |\Delta(K)|^{\frac14 -\delta +\epsilon}.
\]
The \emph{convexity bound} $\delta=0$ still yields the same kind of
asymptotic results for $\delta_\Sigma(X)$, only with the weaker lower bound of $\frac23$.  
The same 
applies to all other results in this paper
so that a reader who wouldn't want to rely on the above recent
subconvexity estimate could stay with $\delta=0$. Other numerical
values for $\delta>0$ have been obtained by
Duke--Friedlander--Iwaniec~\cite{Duke-Friedlander-Iwaniec},
Blomer--Harcos--Michel
\cite[Corollary~2]{Blomer-Harcos-Michel}, and Wu \cite{Wu}.

Conditional on the Lindel\"of Hypothesis for all $\zeta_K(\frac
12)$, $K\in \FF_\Sigma$, we would have $
\lim\limits_{X\to\infty}\delta_\Sigma(X)=1$.
Even this conditional result would not imply that a positive proportion
subset of $\FF_\Sigma(X)$ is non-vanishing, it does only guarantee
the existence of $\gg_\epsilon X^{1-\epsilon}$ cubic fields
$K\in\FF_\Sigma(X)$ with $\zeta_K(\frac 12)<0$ for every $\epsilon>0$.

A cubic number field is an $S_3$-field if and only if it is not
Galois; hence we refer to non-Galois cubic fields as $S_3$-fields.
Galois cubic fields are cyclic and (as is already noted above) the
number of cyclic cubic fields $K$ of discriminant less than $X$ is
about $X^{\frac 12}$. The zeta function of a cyclic cubic field $K$
factors as a product of Dirichlet $L$-functions of conjugate cubic
characters of conductor $|\Delta(K)|^{\frac12}$
(see \S\ref{sHecke}). It follows from a result of
Baier--Young~\cite[Corollary 1.2]{Baier-Young} that for $\gg X^{\frac
37}$ cyclic cubic fields of discriminant less than $X$ the Dedekind
zeta function is negative at the central point. Recently,
David--Florea--Lalin \cite{DFL} have studied the analogous problem of
cyclic cubic field extensions of the rational function field
$\F_q(T)$, where they obtain a positive proportion of non-vanishing. Their results and methods would also yield a positive proportion of non-vanishing (conditional on GRH) for the family of cyclic cubic extensions over $\Q$. See also the papers of David--G\"ulo\u{g}lu \cite{DG}, G\"ulo\u{g}lu--Yesilyurt \cite{GY}, and G\"ulo\u{g}lu \cite{G} for analogous results for families of extensions of the Eisenstein field $\Q(\zeta_3)$.

\subsubsection*{The first moment of the central values of Artin
$L$-functions of cubic fields}

There is an extensive literature on the non-vanishing at special
points of $L$-functions varying in families. The present situation of
cubic fields is an important geometric family. Its central values are
of $\GL_2$-type and well-studied from an analytic perspective. At the
same time, the geometry of the count of cubic number fields with
bounded discriminant has a rich history.

Let $K$ be a cubic field. The Dedekind zeta function of $K$ factors as
$\zeta_K(s)=\zeta_\Q(s)L(s,\rho_K)$, where $L(s,\rho_K)$ denotes the
Artin $L$-function associated with the $2$-dimensional Galois representation
\begin{equation*}
\rho_K:\Gal(M/\Q)\hookrightarrow S_3\hookrightarrow\GL_2(\C),
\end{equation*}
where $M$ is the Galois closure of $K$.
It is known from work of Hecke that $L(s,\rho_K)$ is an
entire function.  It will be more convenient for us to work with the
central $L$-value $L(\tfrac12,\rho_K)$ rather than
$\zeta_K(\tfrac12)$, which is equivalent since they differ by the non-zero
constant $\zeta_\Q(\tfrac12)$.

In order to prove Theorem~\ref{thmnv2}, the standard approach is to
estimate the \emph{first moment} of $L(\tfrac 12,\rho_K)$ for
$K\in \FF_\Sigma$. Thus we ask the question: can one obtain an asymptotic
for
\[
\sum_{K\in \FF_\Sigma(X)} L(\tfrac 12,\rho_K),\quad \text{as $X\to \infty$}?
\]
This question is still open.  Fortunately, we observe that we may
weaken the question in the following three ways: First, we shall study the
smooth version which is technically much more convenient. Second, we
shall impose a local inert specification $\Sigma_p$ at an additional
prime $p$. Third, and this is our most important point, we observe
that it suffices that the remainder term can be expressed in terms of
central values of cubic fields with lower discriminant. Indeed, we
then have a dichotomy of either an asymptotic for the first moment or
an unusually large remainder term, either of which implies the
non-vanishing of many central values.

\begin{thm}\label{thmMoment}
There exists an absolute constant $\mu>0$ such that the following
holds.  Suppose that for some prime $p$, the specification $\Sigma_p$
consists only of the unramified cubic extension of $\Q_{p}$ (i.e.,
the cubic fields in $\FF_\Sigma$ are prescribed to be inert at $p$).  Let
$\Psi:\R_{>0}\to \C$ be a smooth compactly supported function and
suppose that $\widetilde \Psi(1) = \int_0^\infty\Psi=1$.  Then, for
every $0<\nu \le \mu$, $\epsilon>0$, and $X\ge 1$, 
\begin{equation*}
\begin{array}{rcl}
\displaystyle\sum_{K\in\FF_\Sigma}L\bigl(\tfrac12,\rho_K\bigr)
\Psi\Bigl (\frac{|\Delta(K)|}{X}\Bigr)
&=&
\displaystyle C_\Sigma \cdot X \cdot \bigl(\log X +  \widetilde \Psi'(1) \bigr)+
C'_\Sigma \cdot X 
\\[.2in]&& \displaystyle + \,O_{\epsilon,\nu,\Sigma,\Psi}\Bigl( X^{1+\epsilon-\nu}+
X^{\frac12+\epsilon}\cdot 
\sum_{K\in \FF_\Sigma\bigl (X^{\frac{3}{4}+\nu}\bigr)}
\frac{
\bigl|L\bigl(\tfrac12,\rho_K\bigr)\bigr|}
{\bigl| \Delta(K) \bigr |^{\frac 12}}
\Bigr),
\end{array}
\end{equation*}
where $C_\Sigma>0$ and $C'_\Sigma\in \R$ depend only on $\Sigma$.
\end{thm}
It is easy to see that Theorem~\ref{thmMoment}
implies that infinitely many fields $K\in\FF_\Sigma$ have nonzero
central values using an argument by contradiction. If there were
finitely many non-vanishing $L$-values, then the left-hand side would
be bounded, and the second term inside $O_{\epsilon,\nu,\Sigma,\Psi}(\cdot)$ of the
right-hand side would be bounded by $X^{\frac12+\epsilon}$.  This is a
contradiction because the term $C_\Sigma X \log X$ would be larger
than all the other terms. The fact that Theorem~\ref{thmMoment} also
implies Theorem~\ref{thmnv2} is established in
Section~\ref{sec:proof}.

The main term of Theorem~\ref{thmMoment} is familiar in the study of
moments of $L$-functions.  In particular the nature of the constants
$C_\Sigma$ and $C'_\Sigma$ is transparent, with $C_\Sigma$
proportional to the Euler product~\eqref{CSigmaProduct}.  We denote
the $n$th Dirichlet coefficient of $L(s,\rho_K)$ by $\lambda_K(n)$,
which is a multiplicative function of $n$. For a prime power $p^k$,
the coefficient $\lambda_{K}(p^k)$ depends only on the cubic \'etale
algebra $K\otimes\Q_p$ over $\Q_p$, and is in fact determined by
$\O_K\otimes\F_p$, where $\O_K$ denotes the ring of integers of
$K$. Therefore, for a fixed positive integer $n$, the asymptotic
average value of $\lambda_K(n)$ over $K\in\FF_\Sigma$ is in fact an
average over a finite set (see~\cite[\S2.11]{SST} and \cite[\S2]{SST1}
for a general discussion of this phenomenon in the context of
Sato--Tate equidistribution for geometric families). We denote this
average by $t_\Sigma(n)$ and note that this is a multiplicative
function of $n$.

We have $t_\Sigma(p)=O_\Sigma(\frac1p)$ as the prime $p\to \infty$,
which also is a general feature~\cite[\S2]{SST1} that implies that the
number field family $\FF_\Sigma$ is expected~\cite[Eq.(11)]{SST} to
have average rank $0$.  Moreover, $t_\Sigma(p^2)= 1 +
O_\Sigma(\frac1{p^2})$ for the present family $\FF_\Sigma$ which implies
that the following normalized Euler product converges:
\begin{equation}\label{CSigmaProduct}
\prod_{p} \Bigl[(1-p^{-1}) \sum^\infty_{k=0} \frac{t_\Sigma(p^k)}{p^{k/2}}\Bigr].
\end{equation}
This product is shown to be positive and to be proportional to 
$C_\Sigma$ (see Section \ref{sec:average}).

We shall discuss the remainder terms and our proof of
Theorem~\ref{thmMoment} in
\S\ref{s_overview}. An
explicit value of $\mu$ is a tenth of a thousandth.  This small
numerical value arises from the complications in bounding the
remainder terms in all of the different ranges in our proof coupled
with that the exponent of the secondary term $X^{\frac 56}$ of the
asymptotic count of cubic fields is already by itself close to $1$.

\subsubsection*{Low-lying zeros of the Dedekind zeta functions of
  cubic fields}

Our equidistribution results in Section~\ref{sec:switch} on the
asymptotic average value of $\lambda_K(n)$ over $K\in \FF_\Sigma(X)$
with robust remainder terms as $n,X\to \infty$ have applications
towards the statistics of low-lying zeros of the Dedekind zeta
functions of cubic fields (the Katz--Sarnak heuristics). A conjecture
in~\cite{SST} predicts that for a homogeneous orthogonal family of
$L$-functions, the low-lying zeros of the family should have
\emph{symplectic symmetry type}.  Given a test function $\Phi:\R\to\C$, let
$\cD(\FF_\Sigma(X),\Phi)$ denote the $1$-level density (defined
precisely in Section \ref{sec:low-lying}) of the family of Dedekind
zeta functions of the fields in $\FF_\Sigma$ with respect to
$\Phi$. Then the Katz--Sarnak heuristics predict the equality
\begin{equation}\label{onelevel}
\lim_{X\to \infty} \cD(\FF_\Sigma(X),\Phi)=\widehat{\Phi}(0)-
\frac12\int_{-1}^1\widehat{\Phi}(t)dt,
\end{equation}
for all even functions $\Phi$, whose Fourier transform
$\widehat{\Phi}$ has support contained in $(-a,a)$ for a constant $a$
to be determined. Yang
\cite{Yang} verifies~\eqref{onelevel} for even functions 
$\Phi$ whose Fourier transform has
support contained in
$(\scriptstyle{-}$$\frac{1}{50},\frac{1}{50})$. The constant
$\frac{1}{50}$ has been subsequently improved to $\frac{4}{41}$
by work of Cho--Kim 
\cite{ChoKim1} and independently~\cite{SST1}. Here, we prove
the following result:

\begin{thm}\label{thmllz}
Let $\Sigma$ be as above, with the same assumption that for at least
one prime $p$, the specification $\Sigma_{p}$ consists only of the
unramified cubic extension of $\Q_p$.  Then~\eqref{onelevel} holds for
even functions $\Phi$ whose Fourier transform has support contained in
$(\scriptstyle{-}$$\frac{2}{5},\frac{2}{5})$.
\end{thm}

\subsection{Overview of the proof of the main theorems}\label{s_overview}

These proofs are carried out in several steps. First, we control the
central value $L(\tfrac 12,\rho_K)$ using the approximate functional
equation. This allows us to approximate $L(\tfrac 12,\rho_K)$
in terms of a smooth sum of the Dirichlet coefficients $\lambda_K(n)$,
where the sum has length
$O_\epsilon(|\Delta(K)|^{1/2+\epsilon})$. More precisely, we have
\begin{equation}\label{eqintroAFE}
  L(\tfrac 12,\rho_K)=\sum^\infty_{n=1}\frac{\lambda_K(n)}{n^{1/2}}
  V^\pm\Bigl(\frac{n}{\sqrt{|\Delta(K)}|}\Bigr),
\end{equation}
where $V^\pm$ is a rapidly decaying smooth function depending only on the
sign $\pm$ of $\Delta(K)$.  Therefore, studying the average value of
$L(\tfrac 12,\rho_K)$ as $K$ varies over the family $\FF_\Sigma(X)$ of cubic fields
with discriminant bounded by $X$ necessitates the study of smoothed
sums of Dirichlet coefficients $\lambda_K(n)$:
\begin{equation}\label{eqintrofirstsum}
\sum_{n\le X^{1/2+\epsilon}}\frac{1}{n^{1/2}}
\sum_{K\in\FF_\Sigma}\lambda_K(n)\Psi\Bigl(\frac{|\Delta(K)|}{X}\Bigr),
\end{equation}
where $\Psi:\R_{>0}\to\C$ is a smooth function with compact
support. In particular, a basic input for the proof is the
determination of the average value $t_\Sigma(n)$ of $\lambda_K(n)$ over
$K\in\FF_\Sigma(X)$. Moreover, it is necessary to obtain good 
error terms for
this average with an explicit dependence on $n$.

\subsubsection*{Expanding the definition of $\lambda_K(n)$ to cubic rings $R$}

In order to compute the average value of $\lambda_K(n)$ over
$K\in\FF_\Sigma$ with good error terms, it is necessary for us to
expand this average to one over cubic orders $R$. This is because cubic
rings can be parametrized by group orbits on a lattice and Poisson
summation, applied through the theory of Shintani zeta functions
following Taniguchi--Thorne~\cite{TT} and \cite{TaTh1}, becomes
available as an important tool.\footnote{This is in direct analogy to
the quadratic case, in which P\'olya--Vinagradov type estimates are
used to estimate the sum of Legendre symbols
$\bigl(\frac{n}{D}\bigr)$, as $D$ varies over all discriminants and
not merely the squarefree ones.} It is therefore necessary for us to
define a quantity $\lambda_R(n)$, for positive integers $n$ and cubic
rings $R$. There are different natural choices for the value of
$\lambda_R(n)$. For example, it is possible to set the Dirichlet
coefficients of $R$ to be equal to the corresponding coefficients of
$R\otimes\Q$. Another possible choice arises from work of
Yun \cite{Yun}, in which Yun defines a natural zeta function
$\zeta_R(s)$ associated to orders $R$ in global fields. It is then
possible to set the Dirichlet coefficients of $R$ to equal the
corresponding coefficients of $\zeta_R(s)/\zeta(s)$.

However, we require $\lambda_R(n)$ to satisfy the following three
conditions:
\begin{itemize}
\item[{\rm (a)}] We require $\lambda_R(n)=\lambda_K(n)$ when $R$ is
  the ring of integers of $K$.
\item[{\rm (b)}] We require $\lambda_R(n)$ to be multiplicative in
  $n$.
\item[{\rm (c)}] When $p$ is prime, we require the value of
  $\lambda_R(p^k)$ to be defined modulo $p$, i.e., $\lambda_R(p^k)$
  should be determined by $R\otimes\F_p$.
\end{itemize}
The above two candidate choices for $\lambda_R(n)$ satisfy the first
two properties, but not the third. In fact, the above three conditions
uniquely determine the value of $\lambda_R(p^k)$ for rings
$R$ such that $R\otimes\Z_p$ is Gorenstein, in the sense that ${\rm Hom}(R,\Z_p)$ is free.\footnote{Non Gorenstein rings $R$ over $\Z_p$ are those such that $R\otimes\F_p$ is of the form $\langle1,x,y\rangle$ with $x^2=y^2=xy=0$ (see \cite{GGS}).} More precisely,
$\lambda_R(n)$ should be defined to be the $n$th Dirichlet coefficient
of $D(s,R)$, where $D(s,R)$ is defined by an Euler product whose $p$th
factor $D_p(s,R)$ is given by
\begin{equation}\label{eqEPR}
  D_p(s,R):=\left\{
  \begin{array}{lll}
    (1-p^{-s})^{-2}&\;{\rm if}\;&R\otimes\F_p=\F_p^3;\\[.15in]
    (1-p^{-2s})^{-1}&\;{\rm if}\;&R\otimes\F_p=\F_p\oplus\F_{p^2};\\[.15in]
    (1+p^{-s}+p^{-2s})^{-1}&\;{\rm if}\;&R\otimes\F_p=\F_{p^3};\\[.15in]
    (1-p^{-s})^{-1}&\;{\rm if}\;&R\otimes\F_p=\F_p\oplus\F_p[t]/(t^2);\\[.15in]
    1&\;{\rm else.}\;&
  \end{array}\right.
\end{equation}
It is clear from the definition that $\lambda_R(n)$ satisfies the
three required properties.

\subsubsection*{Summing $\lambda_R(n)$ over cubic rings $R$ with bounded 
discriminant}

Next, we need to evaluate a smoothed sum of $\lambda_R(n)$, for $R$
varying over cubic rings having bounded discriminant. Such a result
follows immediately from the following three ingredients. First, the
Delone--Faddeev parametrization of cubic rings in terms of
$\GL_2(\Z)$-orbits on $V(\Z)$, the space of integral binary cubic
forms. Second, results of Shintani \cite{Shintani} on the analytic
properties of the Shintani zeta functions associated to
$V(\Z)$. Third, local Fourier transform computations of
Mori \cite{Mori} on $V(\F_p)$.

Let $n$ be a positive integer, and write $n=mk$, where $m$ is
squarefree, $k$ is powerful, and $(m,k)=1$.  Then we have the
following result, stated for primes and prime powers as
Theorem \ref{t_Polya}, which is a smoothed cubic analogue of the
P\'olya--Vinogradov inequality:
There exist explicit constants $\alpha(n)$ and $\gamma(n)$ such that
\begin{equation}\label{thmpv}
\displaystyle\sum_{[R:\Z]=3}\lambda_R(n)\Psi\Bigl(\frac{|\Delta(R)|}{X}\Bigr)=
\alpha(n) X+
\gamma(n) X^{5/6}+O_\epsilon\big(n^\epsilon \cdot m\cdot \rad(k)^2\big),
\end{equation}
where
$\rad(k)$ denotes the radical of $k$, and the sum over rings is weighted by the inverse of 
the size of the stabilizer, $|\operatorname{Stab}(R)|^{-1}$.

\subsubsection*{Sieving to maximal orders}

We define the quantity
\begin{equation*}
S(R)=\sum_{n}\frac{\lambda_R(n)}{n^{1/2}}
  V^\pm\Bigl(\frac{n}{\sqrt{|\Delta(R)|}}\Bigr).
\end{equation*}
We note that $S(R)=L(\tfrac 12,\rho_K)$ when $R$ is the ring of integers of
$K$. However, when $R$ is not maximal, it is {\em not} necessarily
true that $S(R)$ is equal to $D(\tfrac 12,R)$. In order to evaluate
\eqref{eqintrofirstsum}, we need to perform an inclusion-exclusion
sieve. Thus, for all squarefree integers $q$, we need estimates on the
sums
\begin{equation}\label{eqintroFS}
\sum_{R\in\cM_q}S(R)\Psi\Bigl(\frac{|\Delta(R)|}{X}\Bigr),
\end{equation}
where $\cM_q$ denotes the space of cubic orders $R$ that have index
divisible by $q$ in the ring of integers of $R\otimes\Q$.  Estimating
sums over $\cM_q$ is tricky since the condition of nonmaximality at
$q$ is defined modulo $q^2$ and not modulo $q$. That is, maximality of
$R$ at a prime $p$ cannot be detected from the local algebra
$R\otimes\F_p$. To reduce our mod $q^2$ sum to a mod $q$ sum, we use
an idea originating in the work of Davenport--Heilbronn \cite{DH} and
further developed as a precise switching trick in \cite{BST}. Namely,
we replace the sum over $\cM_q$ with a sum over the set of overorders of
$\cM_q$ of index-$q$.

For $q$ in what we call the ``small range'', i.e., $q\le
X^{1/8-\epsilon}$, the switching trick in conjunction with
\eqref{thmpv} allows us to estimate each summand in \eqref{eqintroFS}
with a sufficiently small error term. Ideally, we would use a tail
estimate for large $q$. This tail estimate requires bounding the value
of $S(R)$ for nonmaximal rings $R$. The convexity bound yields the
following estimate for rings $R\in\cM_q$ with $\Delta(R)\asymp X$:
\begin{equation}\label{eqintroconvexity}
|S(R)|\ll_\epsilon \frac{X^{1/4+\epsilon}}{q^{1/2}}.
\end{equation}
Neither the convexity bound nor the best known subconvexity bounds
give sufficiently good estimates to cover all squarefree integers
$q> X^{1/8-\epsilon}$. However, assuming the generalized Lindel\"of
Hypothesis (or indeed, a sufficiently strong subconvexity bound) is
enough to determine the first moment for $L(\tfrac
12,\rho_K)$. Moreover, this method yields unconditional upper bounds
on the average value of $L(\tfrac 12,\rho_K)$, a slightly stronger
version of which is proven in Theorem~\ref{thuncondupbd}:

\begin{thm}\label{introupper}
Let $\Sigma$ be a finite set of local specifications and assume that
for some prime $p$, we have $\Sigma_p=\{\Q_{p^3}\}$. Then for $X\geq
1$, we have
\begin{equation}\label{equncondbound}
\sum_{K\in\FF_\Sigma}L\bigl(\tfrac12,\rho_K\bigr)
\Psi\Bigl(\frac{|\Delta(K)|}{X}\Bigr)\ll_{\Sigma,\Psi} X^{29/28}.
\end{equation}
\end{thm}

We note that this average bound is significantly stronger than the
bound obtained by simply summing the best known pointwise upper bounds
for $L(\tfrac12,\rho_K)$.

\subsubsection*{The approximate functional equation for cubic rings}

The first ingredient required for estimating $S(R)$, when $R$ is a
nonmaximal cubic order with index $> X^{1/8-\epsilon}$, is a
generalization of the approximate functional equation \eqref{eqintroAFE}
to the setting of cubic orders. This modification is proved in
Proposition~\ref{thm_AFE2}, and expresses $S(R)-D(\tfrac 12,R)$ as a
sum of arithmetic quantities associated to $R$.  The advantage of
expressing $S(R)$ in this way is that this latter sum is much shorter
than the original sum defining $S(R)$: of length $\ll_\epsilon
X^{1/2+\epsilon}/q$ rather than $\ll_\epsilon
X^{1/2+\epsilon}$. However, this shortening comes at a cost. The
summands of this new sum involve Dirichlet coefficients from both
$D(s,R)$ and $L(s,\rho_{R\otimes\Q})$.

In order to control the coefficients of $L(s,\rho_{R\otimes\Q})$, it
is necessary to isolate the exact index of $R$ in the ring of integers
of $R\otimes\Q$. Merely knowning that $q$ divides the index is not
enough. To precisely control the index, a secondary sieve is
necessary. Carrying out this secondary sieve yields the following
estimate for $q > X^{1/8-\epsilon}$:
\begin{equation}\label{eqintroSD}
\sum_{R\in \cM_q}S(R)\Psi\Bigl(\frac{|\Delta(R)|}{X}\Bigr)
\approx
\sum_{R\in \cM_q}D(\tfrac 12,R)\Psi\Bigl(\frac{|\Delta(R)|}{X}\Bigr).
\end{equation}
This estimate is proved in Section~\ref{sConditional}, and is the
crucial technical ingredient in the proof of
Theorem~\ref{thmMoment}. Equation
\eqref{eqintroSD} allows us to exploit the advantages of using $S(R)$
and $D(\tfrac 12,R)$ in the original inclusion exclusion sieve. Namely, for
small $q$, the sum of $S(R)$ over $R\in\cM_q$, can be well estimated
with Equation~\eqref{thmpv} since $S(R)$ is simply a sum of the
coefficients $\lambda_R(n)$. However for large $q$, it is advantageous
to instead sum $D(\tfrac 12,R)$ over $R\in\cM_q$. This is because the value
of $D(\tfrac 12,R)$ behaves predictably as $R$ varies over suborders of a
fixed cubic field.

\subsubsection*{Summing $D(\tfrac 12,R)$ over $R\in\cM_q$ and over large 
$q$}

We are left to estimate the sum
\begin{equation}\label{eqintroD12sum}
  \sum_{q> X^{1/8-\epsilon}}\mu(q)\sum_{R\in\cM_q}D(\tfrac 12,R)
  \Psi\Bigl(\frac{|\Delta(R)|}{X}\Bigr).
\end{equation}
Expressing $D(\tfrac 12,R)$ in terms of $L(\tfrac
12,\rho_{R\otimes \Q})$ allows us to repackage
\eqref{eqintroD12sum} into sums of
the following form:
\begin{equation}\label{eqintrorepack}
  \sum_{\substack{K\in\FF_\Sigma\\|\Delta(K)|\asymp Y}}
  \sum_{\substack{R\subset\O_K\\\ind(R)\asymp\sqrt{X/Y}}} D(\tfrac 12,R)
  \ll_{\epsilon,\Sigma}
  X^\epsilon
  \sum_{\substack{K\in\FF_\Sigma\\|\Delta(K)|\asymp Y}}
  \#\bigl\{R\subset\O_K:\ind(R)\asymp\sqrt{X/Y}\bigr\}\cdot
  |L(\tfrac 12,\rho_K)|.
\end{equation}
Let $K$ be a fixed cubic field. A result of Datskovsky--Wright
\cite{DW} gives asymptotics for the number of suborders of $K$ having
bounded index. This yields Theorem~\ref{thmMoment}.

Our next idea is to assume the nonnegativity of
$L(\tfrac12,\rho_K)$. Since the result of Datskovsky--Wright is very
precise, it turns out that we can input the unconditional upper bound
on the sums of $L(\tfrac12,\rho_K)$ in \eqref{equncondbound}, to
obtain an improved upper bound on the right-hand side
of \eqref{eqintrorepack}. This improved upper bound is enough to
obtain asymptotics for the first moment of $L(\tfrac12,\rho_K)$,
conditional on its nonnegativity.

Finally, we obtain Theorem \ref{thmnv2} by making a version of the
following simple idea precise: If $L(\tfrac 12,\rho_K)$ does indeed
vanish for most fields $K$, then the right-hand side
of \eqref{eqintrorepack} is forced to be small, which in turn implies
an upper bound on the left-hand side of \eqref{eqintrorepack}, which
in turn allows for the computation of the first moment of
$L(\tfrac12,\rho_K)$, which in turn implies non-vanishing for many
fields $K$. This leads to a contradiction, and it follows that
$L(\tfrac 12,\rho_K)$ does not vanish for many fields $K$.

\medskip

Finally, we observe that the same method of proof applies to the
values $L(\tfrac12+it,\rho_K)$ for a fixed $t\in \R$ and yield
variants of
Theorems~\ref{thm1}, \ref{thmnv1}, \ref{thmnv2}, \ref{thmMoment},
\ref{introupper}
with suitable modifications.

\subsection{Organization of the paper}
This paper is organized as follows. In Section~\ref{sec:pre}, we
collect preliminary results on the space of cubic rings and fields. In
particular, we recall the Delone--Faddeev parametrization of cubic
rings in terms of $\GL_2(\Z)$-orbits on integral binary cubic
forms. We also discuss Fourier analysis on the space of binary cubic
forms over $\F_p$ and $\Z/n\Z$. In Section~\ref{sec:Artin}, we
introduce the Artin character on cubic fields $K$ that arise as
Dirichlet coefficients of $L(s,\rho_K)=\zeta_K(s)/\zeta(s)$. We then
define an extension to the space of cubic rings (and thus also the
space of binary cubic forms).  Next, in Section~\ref{sec:dedekind}, we
recall the analytic properties of $L(s,\rho_K)$, for a cubic field
$K$. In particular, we recall the approximate functional equation. We
then discuss an unbalanced form of the approximate functional equation
for orders within cubic fields.

In Section~\ref{secszf}, we recall Shintani's theory of the zeta
functions associated to the space of binary cubic forms. As a
well-known consequence of this theory, we derive estimates for the
sums of congruence functions (i.e., functions $\phi$ on the space of
cubic rings $R$ such that $\phi$ is determined by $R\otimes\Z/n\Z$ for
some integer $n$) over the space of cubic rings with bounded
discriminant. Then in Section~\ref{sec:switch}, we apply a squarefree
sieve to determine the sum of these congruence functions over the
space of cubic fields.

In Section~\ref{sec:low-lying}, we use the results from
Section~\ref{sec:switch} to prove Theorem \ref{thmllz} on the
statistics of the low-lying zeros of the zeta functions of cubic
fields. Next, in Section~\ref{sec:average}, we start our analysis of
the average central values of $L(s,\rho_K)$, where $K$ ranges over
cubic fields. In particular we prove the upper bound
Theorem~\ref{thuncondupbd}, obtaining an improved estimate on the
average size of $L(\tfrac 12,\rho_K)$ compared to the pointwise bound.

In Section~\ref{sConditional}, we complete the most difficult part of
the proof, in which we show that for each somewhat large $q$, the
values of $S(R)$ and $D(\tfrac 12,R)$ are close to each other, on
average over $R\in\cM_q$. We use this result in
Section~\ref{sec:proof} to first prove Theorem~\ref{thmMoment}, and
using this in addition, to prove our main result Theorem~\ref{thmnv2}.

\subsection{Notations and conventions}

\begin{itemize}
\item A positive integer $k$ is said to be \emph{powerful} if $v_p(k)\ge 2$
for every prime $p|k$.
\index{$v_p(k)\ge 2$ for every $p\mid k$, powerful integer}

\item The \emph{radical}, also called the
square-free kernel, of a positive integer $k$ is the product of its
prime factors, $\rad(k):=\prod_{p|k} p$.
\index{$\rad(k)$, radical of the positive integer $k$}

\item We shall always use $\Sigma$ to refer to the finite
set of local conditions imposed on the family of cubic fields.

\item We shall always use $\Psi$ to denote a compactly
supported Schwartz function that will control the discriminants
of binary cubic forms, cubic rings, or cubic fields.

\end{itemize}

\subsection*{Acknowledgements}
We are very grateful to the anonymous referee for a thorough reading and several valuable comments and suggestions. We are also happy to thank Chantal David, Mehmet Durlanik, and Jacob Tsimerman for many helpful conversations.

ASh is supported by an NSERC discovery grant and a Sloan fellowship.
ASo was supported by the grant 2016-03759 from the Swedish Research
Council. NT acknowledges support by NSF grants DMS-1454893 and DMS-2001071.

\section{Preliminaries on cubic rings and fields}\label{sec:pre}

Let $V=\Sym^3(2)$ denote the space of binary cubic forms. The group
$\GL_2$ acts on $V$ via the following twisted action:
\begin{equation*}
\gamma \cdot f(x,y) := \det(\gamma)^{-1} f((x,y)\cdot \gamma).
\end{equation*}
\index{$V$, space of binary cubic forms with twisted action by $\GL_2$}
It is well-known that the representation $(\GL_2,V)$ is {\it
  prehomogeneous} and that the ring of relative invariants for the
action of $\GL_2$ on $V$ is freely generated by the {\it discriminant}
which we denote by $\Delta$.  We have that $\Delta$ is homogeneous of
degree $4$ and $\Delta(\gamma \cdot f)=(\det \gamma)^2 \Delta(f)$. In this
section, we describe the parametrization of cubic rings and fields in
terms of $\GL_2(\Z)$-orbits on $V(\Z)$.  We also discuss Fourier
analysis on the space $V(\Z/n\Z)$, and in particular describe the
Fourier transforms of all $\GL_2(\F_p)$-invariant functions on $V(\F_p)$.

\subsection{Binary cubic forms and the parametrization of cubic 
rings}\label{s_binary_cubic}

Levi \cite{Levi} and Delone--Faddeev \cite{DF}, further refined by
Gan--Gross--Savin \cite{GGS}, prove that there is a bijection between
the set of $\GL_2(\Z)$-equivalence classes of integral binary cubic
forms and isomorphism classes of cubic rings over~$\Z$:

\begin{proposition}\label{df}
  There is a bijection between the set of isomorphism classes of cubic
  rings and the set of $\GL_2(\Z)$-orbits on $V(\Z)$, given as
  follows. A cubic ring $R$ is associated to the
  $\GL_2(\Z)$-equivalence class of the integral binary cubic form
  corresponding to the map
\begin{equation*}
\begin{array}{rcl}
  R/\Z&\to& \wedge^2(R/\Z)\\[.07in]
  \theta&\mapsto& \theta\wedge\theta^2.
\end{array}
\end{equation*}
\end{proposition}

Throughout this paper, for an integral binary cubic form $f\in V(\Z)$,
we denote the cubic ring corresponding to $f$ by $R_f$, the cubic
algebra $R_f\otimes\Q$ by $K_f$, and the ring of integers of $K_f$ by
$\O_{K_f}$.
\index{$R_f$, cubic ring corresponding to a form $f\in V(\Z)$}
\index{$K_f = R_f\otimes \Q$, cubic field corresponding to the form $f\in V(\Z)^{\irr}$}
We have
\index{$\Delta(f)$, discriminant of the binary cubic form $f$}
\index{$\Delta(R)$, discriminant of the cubic ring $R$}
\index{$\Delta(K)$, discriminant of the cubic field $K$}
\[
\Delta(R_f) = \Delta(f) = b^2c^2-4ac^3-4b^3d-27a^2d^2+18abcd,
\]
for $f(x,y)=ax^3+bx^2y+cxy^2+dy^3$, and where we denote
by the same letter $\Delta$ the discriminants of rings and
algebras. Since $\Delta(K_f)=\Delta(\O_{K_f})$ by definition,
we have the equality
\begin{equation}\label{disc_form-field}
\Delta(f) = \Delta(K_f) [\mathcal{O}_{K_f}:R_f]^2=\Delta(K_f)\ind(f)^2,
\end{equation}
where we define the {\it index} of $f$, or $\ind(f)$, to be
$[\O_{K_f}:R_f]$. 
\index{$\ind(f)$, index of $R_f$ in $\O_{K_f}$}

In particular, we see that $|\Delta(K_f)|\le |\Delta(f)|$,
and that the signs of $\Delta(f)$ and $\Delta(K_f)$ coincide.  
If $\Delta(f)\neq 0$, then the algebra $K_f$ is \'etale.
If $f\in V(\Z)^{\rm irr}$ is irreducible, then $K_f$ is a field.
Furthermore,
$\Delta(f)>0$ when $K_f$ is totally real, and $\Delta(f)<0$ when $K_f$
is complex.
\index{$V(\Z)^{\irr}$, subset of irreducible binary cubic forms}

\medskip

We say that a ring $R$ has \emph{rank $n$} if it is free of rank $n$
as a $\Z$-module.  We say that a rank $n$ ring $R$ is {\it maximal} if
it is not a proper subring of any other ring of rank $n$. For a prime
$p$, we say that a rank $n$ ring $R$ is {\it maximal at $p$} if
$R\otimes\Z_p$ is maximal in the sense that it is not a proper subring
of any other ring that is free of rank $n$ as a $\Z_p$-module.
We have that $R$ is maximal if and only if it
is maximal at $p$ for every prime $p$. 
\index{$V(\Z)^\max$, subset of maximal binary cubic forms}

We say that an integral binary cubic form $f$ is {\it maximal}
(resp.\ {\it maximal at $p$}) if the corresponding cubic ring $R_f$ is
maximal (resp.\ maximal at $p$). We have the following result \cite[\S
  3]{BST} characterizing binary cubic forms that are maximal at $p$.

\begin{proposition}\label{p:nonmaximal}
An integral binary cubic form $f\in V(\Z)$ is maximal at a prime $p$
if and only if both of the following two properties hold:
\begin{itemize}
\item[(i)] $f$ is not a multiple of $p$, and
\item[(ii)] $f$ is not $\GL_2(\Z)$-equivalent to a form
  $ax^3+bx^2y+cxy^2+dy^3$, with $p^2\mid a$ and $p\mid b$.
\end{itemize}
\end{proposition}

We will also need the following result, proved in \cite[Props.15-16]{BST}, that determines the number of index-$p$ subrings
and index-$p$ overrings of a cubic ring.

\begin{proposition}\label{subsupring}
  For an integral binary cubic form $f\in V(\Z)$, the number of cubic
  rings in $K_f$ containing $R_f$ with index $p$ is equal to the
  number of double zeros $\alpha\in \P^1(\F_p)$ of $f$ modulo $p$ such
  that $p^2 | f(\alpha')$ for all $\alpha'\in \P^1(\Z)$ with
  $\alpha'\equiv \alpha$ ${\rm mod}$~$p$.

  For an integral binary cubic form $g\in V(\Z)$, there is a bijection between
  index-$p$ subrings of $R_g$ and zeros in
  $\P^1(\F_p)$ of $g$ modulo $p$, whose number we denote by $\omega_p(g)$.
\end{proposition}
\index{$\omega_p(g)$, number of zeros in $\P^1(\F_p)$ of $g$ modulo $p$}

\begin{example}
Consider a form $f(x,y)=ax^3+bx^2y+cxy^2+dy^3 \in V(\Z)$, with
$p^2\mid a$ and $p\mid b$ which is nonmaximal by
Proposition~\ref{p:nonmaximal}.(ii).  Then $\alpha = [1:0]\in
\P^1(\F_p)$ is a double root of $f$ modulo $p$. The form
$\bigl(\begin{smallmatrix}\frac1{p}&{}\\{}&1\end{smallmatrix}\bigr)\cdot
  f(x,y)=(a/p^2)x^3+(b/p)x^2y+cxy^2+pdy^3$ corresponds to an index-$p$
  overring of $R_f$.
This is
consistent with Proposition~\ref{subsupring} which implies that the
number of cubic rings in $K_f$ containing $R_f$ with index $p$ is at
least one.
\end{example}

\subsection{Binary cubic forms over $\F_p$ and $\Z/n\Z$}\label{s_binary_Fp}
\index{$V^*$, dual of $V$ with compatible action by $\GL_2$}
Let $V^*=\Sym_3(2)$ denote the \emph{dual} of $V$, and denote by $[,]$
the duality pairing.  The $\GL_2$-action on $V^*$ is defined by the
rule that $[,]$ is relatively invariant:
\begin{equation}\label{relatively}
 [\gamma \cdot f, \gamma \cdot f_* ] = \det(\gamma) [f,f_*],\quad \forall \gamma \in \GL_2,\ f\in V,\ f_*\in V^*.
\end{equation}
The scalar matrices in $Z(\GL_2)$ act by scalar multiplication on 
both $V$ and $V^*$.

Let $a_*:=[y^3,f_*]$, $b_*:=-[xy^2,f_*]$, $c_*:=[x^2y,f_*]$, $d_*:=-[x^3,f_*]$, and
\[
\Delta_*(f_*):= 3 b_*^2c_*^2 + 6 a_*b_*c_*d_* - 4 a_*c_*^3 - 4 b_*^3d_* - a_*^2 
d_*^2.
\]
Both $\Delta$ and $\Delta_*$ are homogeneous of degree 
  $4$ and satisfy  $\Delta(\gamma \cdot f)=(\det \gamma)^2 \Delta(f)$ and $\Delta_*(\gamma \cdot f_*)=(\det 
  \gamma)^2 
  \Delta_*(f_*)$.

Following \cite[\S3]{Shintani} and \cite[Table 1]{HCL4},
the 
lattice $V^*(\Z)$ is isomorphic to the sub-lattice
\begin{equation}\label{eq:formsdualisom}
V^*(\Z) \simeq \{a_*x^3+3b_*x^2y+3c_*xy^2+d_*y^3:\;a_*,b_*,c_*,d_*\in\Z\} \subset V(\Z),
\end{equation}
with compatible $\GL_2(\Z)$-action. The restriction of $\Delta$ to $V^*(\Z)$ 
coincides with $27 \Delta_*$ as a direct calculation shows.
We also see that the pairing $[,]:V(\Z)\times V^*(\Z)\to \Z$ coincides with the 
restriction of the antisymmetric bilinear form
\begin{equation*}
\begin{array}{rcl}
V(\Z)\times V(\Z)&\to& \frac 13\Z \\[.1in]
(f_1,f_2) &\mapsto &d_1a_2 - \frac{c_1b_2}{3}+ \frac{b_1c_2}{3} - a_1d_2.
\end{array}
\end{equation*}

For an integer $n\ge 1$, the $\Z/n\Z$ points of $V$, which we denote
by $V(\Z/n\Z)$, form a finite abelian group which can be identified
with the quotient $V(\Z)/nV(\Z)$. The same holds for $V^*(\Z/n\Z)
\simeq V^*(\Z)/nV^*(\Z)$.  We obtain a perfect pairing $[,]:V(\Z/n\Z)
\times V^*(\Z/n\Z)\to \Z/n\Z$.

The finite abelian group $V^*(\Z/n\Z)$ is in natural
bijection with the \emph{group of characters} $V(\Z/n\Z)\to S^1$,
where $S^1$ denotes the unit circle in $\C^\times$.  
Indeed, given $f_*\in
V^*(\Z/n\Z)$, we associate the character
\begin{equation*}
  \begin{array}{rcc}
    \chi_{f_*}:V(\Z/n\Z)&\to& S^1\\[.1in]
    f&\mapsto& e\Bigl(\frac{[f,f_*]}{n}\Bigr),
  \end{array}
\end{equation*}
where $e(\alpha):= e^{2\pi i\alpha}$.

\index{$\widehat \phi:V(\Z/n\Z)\to \C$, Fourier transform of function
  $\phi$ on $V(\Z/n\Z)$} Given a function $\phi:V(\Z/n\Z)\to \C$, we
have the notion of its \emph{Fourier transform} $\widehat{\phi}$ given
by
\begin{equation*}
  \begin{array}{rcl}
    \widehat{\phi}: V^*(\Z/n\Z)&\to& \C\\[.15in]
    \widehat{\phi}(f_*)&:=&
    \displaystyle\frac{1}{n^4}\sum_{f\in V(\Z/n\Z)}
    e\Bigl(\frac{[f,f_*]}{n}\Bigr) \phi(f).
\end{array}
\end{equation*}
In this paper, we will be concerned with the Fourier transforms of
$\GL_2(\Z/n\Z)$-invariant functions. Regarding this, we have the
following result which is probably known although we couldn't find the statement in the 
literature.
\begin{lemma}
The Fourier transform $\widehat{\phi}$ of a $\GL_2(\Z/n\Z)$-invariant
function $\phi$ is $\GL_2(\Z/n\Z)$-invariant.
\end{lemma}

\begin{proof}
Let $\gamma\in \GL_2(\Z/n\Z)$, $f_*\in V^*(\Z/n\Z)$ and the function $\phi$ be given. We have
\begin{equation}\label{SL2}
\begin{array}{rcl}
\displaystyle\widehat{\phi}(\gamma \cdot f_*)&=&
\displaystyle\frac{1}{n^4}\sum_{f\in V(\Z/n\Z)}
e\Bigl(\frac{[f,\gamma \cdot f_*]}{n}\Bigr) \phi(f)
\\[.2in]&=&
\displaystyle\frac{1}{n^4}\sum_{f\in V(\Z/n\Z)}
e\Bigl(\frac{\det(\gamma)[\gamma^{-1}\cdot f,f_*]}{n}\Bigr) \phi(f)
\\[.2in]&=&
\displaystyle\frac{1}{n^4}\sum_{f\in V(\Z/n\Z)}
e\Bigl(\frac{\det(\gamma)[f,f_*]}{n}\Bigr) \phi(f),
\end{array}
\end{equation}
where the first equality is by definition, the second equality follows from~\eqref{relatively}, and the 
third equality follows from the $\GL_2(\Z/n\Z)$-invariance of $\phi$ and the bijective change 
of variable $f$ by $\gamma \cdot f$.
To finish the proof of the lemma, we absorb the $\det(\gamma)$ factor into the sum over $f$ 
since $\phi(uf)=\phi(f)$ for every $u\in (\Z/n\Z)^\times$ and $f\in V(\Z/n\Z)$ because $Z(\GL_2)$ acts 
by 
scalar multiplication on $V$. 
\end{proof}

\subsection{Fourier transforms of $\GL_2$-orbits}
We now consider a prime number $p\neq 3$.  The \emph{orbits} for the action of
$\GL_2(\F_p)$ on $V(\F_p)$ and $V^*(\F_p)$ are characterized as
follows~\cite[\S5]{TT}. There are six $\GL_2(\F_p)$-orbits on $V(\F_p)$ depending on how a binary
cubic form factors over $\F_p$. Using \eqref{eq:formsdualisom}, we may identify $V^*(\F_p)=V^*(\Z)\otimes\F_p$ with $V(\F_p)$. There are thus also six $\GL_2(\F_p)$-orbits on $V^*(\F_p)$.
We denote the orbits on $V(\F_p)$ by
\index{$\O_\sigma$, orbits for the action of $\GL_2(\F_p)$ on
  $V(\F_p)$}
\index{$\O^*_\sigma$, orbits for the action of
  $\GL_2(\F_p)$ on $V^*(\F_p)$}
\begin{equation}\label{eqbcforbit}
\O_{(111)},\O_{(12)},\O_{(3)},\O_{(1^21)},\O_{(1^3)},\O_{(0)},
\end{equation}
and the orbits on $V^*(\F_p)$ by
\begin{equation}\label{eqbcforbitd}
\O_{(111)}^*,\O^*_{(12)},\O_{(3)}^*,\O_{(1^21)}^*,\O_{(1^3)}^*,\O_{(0)}^*,
\end{equation}
respectively, where $\O_{(111)},\O_{(111)}^*$ denote
the sets of forms having three distinct rational roots in
$\P^1(\F_p)$, the sets $\O_{(12)},\O_{(12)}^*$ consist of forms having
one root in $\P^1(\F_p)$ and one pair of conjugate roots defined over
the quadratic extension of $\F_p$, the sets $\O_{(3)},\O_{(3)}^*$
consist of forms irreducible over $\F_p$, the sets
$\O_{(1^21)},\O_{(1^21)}^*$ (resp.\ $\O_{(1^3)},\O_{(1^3)}^*$) consist
of forms having a root in $\P^1(\F_p)$ of multiplicity $2$
(resp.\ $3$), and $\O_{(0)},\O_{(0)}^*$ is the singleton set
containing the zero form.  Given a subset $S$ of $V(\F_p)$ or
$V^*(\F_p)$, let $C_S$ denote its characteristic function.  Every
$\GL_2(\F_p)$-invariant function on $V(\F_p)$ (resp.\ $V^*(\F_p)$) is
a linear combination of the six functions
\[
C_{\O_{(0)}},\;C_{\O_{(1^3)}},\;C_{\O_{(1^21)}},
C_{\O_{(111)}},\;C_{\O_{(12)}},\;C_{\O_{(3)}},\mbox{ (resp. }
C_{\O_{(0)}^*},\;C_{\O_{(1^3)}^*},\;C_{\O_{(1^21)}^*},
C_{\O_{(111)}^*},\;C_{\O_{(12)}^*},\;C_{\O_{(3)}^*}).
\]
Therefore, the Fourier transforms of the first six of the above
functions determine the Fourier transforms of every
$\GL_2(\F_p)$-invariant function on $V(\F_p)$.  

\index{$M$, matrix of the Fourier transform of $\GL_2(\F_p)$-orbits on $V(\F_p)$}
\begin{proposition}[Mori~\cite{Mori}] \label{thFT}
  Let $p\neq 3$ be a prime number, and $M=(m_{ij})$ be the following $6\times 6$
  matrix 
\begin{equation*}
M:=\frac{1}{p^4}
\left[
\begin{array}{cccccc}
1 &(p+1)(p-1) & p(p+1)(p-1)  &p(p+1)(p-1)^2/6 & p(p+1)(p-1)^2/2 & p(p+1)(p-1)^2/3 \\[.05in]
1 &-1    & p(p-1)  &p(p-1)(2p-1)/6   & -p(p-1)/2       & -p(p+1)(p-1)/3      \\[.05in]
1 &p-1   &p(p-2) & -p(p-1)/2       & -p(p-1)/2       &0                \\[.05in]
1 &2p-1  &-3p    &p(\pm p + 5)/6      &-p(\pm p - 1)/2   &p(\pm p -1 )/3     \\[.05in]
1 &-1    &-p     &-p(\pm p - 1)/6    &p(\pm p + 1)/2     &-p(\pm p -1 )/3    \\[.05in]
1 &-p-1  &0      &p(\pm p - 1)/6     &-p(\pm p - 1)/2     &p(\pm p + 2)/3
\end{array}
\right],
\end{equation*}
where the signs $\pm$ appearing in the bottom-right $3\times 3$ corner are
according as $p \equiv \pm1 \pmod{3}$.
Then  
\[\widehat{C_j}=\sum_{i=1}^6 m_{ij}C_{i}^*,\quad 1\le j \le 6,
\]
where we have set
\begin{equation*}
\begin{array}{rcl}
(C_1,C_2,C_3,C_4,C_5,C_6)&:=&
  (C_{\O_{(0)}},C_{\O_{(1^3)}},C_{\O_{(1^21)}},C_{\O_{(111)}},C_{\O_{(12)}},C_{\O_{(3)}});
\\[.07in]
  (C_1^*,C_2^*,C_3^*,C_4^*,C_5^*,C_6^*)&:=&
  (C_{\O_{(0)}^*},C_{\O_{(1^3)}^*},C_{\O_{(1^21)}^*},C_{\O_{(111)}^*},
C_{\O_{(12)}^*},C_{\O_{(3)}^*}).
\end{array}
\end{equation*}
\end{proposition}

\begin{proof}
The result was announced in \cite{Mori}, and a proof appears in the work of 
Taniguchi--Thorne~\cite[Thm.11]{TaTh2} and~\cite[Rem.6.8]{TT}.
\end{proof}

\begin{remarks*}
(i) For $j=1$, that is for the first column of $M$, Proposition~\ref{thFT} 
says that the 
Fourier 
transform of $C_{\O_{(0)}}$, which is the 
Dirac function of the origin, is equal to the constant function $1/p^4$ as should be.

(ii) For $i=1$, the first row of $M$ in Proposition~\ref{thFT} provides the respective sizes of 
each of the $6$ conjugacy 
classes, because
\[
\sum_{f\in V(\F_p)} C_j(f) = p^4\widehat {C_j}(0) = p^4 m_{1j}.
\]
They add up to $m_{11}+ m_{12} + \cdots + 
m_{16} = 1$ as should be.

(iii) For every $j,k$, we have $\sum_{f\in V(\F_p)} C_j(f)C_k(f) = p^4 \delta_{jk} m_{1j}$, because 
the characteristic functions are pairwise orthogonal since the orbits are pairwise 
disjoint. This 
implies, by the Plancherel formula,
$\sum_{f_*\in V^*(\F_p)} \widehat{C_j}(f_*)\widehat{C_k}(f_*) = \delta_{jk} m_{1j}$.
Hence, Proposition~\ref{thFT} implies
\begin{equation}\label{verif-orth}
p^4 \sum^6_{i=1} m_{ij} m_{ik}m_{1i} = \delta_{jk} m_{1j},
\quad 1\le j,k\le 6,
\end{equation}
which indeed holds true as a direct verification shows. Because of the symmetry between 
$j,k$, verifying~\eqref{verif-orth} entails to verifying $21$ equalities. 
\end{remarks*}

Proposition \ref{thFT} has the following important consequence.

\begin{corollary}\label{corphihatbound}
Let $p\neq 3$ be a prime number, and let $\phi:V(\F_p)\to\C$ be a
$\GL_2(\F_p)$-invariant function such that $|\phi(f)|\le 1$ for every $f\in V(\F_p)$. 
Then we have
\begin{equation*}
  \widehat{\phi}(f_*)\ll\left\{
    \begin{array}{ccl}
      \displaystyle p^{-2}&{\rm if}& f_*\in \O^*_{(111)},\O^*_{(12)},
      \O^*_{(3)},\O^*_{(1^21)};
      \\[.15in]
      \displaystyle  p^{-1} &{\rm if}& f_*\in \O^*_{(1^3)};
      \\[.15in]
      \displaystyle 1 &{\rm if}& f_*\in\O^*_{(0)}.
  \end{array}
  \right.
\end{equation*}
The absolute constant in $\ll$ can be taken to be $4$.
\end{corollary}
\begin{proof}
The rows of $M$ are bounded by $m_{1\bullet} = O(1)$, $m_{2\bullet} = O(p^{-1})$ and $m_{i\bullet} 
= 
O(p^{-2})$ for $3\le i \le 6$, or equivalently
$M=\bigl[O(1),O(p^{-1}),O(p^{-2}),O(p^{-2}),O(p^{-2}),O(p^{-2})\bigr]^T$.
For example, we can make the absolute constant explicit as follows:
$\sum^6_{j=1} m_{1j} = 1$, $\sum^6_{j=1} |m_{2j}| \le 1/p$, 
$\sum^6_{j=1} |m_{3j}| \le 2/p^2$,
$\sum^6_{j=1} |m_{4j}|\le 4/p^2$, $\sum^6_{j=1} |m_{5j}|\le 2/p^2$, $\sum^6_{j=1} |m_{6j}|\le 
2/p^2$. 

By assumption, $\phi = \sum\limits^6_{j=1} a_j C_j$ with $|a_j|\le 1$.
Proposition~\ref{thFT} implies that 
\[
|\widehat \phi(f_*)| \le \sum^6_{i=1} C^*_i(f_*) \sum^6_{j=1} 
|m_{ij}|.
\]
We deduce
\[
 |\widehat \phi(f_*)| \ll C_1^*(f_*) + p^{-1} C_2^*(f_*)
+ p^{-2} \left(C_3^*(f_*) + C_4^*(f_*) + C_5^*(f_*) + C_6^*(f_*) \right),
\]
from which the corollary follows.
\end{proof}

\section{The Artin character of cubic
  fields and rings}\label{sec:Artin}

Let $K$ be a cubic field extension of $\Q$, with normal closure $M$.
The Dedekind zeta function $\zeta_K(s)$
of $K$ factors as
\begin{equation*}
\zeta_K(s)=\zeta_\Q(s)L(s,\rho_K),
\end{equation*}
where $\zeta_\Q(s)$ denotes the Riemann zeta function and $L(s,\rho_K)$
is an Artin $L$-function associated to the two-dimensional
representation $\rho_K$ of $\Gal(M/\Q)$,
\[
 \rho_K: \Gal(M/\Q) \hookrightarrow S_3 \to \GL_2(\C).
\]

In this section, we first begin by collecting some well-known
properties of $L(s,\rho_K)$. We denote the Dirichlet coefficients of
$L(s,\rho_K)$ by $\lambda_K(n)$. Then we extend the definition of
$\lambda_K(n)$ to the set of all cubic rings $R$. We do this by
defining $\lambda_n(f)$ for all binary cubic forms $f$. Finally, for
primes $p\neq 3$, we compute the Fourier transform of the function
$\lambda_p$.
\index{$\rho_K$, two-dimensional Galois representation}

\subsection{Standard properties of $L(s,\rho_K)$}\label{sHecke}

\index{$\lambda_K(n)$, $n$th Dirichlet coefficient of $L(s,\rho_K)$}
We denote the Euler factors of $L(s,\rho_K)$ at primes $p$ by
$L_p(s,\rho_K)$, and the $n$th Dirichlet coefficient of $L(s,\rho_K)$
by $\lambda_K(n)$. We have that $\lambda_K$ is multiplicative. We
write the $p^k$th Dirichlet coefficient of the logarithmic derivative
of $L(s,\rho_K)$ as $\theta_K(p^k)\log p$.  That is, we have for
$\Re(s)>1$,
\begin{equation}\label{def:Dirichlet}
\begin{array}{rcccl}
\displaystyle
L(s,\rho_K)&=&
\displaystyle\prod_{p\;{\rm prime}} L_p(s,\rho_K)&=&
\displaystyle\sum^\infty_{n=
  1}\frac{\lambda_K(n)}{n^s},\\[.2in]
\displaystyle-\frac{L'(s,\rho_K)}{L(s,\rho_K)}&=&
\displaystyle-\sum_{p\;{\rm prime}}\frac{L_p'(s,\rho_K)}{L_p(s,\rho_K)}&=&
  \displaystyle\sum_{n=1}^{\infty}\frac{\theta_K(n)\Lambda(n)}{n^s}.
\end{array}
\end{equation}
Note that
$\theta_K$ is supported on prime powers.
\index{$\theta_K(n)$, coefficient of the logarithmic derivative of $L(s,\rho_K)$}

\medskip

Next, we recall some classical facts about $L(s,\rho_K)$.
Let $\Gamma_\R(s):=\pi^{-s/2} \Gamma(\frac s2)$ and
$\Gamma_\C(s):=2(2\pi)^{-s}\Gamma(s)$. Hecke proved that the
\emph{completed Dedekind zeta function}
\[
 \xi_K(s) := |\Delta(K)|^{s/2} \zeta_K(s) \cdot 
\begin{cases}
\Gamma_\R(s)^3, &  \text{if $\Delta(K)>0$},\\
\Gamma_\R(s)\Gamma_\C(s), &  \text{if $\Delta(K)<0$},
\end{cases}
\]
has a meromorphic continuation to $s\in \C$ with simple poles at $s=0,1$ and satisfies the 
functional equation $\xi_K(s)=\xi_K(1-s)$.
We introduce the following notation:
\begin{equation*}
\begin{array}{rcl}
\gamma^+(s)&:=&\Gamma_\R(s)^2 = \pi^{-s}\Gamma(\frac s2)^2;\\[.1in]
\gamma^-(s)&:=&\Gamma_\C(s) = 2(2\pi)^{-s}\Gamma(s).
\end{array}
\end{equation*}
\index{$\gamma^{\pm}(s)$, Gamma factor in the functional equation of $L(s,\rho_K)$}
\begin{proposition}[Hecke]\label{p:Hecke}
$L(s,\rho_K)$ is entire and satisfies the functional equation
$\Lambda(s,\rho_K) = \Lambda(1-s,\rho_K)$,
where $\Lambda(s,\rho_K) := |\Delta(K)|^{s/2} L_\infty(s,\rho_K) L(s,\rho_K)$ is the completed 
$L$-function, and
\[
 L_\infty(s,\rho_K) := \gamma^{\sgn(\Delta(K))}(s) = 
\begin{cases}
\Gamma_\R(s)^2, &  \text{if $\Delta(K)>0$},\\
\Gamma_\C(s), &  \text{if $\Delta(K)<0$}.
\end{cases}
\]
\end{proposition}

\begin{proof}
The functional equation of $L(s,\rho_K)$ follows from the 
functional equations of $\zeta_K(s)$ and $\zeta_\Q(s)$. It remains to show that $L(s,\rho_K)$ is 
entire and there 
are two cases to distinguish:
If $K$ is non-Galois, then $M/\Q$ is Galois with Galois group
isomorphic to $S_3$, whereas if $K$ is Galois, then $M=K$ with Galois
group isomorphic to $\Z/3\Z$.

(i) If $K=M$ is
Galois, then the Artin representation
\begin{equation*}
\rho_K:{\rm Gal}(M/\Q)\cong \Z/3\Z \hookrightarrow S_3 \to\GL_2(\C)
\end{equation*}
is the direct sum of the two nontrivial characters of $\Z/3\Z$. Hence
$L(s,\rho_K)= L(s,\chi_K)L(s,\overline {\chi_K})$ for two conjugate Dirichlet
characters $\chi_K$ and $\overline{\chi_K}$ of order $3$ and conductor $|\Delta(K)|^{\frac12}$.
Dirichlet proved that $L(s,\chi_K)$ and $L(s,\overline{\chi_K})$ are entire.

(ii) If $K$ is non-Galois, then the Artin representation
\begin{equation*}
\rho_K:{\rm Gal}(M/\Q)\cong S_3\to\GL_2(\C)
\end{equation*}
obtained from the standard representation of $S_3$ is irreducible. In this case, the sextic 
field $M$ has a unique quadratic subfield
denoted $L$.  We have an exact sequence
\[
\Gal(M/L) \hookrightarrow \Gal(M/\Q) \twoheadrightarrow \Gal(L/\Q),
\]
and the representation $\rho_K$ of $\Gal(M/\Q)\simeq S_3$ is induced
from a character $\chi_K$ of $\Gal(M/L)\simeq A_3=\Z/3\Z$:
\[
 \rho_K \simeq \operatorname{Ind}^{\Gal(M/\Q)}_{\Gal(M/L)} (\chi_K).
\]
 Thus we
have $L(s,\rho_K) = L(s,\chi_K)$.  Via class field theory, $\chi_K$ corresponds to a ring-class 
character of $L$ of order $3$. We have that $L(s,\chi_K)$ is entire by work of Hecke on the 
$L$-functions attached to Gr\"ossencharacters.
\end{proof}

The following standard result isn't directly used in the rest of the paper, except that the second case of the proposition when  \(K\) is an  \(S_3\)-field is relevant to Theorem~\ref{p_Wu} below. The reader can safely skip it.
\begin{proposition}[Hecke, Maass]\label{HeckeMaass}
The representation $\rho_K$ is modular. That is, there exists a unique
automorphic representation $\pi_K$ of $\GL_2$ such that
$L(s,\rho_K)$ is equal to the principal $L$-function $L(s,\pi_K)$.
\begin{itemize}
\item If $K/\Q$ is cyclic, then $\pi_K$ is an Eisenstein
  series with trivial central character.
\item If $K$ is an $S_3$-field, then $\pi_K$ is cuspidal and its central
  character is the quadratic Dirichlet character associated to the
  quadratic resolvant of $K$. Moreover,
\begin{itemize}
\item if $\Delta(K)<0$ then $\pi_{K,\infty}$ is
  holomorphic of weight $1$,
\item if $\Delta(K)>0$ then $\pi_{K,\infty}$ is
  spherical of weight $0$.
\end{itemize}
\end{itemize} 
\end{proposition}
\begin{proof}[Sketch of proof]
The construction of  \(\pi_K\) is due to Hecke and Maass and comes from the theory of theta series.
The unicity of $\pi_K$ follows from the strong multiplicity-one
theorem for $\GL_2$.  The central character of $\pi_K$ corresponds
under class field theory to the determinant character 
\[
 \det \rho_K: \Gal(M/\Q) \twoheadrightarrow \Gal(M/\Q)^{\text ab} \to \C^\times.
\]
If  \(K\) is Galois, then the permutations in \(\Z/3\Z\) have trivial determinant. If $K$ is non-Galois with quadratic resolvant $L$, then the transposition permutations in  \(S_3\) have non-trivial determinant, and since
$\Gal(M/\Q)^{\text ab} = \Gal(L/\Q)\cong \Z/2\Z$ we obtain that  \(\det \rho_K\) is the quadratic Dirichlet character associated with  \(L/\Q\).
\end{proof}

\subsection{Definition and properties of $\lambda_n(f)$}
\label{sec:lambdan}

Let $K$ be a cubic field with ring of integers $\O_K$. We say that $K$
has \emph{splitting type} $\sigma_p(K)$ to be $(111)$, $(12)$, $(3)$,
$(1^21)$ or $(1^3)$ at $p$ if $p$ factors as $\fp_1\fp_2\fp_3$,
$\fp_1\fp_2$, $p$, $\fp_1^2\fp_2$, or $\fp^3$, respectively.  Recall
that $L(s,\rho_K)$ has an Euler factor decomposition, where it may be
checked that the $p$th Euler factor $L_p(s,\rho_K)$ only depends on
the splitting type of $K$ at $p$, and is as follows:
\begin{equation}\label{eqEP}
  L_p(s,\rho_K)=\left\{
  \begin{array}{lclll}
    (1-p^{-s})^{-2}&=&\displaystyle\sum_{m=0}^\infty (m+1)p^{-ms}&\;{\rm 
    if}\;&\sigma_p(K)=(111);\\[.15in]
    (1-p^{-2s})^{-1}&=&\displaystyle\sum_{m=0}^\infty p^{-2ms}&\;{\rm 
    if}\;&\sigma_p(K)=(12); \\[.15in]
    (1+p^{-s}+p^{-2s})^{-1}&=&\displaystyle\sum_{m=0}^\infty 
    (p^{-3ms}-p^{-(3m+1)s})&\;{\rm if}\;&\sigma_p(K)=(3);\\[.15in]
    (1-p^{-s})^{-1}&=&\displaystyle\sum_{m=0}^\infty p^{-ms}&\;{\rm 
    if}\;&\sigma_p(K)=(1^21);\\[.15in]
    1&&&\;{\rm if}\;&\sigma_p(K)=(1^3).
  \end{array}\right.
\end{equation}

For a prime $p$, recall the six $\GL_2(\F_p)$-orbits
$\cO_\sigma$ on $V(\F_p)$ defined in \eqref{eqbcforbit}. 
\begin{definition}\label{def:lambda}
Given an element $f\in V(\F_p)$, we define the {\it splitting type}
$\sigma_p(f)$ of $f$ to be $\sigma$ if $f\in \cO_\sigma$. For $m\geq
1$, we define the function $\lambda_{p^m}:V(\F_p)\to \Z$ as follows:

Let $f\in V(\F_p)$ have splitting type $\sigma$. Let $K$ be any field
also having splitting type $\sigma$ at $p$. Then we define
$\lambda_{p^m}(f):=\lambda_K(p^m)$.  This serves as a definition for
all nonzero $f$. For the zero form, we simply define $\lambda_{p^m}(0):=0$.
\end{definition}
\index{$\sigma_p(f)$, splitting type of $f$ at $p$}
\index{$\lambda_n(f)$, Artin character on the space of cubic forms}

Explicitly, we compute
\begin{equation}\label{def_lambda}
  \lambda_{p^m}(f):=\left\{
  \begin{array}{rcl}
    (m+1)&{\rm if}& \sigma_p(f)=(111);\\[.05in]
    1&{\rm if}& \sigma_p(f)=(12)\;\; {\rm and}\;\; m\equiv{0}\pmod{2};\\[.05in]
    0&{\rm if}& \sigma_p(f)=(12)\;\; {\rm and}\;\; m\equiv{1}\pmod{2};\\[.05in]
    1&{\rm if}& \sigma_p(f)=(3)\;\; {\rm and}\;\; m\equiv{0}\pmod{3};\\[.05in]
    -1&{\rm if}& \sigma_p(f)=(3)\;\; {\rm and}\;\; m\equiv{1}\pmod{3};\\[.05in]
    0&{\rm if}& \sigma_p(f)=(3)\;\; {\rm and}\;\; m\equiv{2}\pmod{3};\\[.05in]
    1&{\rm if}& \sigma_p(f)=(1^21);\\[.05in]
    0&{\rm if}& \sigma_p(f)=(1^3);\\[.05in]
    0&{\rm if}& \sigma_p(f)=(0).
  \end{array}
  \right.
\end{equation}

Extending notation, we set $\lambda_{p^m}:V(\Z)\to\Z$ by defining
$\lambda_{p^m}(f):=\lambda_{p^m}(f \pmod{p} )$, where on the right-hand
side we have the reduction of $f$ modulo $p$.  We also write
$\sigma_p(f) = \sigma_p(f\pmod{p})$ for the splitting type of $f$ at
$p$. For a positive integer $n\ge 1$, we define
$\lambda_n:V(\Z)\to\Z$ multiplicatively in $n$, i.e., we set
\begin{equation*}
\lambda_n(f):=\prod_{p^m\parallel n}\lambda_{p^m}(f).
\end{equation*}
The function $\lambda_n(f)$ is $\GL_2(\Z)$-invariant and only depends
on the reduction of $f$ modulo $\operatorname{rad}(n)$, where
$\operatorname{rad}(n)$ is the radical of $n$, that is the largest
square-free divisor of $n$.

\medskip

Next, given a binary cubic form $f\in V(\Z)$, we define the following Dirichlet
series:
\index{$D(s,f)$, Dirichlet series of $\lambda_n(f)$}
\begin{equation}\label{eqLsf}
  D(s,f):=\sum_{n\geq 1}\frac{\lambda_n(f)}{n^s}=\prod_pD_p(s,f),
\end{equation}
where the function $D_p(s,f)$ depends only on the splitting type of
$f$ at $p$. In fact, if a cubic field $K$ has the same splitting type
as $f$ at $p$, then $D_p(s,f)=L_p(s,\rho_K)$, where $L_p(s,\rho_K)$ is
given explicitly in \eqref{eqEP}.
When $f$ is a multiple of $p$, we have $D_p(s,f)=1$. 

For an irreducible integral binary cubic form $f$, with associated
number field $K_f$ as in Proposition~\ref{df}, the relationship between
$D(s,f)$ and $L(s,\rho_{K_f})$ is given by the following.

\begin{lemma}\label{l_Euler}
Let $f\in V(\Z)^{\rm irr}$ be irreducible. Assume that $f$ is maximal
at the prime $p$. Then $\sigma_p(f)=\sigma_p(K_f)$, and therefore
\begin{equation}\label{eqEP1}
 D_p(s,f)=  L_p(s,\rho_{K_f}).
\end{equation}
\end{lemma}

\begin{proof}
Since $f$ is maximal at $p$, we have $R_f\otimes\Z_p\cong
\O_{K_f}\otimes\Z_p$, where $R_f$ denotes the cubic ring corresponding
to $f$ and $\O_{K_f}$ denotes the ring of integers of $K_f$. Further
tensoring with $\F_p$, we obtain $R_f\otimes\F_p\cong
\O_{K_f}\otimes\F_p$. The former determines $\sigma_p(f)$ while the
latter determines the splitting of $K_f$ at $p$. Thus, the claim
follows.
\end{proof}

\begin{corollary}\label{c_Euler}
If $f\in V(\Z)^{\rm irr,max}$ is irreducible and maximal, that is if
$R_f$ is the ring of integers of the number field $K_f$, then $
L(s,\rho_{K_f})=D(s,f)$, and $\lambda_{K_f}(n)=\lambda_n(f)$ for all
$n\geq 1$.
\end{corollary}

\begin{proof}
This is immediate from Definition \ref{def:lambda} and the previous
Lemma \ref{l_Euler}.
\end{proof}

\begin{corollary}\label{lem_Dpsf}
Let $f\in V(\Z)^{\rm irr}$ be irreducible. Then the function $D(s,f)$ converges absolutely for $\Re(s)>1$.
\end{corollary}
\begin{proof}
This is immediate since $D(s,f)$ and $L(s,\rho_{K_f})$ can differ only
at the finitely many Euler factors at $p$, where $f$ is nonmaximal at $p$.
\end{proof}

For every $f\in V(\Z)$, and prime power $n=p^m$, define
$\theta_{p^m}(f)$ from the $p^m$th-coefficient of the logarithmic
derivative,
\begin{equation*}
\begin{array}{rcccl}
\displaystyle
\displaystyle-\frac{D'(s,f)}{D(s,f)}&=&
\displaystyle-\sum_{p} \frac{D_p'(s,f)}{D_p(s,f)}&=&
  \displaystyle\sum_{n=1}^{\infty}\frac{\theta_n(f)\Lambda(n)}{n^s},
\quad \Re(s)>1.
\end{array}
\end{equation*}
\index{$\theta_n(f)$, coefficients of the logarithmic derivative of $D(s,f)$}

\begin{lemma}\label{l_thetap2}
For every prime $p$ and $f\in V(\Z)$, we have
$\theta_p(f)=\lambda_p(f)$ and
$\theta_{p^2}(f)=2\lambda_{p^2}(f)-\lambda_p(f)^2$. Furthermore, we
have the bound $|\theta_{p^m}(f)|\leq 2$ for every prime $p$, integer
$m\ge 1$ and $f\in V(\Z)$.
\end{lemma}
\begin{proof}
The first two claims follow from  \(D_p(s,f) = 1 + \lambda_p(f) p^{-s} + \lambda_{p^2}(f)p^{-2s} + O(p^{-3s})\) and expanding its logarithmic derivative. The third claim is the case  \(n=3\) of \cite[Lem.2.2]{SST1}, of which we now repeat the argument for completeness.
We have  \(D_p(s,f) = (1-\alpha_1p^{-s})^{-1}(1-\alpha_2p^{-s})^{-1}\), where  \(|\alpha_1|,|\alpha_2|\le 1\) as can be seen by inspecting each case of~\eqref{eqEP}. 
Then  \(\theta_{p^m}(f) = \alpha_1^m + \alpha_2^m\), which implies the desired inequality  \(|\theta_{p^m}(f)|\le 2\).
\end{proof}

We conclude this section with certain Fourier transform
computations. First, we have the following result, which will be
useful in the sequel when we sum $\lambda_p$ and $\theta_{p^2}$ over
$\GL_2(\Z)$-orbits on integral binary cubic forms having bounded
discriminant.

\begin{proposition}\label{c_thetahat}
  Let $p\neq 3$ be a prime. Then
\begin{equation*}
  \widehat{\lambda_p}(f_*)\;\;=\;\;\left\{
    \begin{array}{ccl}
    \displaystyle\frac{-1}{p^3}&{\rm if}& f_* \in \mathcal O^*_{(111)},
\mathcal O^*_{(12)}, \mathcal O^*_{(3)}, \mathcal O^*_{(1^21)}
;\\[.15in]
    \displaystyle\frac{p^2-1}{p^3} &{\rm if}& f_* \in\mathcal O^*_{(1^3)},
\mathcal O^*_{(0)}.
  \end{array}
  \right.
\end{equation*}
Moreover, $
\widehat{\theta_{p^2}}(0)\;\;=\;\;1-\displaystyle\frac{1}{p^2}$.
\end{proposition}
\begin{proof}
A beautiful proof of a related result can be found in \cite[Prop.1]{TaTh2}. However, for the sake of completeness, we 
explain how we can recover this result (and indeed can compute the Fourier transform of any $\GL_2(\F_p)$-invariant function) from a simple application of Proposition~\ref{thFT}. When $f_* \in \mathcal O^*_{(111)}$, we compute
\begin{equation*}
\begin{array}{rcl}
\displaystyle\widehat{\lambda_p}(f_*)&=&\displaystyle\frac{1}{p^4}\Bigl(
\lambda_p(0)+\lambda_p(1^3)(2p-1)
+\lambda_p(1^21)(-3p)+\lambda_p(111)(p(5\pm p)/6)
\\[.2in]&&\displaystyle
+\lambda_p(12)(-p(-1\pm p)/2)+\lambda_p(3)(p(-1\pm p)/3)\Bigr)
\\[.2in]&=&\displaystyle
\frac{1}{p^4}\Bigl(0+0-3p+p(5\pm p)/3-0-p(-1\pm p)/3\Bigr)
\\[.2in]&=&\displaystyle
\frac{-1}{p^3},
\end{array}
\end{equation*}
as claimed. The computation when $f_*$ is in the other orbits is similar.

Finally, note that $\theta_{p^2}(f)$ is equal to $2$ when
$\sigma_p(f)\in\{(111),(12)\}$, equal to $-1$ when $\sigma_p(f)=(3)$,
equal to $1$ when $\sigma_p(f)=(1^21)$, and equal to $0$ otherwise.
Therefore, from the first row of the table in Proposition \ref{thFT},
we have
\begin{equation*}
\begin{array}{rcl}
\displaystyle\widehat{\theta_{p^2}}(0)&=&
\displaystyle\bigl(\frac{2}{6}+1-\frac{1}{3}\bigr)
\frac{p(p+1)(p-1)^2}{p^4}+\frac{p(p+1)(p-1)}{p^4}
\\[.2in]&=&\displaystyle
(p-1+1)\Bigl(\frac{(p+1)(p-1)}{p^3}\Bigr)
\\[.2in]&=&\displaystyle
1-\frac{1}{p^2},
\end{array}
\end{equation*}
as necessary.
\end{proof}

\begin{remark*}
Requiring that equality~\eqref{eqEP1} of Lemma~\ref{l_Euler}
holds is enough to force the value of $\lambda_{p^m}(f)$ for every
non-zero element $f\in V(\F_p) - \{0\}$ to be as
in~\eqref{def_lambda}.  We have then chosen $\lambda_{p^m}(0):=0$
specifically so that the identities of Proposition~\ref{c_thetahat}
hold.
\end{remark*}

Let $u_p:V(\Z/p^2\Z)\to \{0,1\}$ denote the
characteristic function of the set of elements that lift to binary
cubic forms in $V(\Z_p)$ that are maximal at $p$. We then have the
following result.
\begin{proposition}\label{propmaxden}
We have
\begin{equation*}
\begin{array}{rcl}
\displaystyle\widehat{u_p\cdot \lambda_p}(0)&=&
\displaystyle\frac{(p-1)(p^2-1)}{p^4};\\[.15in]
\displaystyle\widehat{u_p\cdot \lambda_{p^2}}(0)&=&
\displaystyle\frac{(p^2-1)^2}{p^4};\\[.15in]
\displaystyle\widehat{u_p\cdot \theta_{p^2}}(0)&=&
\displaystyle\frac{(p^2-1)^2}{p^4}.
\end{array}
\end{equation*}
\end{proposition}
\begin{proof}
The Fourier transform at $0$ can be evaluated by a density
computation. That it so say, for any function $\phi:V(\Z/p^2\Z)\to\R$, we have
\begin{equation*}\widehat{\phi}(0)=\frac{1}{p^8}\sum_{f\in V(\Z/p^2\Z)}\phi(f).
\end{equation*}

In \cite[Lem.18]{BST}, the densities of $u_p$ are listed for each
splitting type, as $\mu(\U_p(111))$, $\mu(\U_p(12))$, and so on, which
we will abbreviate simply as $\mu(111)$, $\mu(12)$, and so on. And so
we may calculate:
\begin{equation*}
\begin{array}{rcl}
\displaystyle\widehat{u_p\cdot \lambda_p}(0)&=&
\displaystyle\mu(111)\lambda_p(111)+\mu(12)\lambda_p(12)+\mu(3)\lambda_p(3)+\mu(1^21)\lambda_p(1^21)+\mu(1^3)\lambda_p(1^3)
\\[.2in]&=&
\displaystyle \frac{1}{p^4}\Bigl(\frac16 (p-1)^2p(p+1)\cdot 2+\mu(12) \cdot 0 +\frac13 (p-1)^2p(p+1) \cdot(-1)+(p-1)^2(p+1)\cdot 1\Bigr)
\\[.2in]&=&
\displaystyle \frac{(p-1)(p^2-1)}{p^4},
\end{array}
\end{equation*}
as necessary. Similarly, we have
\begin{equation*}
\begin{array}{rcl}
\displaystyle\widehat{u_p\cdot \lambda_{p^2}}(0)&=&
\displaystyle\mu(111)\lambda_{p^2}(111)+\mu(12)\lambda_{p^2}(12)+
\mu(3)\lambda_{p^2}(3)+\mu(1^21)\lambda_{p^2}(1^21)+\mu(1^3)\lambda_{p^2}(1^3)
\\[.2in]&=&
\displaystyle \frac{1}{p^4}\Bigl(\frac16 (p-1)^2p(p+1)\cdot 3+\frac12
(p-1)^2p(p+1)+ \mu(3)\cdot 0 +(p-1)^2(p+1)\cdot 1\Bigr)
\\[.2in]&=&
\displaystyle \frac{(p^2-1)(p-1)(p+1)}{p^4},
\end{array}
\end{equation*}
as necessary. Finally, we have
\begin{equation*}
\begin{array}{rcl}
\displaystyle\widehat{u_p\cdot \theta_{p^2}}(0)&=&
\displaystyle\mu(111)\theta_{p^2}(111)+\mu(12)\theta_{p^2}(12)+\mu(3)\theta_{p^2}(3)+\mu(1^21)\theta_{p^2}(1^21)+\mu(1^3)\theta_{p^2}(1^3)
\\[.2in]&=&
\displaystyle \frac{1}{p^4}\Bigl(\Bigl(\frac16+\frac12\Bigr) (p-1)^2p(p+1)\cdot 2 +\frac13 (p-1)^2p(p+1) \cdot(-1)+(p-1)^2(p+1)\cdot 1\Bigr)
\\[.2in]&=&
\displaystyle \frac{1}{p^4}\bigl((p-1)^2p(p+1)+(p-1)^2(p+1)\bigr)
\\[.2in]&=&
\displaystyle \frac{(p^2-1)^2}{p^4},
\end{array}
\end{equation*}
as necessary.
\end{proof}

\section{Estimates on partial sums of Dirichlet coefficients of cubic 
fields and rings}\label{sec:dedekind}

In this section, we compute smoothed partial sums of the coefficients
$\lambda_K(n)$ as well as of $\lambda_n(f)$. This
section is organized as follows. First we collect
some preliminary facts about Mellin inversion. Then, we recall the convexity bounds as 
well as current
records towards the Lindel\"of Hypothesis for principal $\GL(2)$
$L$-functions.
We use these estimates to obtain bounds on smooth sums of the
Dirichlet coefficients $\lambda_K(n)$ in terms of $|\Delta(K)|$, where
$K$ is a cubic field.  Finally in \S\ref{s_lambdaf}, we prove
analogous bounds on smooth sums of $\lambda_n(f)$ in terms of
$|\Delta(f)|$, where $f\in V(\Z)^{\rm irr}$ is an irreducible integral
binary cubic form.

\subsection{Upper bounds on smooth sums of $\lambda_K(n)$}
\label{s_lambdaK}

We begin with a discussion of Mellin inversion, which will be used
throughout this paper.
Let $\Phi:\R_{\geq 0}\to\C$ be a smooth function that is rapidly
decaying at infinity.  We recall the definition of the \emph{Mellin
transform}
\index{$\widetilde \Phi$, $\widetilde \Psi$, Mellin transforms of $\Phi$, $\Psi$}
$$
\widetilde{\Phi}(s):=\int_0^\infty x^s\Phi(x)\frac{dx}{x}.
$$ 
The integral converges absolutely for $\Re(s)>0$. Integrating by
parts yields the functional equation
$\widetilde{\Phi}(s)=-\widetilde{\Phi'}(s+1)/s$. Hence, it follows that
$\widetilde{\Phi}$ has a meromorphic continuation to $\C$,
with possible simple poles at non-positive integers. 
Furthermore, $\widetilde{\Phi}(s)$ has superpolynomial decay on
vertical strips.  \emph{Mellin inversion} states that we have, for every $x\in \R_{>0}$,
\begin{equation*}
\Phi(x)=\int_{\Re(s)=2}x^{-s}\widetilde{\Phi}(s)\frac{ds}{2\pi i}.
\end{equation*}
Consider a general Dirichlet series $ D(s)=\sum_{n=1}^\infty
\frac{a_n}{n^s} $ which converges absolutely for $\Re(s) > 1$. We can
then express the smoothed sums of the Dirichlet coefficients $a_n$ as
line integrals.
For every positive real number $X\in \R_{>0}$, we have 
\begin{equation*}
\sum_{n\geq 1} a_n\Phi\Bigl(\frac{n}{X}\Bigr)
=\int_{\Re(s)=2}
D(s)X^s\widetilde{\Phi}(s)\frac{ds}{2\pi i}.
\end{equation*}

\medskip

Consider the function $L(s,\rho_K)$ for a cubic field $K$. The
\emph{convexity bound} obtained from the Phragm\'en--Lindel\"of principle,
\[
L\bigl(\tfrac12+it,\rho_K\bigr) \ll_\epsilon
(1+|t|)^{\frac12+\epsilon}|\Delta(K)|^{\frac14+\epsilon},
\] 
will suffice for our purpose of establishing the main
Theorem~\ref{thmMoment}.  We shall also use the current best bound for
$L(\frac12+it,\rho_K)$ due to Blomer--Khan~\cite{Blomer-Khan} to
achieve an improved numerical quality of the exponents in
Theorem~\ref{thmnv2} and in the other results.
\begin{theorem} [Bound for $\GL(2)$ $L$-functions in the level aspect] 
\label{p_Wu}
For every $\epsilon >0$,  $t\in \R$ and cubic number field $K$,
\[
L\bigl(\tfrac12+it,\rho_K\bigr)
\ll_\epsilon (1+|t|)^{O(1)} |\Delta(K)|^{\theta  + \epsilon},
\]
where $\theta:=\besttheta$ and $\delta:=\frac{1}{128}$.
\end{theorem}
\index{$\delta>0$, subconvexity exponent for $\zeta_K(\tfrac 12)$}

\begin{proof}
In the proof of Proposition~\ref{p:Hecke}, we have seen that if $K$ is cyclic, then $L(s,\rho_K)=L(s,\chi_K)L(s,\overline 
\chi_K)$. We then apply Burgess estimate for Dirichlet characters, which yields
the  upper bound
\[
L\bigl(\tfrac12+it,\rho_K\bigr)\ll_\epsilon
(1+|t|)^{O(1)} |\Delta(K)|^{\frac14 - \frac{1}{16}  + \epsilon}.
\]
If $K$ is an $S_3$-field, then $L(s,\rho_K) = L(s,\pi_K)$ is the
$L$-function of a $\GL(2)$ form of level $|\Delta(K)|$, unitary
central character and weight $0$ or $1$.  We then apply the estimate of
Blomer--Khan~\cite[Thm.1]{Blomer-Khan}, which yields the desired bound.
\end{proof}

The above result allows us to bound smoothed weighted partial sums of
the Dirichlet coefficients of $L(s,\rho_K)$.

\begin{corollary}\label{propsubcv}
For every smooth function with compact support $\Phi:\R_{\geq 0}\to
\C$, $\epsilon>0$, $T\ge 1$ and cubic number field $K$,
\begin{equation*}
\sum_{n\ge 1}\frac{\lambda_K(n)}{n^{1/2}}
\Phi\left(\frac nT\right)\ll_{\epsilon,\Phi}
T^\epsilon |\Delta(K)|^{\theta+\epsilon}.
\end{equation*}
\end{corollary}
\begin{proof}
Applying Mellin inversion, we obtain
\begin{equation*}
\begin{array}{rcl}
\displaystyle\sum_{n\ge 1}\frac{\lambda_K(n)}{n^{1/2}}
\Phi\left(\frac nT\right)&=&
\displaystyle\frac{1}{2\pi i}
\int_{\Re(s)=2}L\bigl(\tfrac{1}{2}+s,\rho_K\bigr)
\widetilde{\Phi}(s)T^sds\\[.2in]
&\ll_{\epsilon,N} & \displaystyle  
T^\epsilon \underset{|t|\le T^\epsilon}\max
\bigl|L\bigl(\tfrac{1}{2}+\epsilon+it,\rho_K\bigr)\bigr|+ 
\underset{|t|> T^\epsilon}
\max |t|^{-N} \bigl|L\bigl(\tfrac{1}{2}+\epsilon+it,\rho_K\bigr)\bigr|,
\end{array}
\end{equation*}
where the bound follows by shifting the integral contour to the line
$\Re(s)=\epsilon$, and using the rapid decay of the Mellin transform
$\widetilde \Phi(\epsilon+it)\ll_N |t|^{-N}$ for $|t|\ge 1$. The corollary now
follows from Theorem~\ref{p_Wu} and the Phragm\'en--Lindel\"of principle, the 
upper-bound on the vertical line $\frac12+it$ being transported to the vertical line 
$\frac12+\epsilon+it$.
\end{proof}

We continue with the \emph{approximate functional equation} which gives the
value of $L(\frac12,\rho_K)$ as a sum of its Dirichlet coefficients
$\lambda_K$. Let $G(u)$ be an even, bounded and holomorphic function in
the strip $|\Re(u)|<A$, where  \(A\) is sufficiently large, and normalized by $G(0)=1$. For example~\cite[p.99]{IK}, we could fix \(G(u):=(\cos \frac{\pi u}{3A})^{-1}\). Define for $y\in
\R_{>0}$
\index{$G(s)$, choice of an even holomorphic function}
\index{$V^{\pm}$, test function in the approximate functional equation}
\begin{equation}\label{defV}
V^\pm(y):=\frac{1}{2\pi
  i}\int_{\Re(u)=3}y^{-u}G(u)\frac{\gamma^\pm(1/2+u)}{\gamma^\pm(1/2)}\frac{du}{u}.
\end{equation}
We have that $V^\pm(y)$ is a rapidly decaying function as $y\to \infty$ that 
extends continuously at the origin with $V^\pm(0)=1$.
\begin{proposition}\label{propAFE}
For every cubic number field $K$ with $\pm\Delta(K)\in\R_{>0}$, we have
\begin{equation}
L\bigl(\tfrac{1}{2},\rho_K\bigr)=
2\sum_{n=1}^\infty
\frac{\lambda_K(n)}{n^{1/2}}V^\pm\Bigl(\frac{n}{|\Delta(K)|^{1/2}}\Bigr).
\end{equation}
\end{proposition}
\begin{proof}
In view of the functional equation of Proposition~\ref{p:Hecke}, this
is~\cite[Thm.5.3]{IK}.
\end{proof}

\subsection{Upper bounds on smooth sums of $\lambda_n(f)$}
\label{s_lambdaf}
Let $f\in V(\Z)^{\rm irr}$ be an irreducible binary cubic form and
recall the Dirichlet series $D(s,f)$ with Dirichlet coefficients
$\lambda_n(f)$ defined in \S\ref{sec:Artin}.
\begin{definition}\label{d:Ep}
For $f\in V(\Z)^{\rm irr}$ and a prime $p$, define $E_p(s,f)$ by
\begin{equation*}
D_p(s,f)=L_p(s,\rho_{K_f})E_p(s,f).
\end{equation*}
Let $E(s,f)=\prod_{p} E_p(s,f)$, hence we have
 $D(s,f)=L(s,\rho_{K_f})E(s,f)$.
\end{definition}
\index{$E_p(s,f)$, Euler factor of the form $f$ nonmaximal at $p$}

It follows from Lemma~\ref{l_Euler} that $E_p(s,f)=1$ if $f$ is maximal at $p$, thus
$E(s,f)=\prod\limits_{p\mid \ind(f)} E_p(s,f)$.

We next list the different possible values taken by $E_p(s,f)$.

\begin{lemma}\label{lemEpsf}
Let $f\in V(\Z)^{\rm irr}$ be an irreducible binary cubic form.  For
every prime $p$, we have that $E_p(s,f)$ is a polynomial in $p^{-s}$
of degree at most two. In fact, it is one of
\begin{equation*}
1,\quad 1-p^{-s},\quad 1+p^{-s},\quad (1-p^{-s})^2,\quad
1-p^{-2s},\quad 1+p^{-s}+p^{-2s}.
\end{equation*}
Moreover, if  $p\parallel\ind(f)$, or if the
  splitting type of $f$ at $p$ is $(1^21)$, then $E_p(s,f)$ is of degree at most one, hence it is 
  one 
  of
\begin{equation*}
1,\quad 1-p^{-s},\quad 1+p^{-s}.
\end{equation*}
\end{lemma}
\begin{proof}
We consider each possible splitting type of $f$ seperately.

If $\sigma_p(f)=(0)$, then $D_p(s,f)=1$ and $p^2|\ind(f)$, hence the lemma follows 
from~\eqref{eqEP}. 

If $\sigma_p(f)=(111), (12)$, or
$(3)$, then $f$ is maximal at $p$, thus 
$E_p(s,f)=1$ by Lemma~\ref{l_Euler}, and the lemma follows. 

Suppose next that $\sigma_p(f)=(1^21)$. Then
we claim that the splitting type of $\O_{K_f}$ at $p$ is either $(111)$,
$(12)$, or $(1^21)$, which implies the lemma by \eqref{eqEP} because then either 
$E_p(s,f)=1-p^{-s}$, 
$E_p(s,f)=1+p^{-s}$, 
or $E_p(s,f)=1$, respectively. Indeed, when $f$ is nonmaximal at $p$, Proposition 
\ref{p:nonmaximal} implies
that by replacing $f$ with a $\GL_2(\Z)$-translate, we may assume that
$f(x,y)=ax^3+bx^2y+pcxy^2+p^2dy^3$, where $p\nmid b$.  The overorder
$S$ of $R_f$ having index $[S:R_f]=p$ corresponds to the form
$g(x,y)=pax^3+bx^2y+cxy^2+dy^3$. Now the splitting type $\sigma_p(g)$
is either $(111)$, $(12)$, or $(1^21)$. In the former two cases, $S$
is maximal at $p$ and the claim is proved. In 
the last case, the claim
follows by induction on the index, by repeating the argument with $g$ instead of $f$.

Suppose finally that $\sigma_p(f)=(1^3)$, then $D_p(s,f)=1$, hence 
$E_p(s,f)=L(s,\rho_{K_f})^{-1}$ is a polynomial in $p^{-s}$ of degree at most two given 
by~\eqref{eqEP}.
Suppose moreover that $p\parallel\ind(f)$. We need to show that $E_p(s,f)$ is 
of degree at most one.  From
Proposition \ref{p:nonmaximal}, we may assume that $f(x,y)$ is of the
form $ax^3+pbx^2y+pcxy^2+p^2dy^3$. The index-$p$ overorder $S$ of
$R_f$ must be maximal at $p$, which implies that the binary cubic form
corresponding to $\O_{K_f}\otimes\Z_p$ is
$pax^3+pbx^2y+cxy^2+dy^3$. Clearly, the splitting type of $\O_{K_f}$ at $p$
is $(1^21)$ or $(1^3)$. Thus $E_p(s,f)=1-p^{-s}$ or $E_p(s,f)=1$, respectively.
\end{proof}

We obtain the following result analogous to Corollary \ref{propsubcv}
for the coefficients $\lambda_n(f)$ where $f$ is an irreducible (not
necessarily maximal) binary cubic form.
\begin{proposition}\label{p_sum_lambdan}
Let $\Phi:\R_{\geq 0}\to \C$ be a smooth function rapidly
decaying at infinity. For every $f\in V(\Z)^{\rm irr}$, $\epsilon>0$ and $T\ge 1$,
\begin{equation}\label{eqpartialLfsum}
 \sum_{n\ge 1}
\frac{\lambda_n(f)}{n^{1/2}}
\Phi\left(\frac nT \right)
\ll_{\epsilon,\Phi}
\ind(f)^{-2\theta}
|\Delta(f)|^{\theta+\epsilon} T^\epsilon,
\end{equation}
where $\theta=\besttheta$ is as in Theorem~\ref{p_Wu}.
\end{proposition}

\begin{proof}
The proof is similar to that of Corollary~\ref{propsubcv}.  We
have that the left-hand side is equal to
\[
\frac{1}{2\pi i}\int_{\Re(s)=2} T^s\sum_{n\ge 1}
\frac{\lambda_n(f)}{n^{\frac12+s}} \widetilde \Phi(s)ds
=\frac{1}{2\pi i}\int_{\Re(s)=2}T^sL\bigl(\tfrac12+s,\rho_K\bigr)
\prod_{p\mid\ind(f)}E_p\bigl(\tfrac12+s,f\bigr)
\widetilde \Phi(s)
ds.
\]
For $\Re(s)\geq 0$, these local factors $E_p(\tfrac12+s,f)$ are absolutely
bounded, (indeed by the number $4$). We have the elementary estimate
\begin{equation*}
\prod_{p\mid\ind(f)}E_p\bigl(\tfrac12+s,f\bigr)\leq
\prod_{p\mid\ind(f)}4\ll_\epsilon |\ind(f)|^\epsilon.
\end{equation*}
As before, pulling the line of integration to $\Re(s)=\epsilon$, we
deduce that
\[
 \sum_{n\ge 1}
\frac{\lambda_n(f)}{n^{1/2}}
\Phi\left(\frac nT \right)
\ll_{\epsilon,\Phi}
T^{\epsilon}
|\Delta(K_f)|^{\theta} |\Delta(f)|^{\epsilon},
\]
from which the assertion follows since $\Delta(f) = \ind(f)^2
\Delta(K_f)$.
\end{proof}

%MARKER
In our next result below (Theorem~\ref{thm_AFE2}), we give a more
precise estimate of the smoothed partial sums of $\lambda_n(f)$ when
we use $\Phi=V^\pm$ as a smoothing function. We start by defining, for
an irreducible binary cubic form $f\in V(\Z)^\irr$, such that
$\pm\Delta(f)\in\R_{>0}$, the quantity
$S(f)$:
%\index{$S(f)$, Dirichlet sum of length $|\Delta(f)|^{\frac12}$ of $\lambda_n(f)$}
\index{{\it S}({\it f} ), truncated Dirichlet sum of $\lambda_n(f)$}
\begin{equation}\label{defSf}
  S(f):=\sum_{n\ge 1} \frac{\lambda_n(f)}{n^{1/2}}
  V^{\pm}\Bigl(\frac{n}{|\Delta(f)|^{1/2}}\Bigr).
\end{equation}
If $f\in V(\Z)^{\irr,\max}$ is irreducible and maximal, then
$2S(f)=L\bigl(\tfrac12,\rho_{K_f}\bigr)$
by Corollary~\ref{c_Euler} and Proposition~\ref{propAFE}.

For general irreducible $f\in V(\Z)^{\irr}$,
Proposition \ref{p_sum_lambdan} yields the bound
\begin{equation}\label{Lf_bound}
S(f)\ll_\epsilon \ind(f)^{-2\theta}|\Delta(f)|^{\theta+\epsilon}.
\end{equation}
Moreover, we have $D(\tfrac12,f) = L(\tfrac12,\rho_{K_f}) E(\tfrac12,f)$ and 
\begin{equation}\label{boundEf}
E\bigl(\tfrac12,f\bigr) 
=\prod_{p|\ind(f)} 
\Bigl(1+O\bigl(p^{-\frac12}\bigr)\Bigr) =  |\ind(f)|^{o(1)},
\end{equation} 
which implies that the same upper bound as~\eqref{Lf_bound} holds for $D(\tfrac12,f)\ll_\epsilon 
\ind(f)^{-2\theta}|\Delta(f)|^{\theta+\epsilon}$.

\begin{definition}\label{defepm}
For $f\in V(\Z)^{\rm irr}$, a prime $p\mid \ind(f)$, and an integer $m\ge 0$,
define $e_{p,m}(f)$ from the following power series expansion:
\[
\frac{E_p\bigl(\tfrac12-s,f\bigr)}
{E_p\bigl(\tfrac12+s,f\bigr)} 
= p^{2s-1} \sum^\infty_{m=0} e_{p,m}(f) p^{m(1/2-s)}.
\]
Recall from Definition~\ref{d:Ep} that $E_p(s,f)$ is a polynomial in
$p^{-s}$ of degree at most two. If $p\nmid \ind(f)$, let $e_{p,m}(f)=0$ for every $m\ge 0$.
\end{definition}
\index{$e_{p,m}(f)$, coefficients of Euler factor of $f$ nonmaximal at $p$}

\begin{examples*}
{\bf (a)} $E_p(s,f)=1- p^{-s}$: In this case, we have
\begin{equation*}
\begin{array}{rcl}
\displaystyle\frac{p}{p^{2s}}
\frac{E_p\bigl(\tfrac12-s,f\bigr)}{E_p\bigl(\tfrac12+s,f\bigr)}&=&
\displaystyle\frac{p}{p^{2s}}\Bigl(1-\frac{p^s}{p^{1/2}}\Bigr)
\Bigl(1-\frac{1}{p^{1/2+s}}\Bigr)^{-1}
\\[.2in]&=&\displaystyle
\Bigl(\frac{p}{p^{2s}}-\frac{p^{1/2}}{p^s}\Bigr)
\Bigl(\sum_{n\geq 0}\frac{1}{p^{n/2+ns}}\Bigr)
\\[.2in]&=&\displaystyle
0-\frac{p^{1/2}}{p^s}+\frac{p- 1}{p^{2s}}+\frac{p^{1/2}- p^{-1/2}}{p^{3s}}
+\cdots+\frac{p^{-(m-4)/2}- p^{-(m-2)/2}}{p^{ms}}+\cdots
\end{array}
\end{equation*}
It therefore follows that we have
\begin{equation*}
e_{p,0}(f)=0,\quad e_{p,1}(f)=-1,\quad e_{p,2}(f)=1-\frac{1}{p},\quad
e_{p,m}(f)=(p^{-m+2}- p^{-m+1}),
\end{equation*}
for all $m\geq 3$. If $E_p(s,f)=1+p^{-s}$, we obtain similar formulas.

\medskip

\noindent{\bf (b)} $E_p(s,f)=(1- p^{-s})^2$: In this case, we have
\begin{equation*}
\begin{array}{rcl}
\displaystyle\frac{p}{p^{2s}}
\frac{E_p\bigl(\tfrac12-s,f\bigr)}{E_p\bigl(\tfrac12+s,f\bigr)}&=&
\displaystyle\frac{p}{p^{2s}}\Bigl(1-\frac{p^s}{p^{1/2}}\Bigr)^2
\Bigl(1-\frac{1}{p^{1/2+s}}\Bigr)^{-2}
\\[.2in]&=&\displaystyle
\Bigl(\frac{\sqrt{p}}{p^{s}}-1\Bigr)^2
\Bigl(\sum_{n\geq 0}\frac{1}{p^{n/2+ns}}\Bigr)^2
\\[.2in]&=&\displaystyle
\Bigl(1-2\frac{p^{1/2}}{p^s}+\frac{p}{p^{2s}}\Bigr)
\Bigl(1+\frac{2}{p^{1/2+s}}+\frac{3}{p^{1+2s}}+\frac{4}{p^{3/2+3s}}+\cdots\Bigr)
\\[.2in]&=&\displaystyle
1+\Bigl(\frac{2}{p^{1/2}}-2p^{1/2}\Bigr)\frac{1}{p^s}
+\Bigl(p+\frac{3}{p}-4\Bigr)\frac{1}{p^{2s}}
+\Bigl(2p^{1/2}-\frac{6}{p^{1/2}}+\frac{4}{p^{3/2}}\Bigr)\frac{1}{p^{3s}}+\cdots,
\end{array}
\end{equation*}
where the coefficient of $1/p^{ms}$ is $\ll m/p^{(m-4)/2}$.
It therefore follows that we have
\begin{equation*}
e_{p,0}(f)=1,\quad e_{p,1}(f)=-2+\frac{2}{p},\quad
e_{p,2}(f)=1-\frac4{p}+\frac3{p^2},
\quad e_{p,m}(f)\ll \frac{m}{p^{m-2}},
\end{equation*}
for all $m\geq 3$.

\medskip

\noindent{\bf (c)} $E_p(s,f)=1+p^{-s}+p^{-2s}$: In this case, we have
\begin{equation*}
\begin{array}{rcl}
\displaystyle\frac{p}{p^{2s}}
\frac{E_p\bigl(\tfrac12-s,f\bigr)}{E_p\bigl(\tfrac12+s,f\bigr)}&=&
\displaystyle\Bigl(1+\frac{p^{1/2}}{p^{s}}+\frac{p}{p^{2s}}\Bigr)
\Bigl(1+\frac{1}{p^{1/2+s}}+\frac{1}{p^{1+2s}}\Bigr)^{-1}
\\[.2in]&=&\displaystyle
\displaystyle\Bigl(1+\frac{p^{1/2}}{p^{s}}+\frac{p}{p^{2s}}\Bigr)
\Bigl(1-\frac{1}{p^{1/2+s}}+\frac{1}{p^{3/2+3s}}+\cdots\Bigr)
\\[.2in]&=&\displaystyle
1+\Bigl(p^{1/2}-\frac{1}{p^{1/2}}\Bigr)\frac{1}{p^s}
+(p-1)\frac{1}{p^{2s}}
+\Bigl(\frac{1}{p^{3/2}}-p^{1/2}\Bigr)\frac{1}{p^{3s}}+\cdots,
\end{array}
\end{equation*}
where the coefficient of $1/p^{ms}$ is $\ll_\epsilon p^{\epsilon m}/p^{(m-4)/2}$.
It therefore follows that once again we have
\begin{equation*}
e_{p,0}(f)=1,\quad e_{p,1}(f)=1-\frac{1}{p},\quad
e_{p,2}(f)=1-\frac1{p},
\quad e_{p,m}(f)\ll_\epsilon \frac{p^{\epsilon m}}{p^{m-2}},
\end{equation*}
for all $m\geq 3$.
\end{examples*}

\medskip

For every integer $k \ge 1$, define 
$e_k(f)$
multiplicatively as 
\[
 e_k(f) := \prod_{p|k} e_{p,{v_p(k)}}(f).
\]
If there exists a prime $p|k$ at which $f$ is maximal, then
$e_k(f)=0$ because $p\nmid \ind(f)$ which implies $e_{p,v_p(k)}(f)=0$. In other words, $e_k(f)$ 
is 
supported on the integers $k$ all of whose prime 
factors divide $\ind(f)$.
\begin{proposition}\label{Efs} 
For every $f\in V(\Z)^{\rm irr}$, and $\Re(s)>-\tfrac12$,
\[
\frac{E\bigl(\tfrac12-s,f\bigr)}
{E\bigl(\tfrac12+s,f\bigr)} 
= \rad(\ind(f))^{2s-1} 
\sum^\infty_{k=1} e_k(f) k^{1/2-s}.
\]
\end{proposition}
\begin{proof}
Since $E(s,f)=\prod\limits_{p|\ind(f)} E_p(s,f)$, the proposition follows from 
Definition~\ref{defepm}, 
and from Lemma~\ref{lemEpsf} which implies that $E_p(\tfrac12+s,f)$ has no zero for 
$\Re(s)>-\tfrac12$.
\end{proof}

We will need the following result, bounding the values of $|e_k(f)|$.
\begin{proposition}\label{ekbound}
For every $f\in V(\Z)^{\rm irr}$, $\epsilon >0$, and $k\ge 1$,
\[
 e_k(f) \ll_\epsilon  k^\epsilon,
\]
where the multiplicative constant depends only on $\epsilon$.
If $k$ is powerful, then we have the improved bound
\[
 e_k(f)  \ll_\epsilon \frac{\rad(k)^2}{k} k^\epsilon.
\]
\end{proposition}
\begin{proof}
The first claim of the proposition would follow from the identity
$e_{p,m}(f)\ll m+p^{\epsilon m}$. The second claim would follow from the
identities $e_p(f),e_{p,2}(f)\ll 1$ and $e_{p,m}(f)\ll_\epsilon
\frac{m+p^{\epsilon m}}{p^{m-2}}$ for $m\geq 3$.

These identities have been verified in Examples {\bf (a)}, {\bf (b)}, and
{\bf (c)} above. (Note that Example {\bf (a)} implies the result for
$E_p(s,f)=(1-p^{-2s})$ and also that the case of $E_p(s,f)=(1+p^{-s})$
is identical to that of Example {\bf (a)}.) This concludes the proof of
the proposition.
\end{proof}

Next, we fix a single form $f$, and analyze the coefficients $e_k(f)$.
\begin{proposition}\label{propklarge}
Let $f\in V(\Z)^{\rm irr}$, and write
$\ind(f)=q_1q_2$, where $q_1$ is squarefree, $(q_1,q_2)=1$, and $q_2$ is
powerful. Then $e_{(\cdot)}(f):\Z_{\geq 1}\to\R$ is supported on
multiples of $q_1$. Namely $q_1\nmid k$ implies $e_k(f) =0$.
\end{proposition}
\begin{proof}
Since $q_1$ is squarefree, it follows from Lemma \ref{lemEpsf} that
for every prime $p\mid q_1$, we have $E_p(s,f)$ is one of $1$, or
$1\pm p^{-s}$. Observe from Example {\bf (a)} above that
$e_{p,0}(f)=0$. The proposition follows immediately.
\end{proof}

The following is an \emph{unbalanced} approximate function equation for
$D(s,f)$ analogous to Proposition~\ref{propAFE} for $L(s,\rho_K)$.
\begin{theorem}\label{thm_AFE2}
For every $f\in V(\Z)^{\rm irr}$,
\begin{equation*}
S(f)=D\bigl(\tfrac 12,f\bigr)-
\sum^\infty_{k=1} 
\frac{e_k(f)k^{1/2}}
{\rad(\ind(f))} 
\sum^\infty_{n=1}\frac{\lambda_n(f)}{n^{1/2}}V^{\sgn(\Delta(f))}\Bigl(
\frac{\ind(f)^2 k n}{\rad(\ind(f))^2|\Delta(f)|^{\frac 12}}\Bigr).
\end{equation*}
\end{theorem}

\begin{proof}
To ease notation for the proof, we let $\pm := {\sgn(\Delta(f))}$ and $K:=K_f$.
We begin by noting that Mellin inversion yields
\begin{equation}\label{eqMelV}
\widetilde{V^\pm}(s)=\frac{G(s)}{s}\frac{\gamma^\pm(\tfrac12+s)}{\gamma^\pm(\tfrac12)},
\end{equation}
implying that $\widetilde{V^\pm}(s)$ decays rapidly and has a pole at $s=0$ with residue $1$.
Hence, by shifting the line of integration, we obtain
\begin{equation*}
\begin{array}{rcl}
S(f)&=&\displaystyle
\int_{\Re(s)=2}D\bigl(\tfrac12+s,f\bigr)|\Delta(f)|^{s/2}
\widetilde{V^\pm}(s)\frac{ds}{2\pi i}
\\[.15in]&=&\displaystyle
D\bigl(\tfrac12,f\bigr)+
\int_{\Re(s)=-1/4}D\bigl(\tfrac12+s,f\bigr)|\Delta(f)|^{s/2}
\widetilde{V^\pm}(s)\frac{ds}{2\pi i}.
\end{array}
\end{equation*}
The functional equation for $L(s+\tfrac12,\rho_K)$ is
\begin{equation*}
L\bigl(\tfrac12+s,\rho_K)\gamma^\pm\bigl(\tfrac12+s)|\Delta(K)|^{\frac{s}{2}}
=L\bigl(\tfrac12-s,\rho_K)\gamma^\pm\bigl(\tfrac12-s)|\Delta(K)|^{-\frac{s}{2}}.
\end{equation*}
Therefore, we have
\begin{equation*}
\begin{array}{rcl}
S(f)-D(\tfrac12,f)&=&\displaystyle
\int_{\Re(s)=-1/4}L\bigl(\tfrac12+s,\rho_K\bigr)E(\tfrac12+s,f)
|\Delta(f)|^{s/2}\widetilde{V^\pm}(s)\frac{ds}{2\pi i}
\\[.15in]&=&\displaystyle
\int_{\Re(s)=-1/4}L\bigl(\tfrac12-s,\rho_K\bigr)
\frac{\gamma^\pm(\tfrac12-s)}{\gamma^\pm(\tfrac12+s)}
E(\tfrac12+s,f)|\Delta(K)|^{-s}
|\Delta(f)|^{s/2}\widetilde{V^\pm}(s)\frac{ds}{2\pi i}
\\[.15in]&=&\displaystyle
\int_{\Re(s)=1/4}D\bigl(\tfrac12+s,f\bigr)
\frac{E(\tfrac12-s,f)}{E(\tfrac12+s,f)}|\Delta(f)/q^4|^{s/2}
\frac{\gamma^\pm(\tfrac12+s)}{\gamma^\pm(\tfrac12-s)}\widetilde{V^\pm}(-s)\frac{ds}{2\pi i},
\end{array}
\end{equation*}
where the final equality follows since $\Delta(f)=q^2\Delta(K)$, where
we have set $q:=\ind(f)$.  
As a consequence of the above and \eqref{eqMelV}, we have
\begin{equation*}
  \frac{\gamma^\pm(\tfrac12+s)}{\gamma^\pm(\tfrac12-s)}\widetilde{V^\pm}(-s)=
  -\frac{G(s)}{s}\frac{\gamma^\pm(\tfrac12+s)}{\gamma^\pm(\tfrac12)}=-\widetilde{V^\pm}(s),
\end{equation*}
which we inject in the previous equality:
\begin{equation}\label{eqapeftemp}
\begin{array}{rcl}
D(\tfrac12,f)-S(f)&=&\displaystyle 
\int_{\Re(s)=1/4}D\bigl(\tfrac12+s,f\bigr)
\frac{E(\tfrac12-s,f)}{E(\tfrac12+s,f)}|\Delta(f)/q^4|^{s/2}
\widetilde{V^\pm}(s)\frac{ds}{2\pi i}\\[.2in]
&=&\displaystyle
\int_{\Re(s)=1/4}D\bigl(\tfrac12+s,f\bigr)
\Bigl( \rad(q)^{2s-1} 
\sum^\infty_{k=1} e_k(f) k^{1/2-s}  \Bigr)|\Delta(f)/q^4|^{s/2}
\widetilde{V^\pm}(s)\frac{ds}{2\pi i},
\end{array}
\end{equation}
where the final equality follows from Proposition \ref{Efs}.
The summand corresponding to $k$ in the second line of
\eqref{eqapeftemp} yields $\rad(q)^{-1}e_k(f)k^{1/2}$ times the integral
\begin{equation*}
\int_{\Re(s)=1/4} D\bigl(\tfrac12+s,f\bigr)
\Bigl(\frac{|\Delta(f)|^{\frac12}\rad(q)^2}{kq^2}
\Bigr)^s
\widetilde{V^\pm}(s)\frac{ds}{2\pi i}
=
\sum_{n\geq 1}\frac{\lambda_n(f)}{n^{1/2}}V^\pm\Bigl(
\frac{nkq^2}{\rad(q)^2|\Delta(f)|^{\frac 12}}\Bigr).
\end{equation*}
Theorem~\ref{thm_AFE2} follows by summing over $k\ge 1$.
\end{proof}

We end this section with the following remark.
\begin{remark}{\rm 
When we consider sums weighted by the function $V^{\pm}(\cdot /X)$, which
is rapidly decaying, we say that the {\it length of the
  sum} is at most $X^{1+\epsilon}$ (since we have that $V^{\pm}(y)$ is negligible for
  $y>X^\epsilon$).

Suppose $f\in V(\Z)^\irr$ has large index $q=\ind(f)$, then all of the inner
sums arising in Theorem \ref{thm_AFE2} to express $S(f)-D(\tfrac12,f)$
are always significantly shorter than the sum defining $S(f)$.
Indeed, the sum defining $S(f)$ has length
$|\Delta(f)|^{1/2+\epsilon}$. The length of any inner sum arising in
Theorem \ref{thm_AFE2} is easily computed. Let $q=q_1q_2$, where
$q_1$ is squarefree, $(q_1,q_2)=1$, and $q_2$ is powerful. Then note
that we have
\begin{equation*}
\frac{q^2}{\rad(q)^2}=\frac{q_2^2}{\rad(q_2)^2}\geq q_2,
\end{equation*}
with equality if and only if the exponent of every prime dividing
$q_2$ is $2$.  Also note that we have $q_1|k$ from Proposition
\ref{propklarge}. Therefore, the length of the inner sum is at most 
$|\Delta(f)|^{1/2+\epsilon}/\ind(f)$.}
\end{remark}

\section{Counting binary cubic forms using
  Shintani zeta functions}\label{secszf}
In this section we recall the asymptotics for the number of
$\GL_2(\Z)$-orbits of integral binary cubic forms ordered by
discriminant. We will impose congruence conditions modulo positive
integers $n$ and study how the resulting error terms depend on
$n$. This section is organized as follows: first, in
\S\ref{sub:basic-shintani}, we collect results from the theory of
Shintani zeta functions corresponding to the representation of $\GL_2$
on $V$. Next, we use standard counting methods to determine the
required asymptotics in~\S\ref{sub:bound-shintani}, and moreover give
an explicit bound on the error terms. Finally, in \S\ref{s_polya}, we
prove a smoothed analogue of the P\'olya--Vinogradov inequality in the
setting of cubic rings.

\subsection{Functional equations, poles, and residues of Shintani zeta
  functions}\label{sub:basic-shintani}

Let $n$ be a positive integer and let $\phi:V(\Z/n\Z)\to \C$ be a
$\GL_2(\Z/n\Z)$-invariant function. Let $\xi(\phi,s)$ denote the
\emph{Shintani zeta function} defined by
\index{$\xi^\pm(\phi,s)$, Shintani zeta function with congruence function $\phi$}
\begin{equation}\label{defxi}
  \xi^\pm(\phi,s):=\sum_{f\in\frac{V(\Z)^\pm}{\GL_2(\Z)}}\phi(f)
\frac{|\Delta(f)|^{-s}}{|\Stab(f)|},
\end{equation}
where we abuse notation and also denote the composition of $\phi$ with
the reduction modulo $n$ map $V(\Z)\twoheadrightarrow V(\Z/n\Z)$ by
$\phi$.  For a function $\psi:V^*(\Z/n\Z) \to \C$, let
$\xi^{*\pm}(\psi,s)$ denote the \emph{dual} Shintani zeta function
defined in~\cite[Def.4.2]{TT}.
\index{$\xi^{* \pm}(\psi,s)$, dual Shintani zeta function with congruence function $\psi$}

\begin{theorem}[F.\ Sato--Shintani]   \label{thshinfs}
The functions $\xi^\pm$ and $\xi^{*\pm}$ have a meromorphic
continuation to the whole complex plane, and satisfy the functional equations
\begin{equation*}\Bigl(\begin{array}{c}
\xi^+(\phi,1-s)\\[.03in]\xi^-(\phi,1-s)
\end{array}
\Bigr)
=n^{4s}
\frac{(3^6\pi^{-4})^s}{18}
\Gamma\Bigl(s-\frac16\Bigr)
\Gamma(s)^2
\Gamma\Bigl(s+\frac16\Bigr)
\Bigl(\begin{array}{cc}
\sin 2\pi s & \sin \pi s\\
3\sin \pi s & \sin 2\pi s
\end{array}
\Bigr)
\Bigl(\begin{array}{c}
\xi^{*+}(\hat\phi,s)\\[.03in]\xi^{*-}(\hat\phi,s)
\end{array}
\Bigr),
\end{equation*}
where $\hat{\phi}:V^*(\Z/n\Z) \to \C$ is the Fourier transform of
$\phi$ as in \S\ref{s_binary_Fp}.
\end{theorem}

\begin{proof}
This is due to Shintani~\cite{Shintani} for $n=1$ and Sato~\cite{Sato}
for general $n$. See also \cite[Thm.4.3]{TT} for a modern exposition.
In fact the above theorem is a special case because the congruence function $\phi$ in 
\cite{Sato,TT} is not necessarily $\GL_2(\Z/n\Z)$-invariant.
In the more general case of an arbitrary congruence function $\phi:V(\Z/n\Z)\to \C$, the 
Shintani zeta functions, respectively its dual, are defined using the principal subgroup 
$\Gamma(n)$ and summing $f$ over the quotient $V(\Z)^\pm/\Gamma(n)$, respectively 
$V^*(\Z)^\pm/\Gamma(n)$.
Assuming that $\phi$ is $\GL_2(\Z/n\Z)$-invariant, the general definition reduces to~\eqref{defxi}.
\end{proof}

The possible poles of $\xi^{\pm}(\phi,s)$ occur at $1$ and $5/6$, and the residues shall be
given in Proposition~\ref{t:residues} below.  First we define
\index{$\alpha^\pm,\beta^\pm,\gamma^\pm$, residues of Shintani zeta function}
\begin{equation*}
  \begin{array}{ccccccc}
    \displaystyle\alpha^+&:=&\displaystyle\frac{\pi^2}{36};\;\;\;\;\; \beta^+&:=&\displaystyle\frac{\pi^2}{12};\;\;\;\;\;\gamma^+&:=&\displaystyle\zeta(1/3)\frac{2\pi^2}{9\Gamma(2/3)^3};\\[.2in]
    \displaystyle\alpha^-&:=&\displaystyle\frac{\pi^2}{12};\;\;\;\;\; \beta^-&:=&\displaystyle\frac{\pi^2}{12};\;\;\;\;\;\gamma^-&:=&\displaystyle\zeta(1/3)\frac{2\sqrt{3}\pi^2}{9\Gamma(2/3)^3}.
  \end{array}
\end{equation*}
Then the functions $\xi^\pm(s)=\xi^\pm(1,s)$, corresponding to the
constant function $\phi=1$, have residues
$\alpha^\pm+\beta^\pm$ at $s=1$ and $\gamma^\pm$ at $s=5/6$. Moreover,
the pole at $1$ has the following interpretation: the term
$\alpha^\pm$ comes from the contribution of irreducible cubic forms
and the term $\beta^\pm$ comes from the contribution of reducible
cubic forms.

As before, let $n$ be a positive integer. Let $\phi:V(\Z/n\Z)\to\C$ be a
function of the form
$\phi=\prod_{p^\beta \parallel n}\phi_{p^\beta}$, where $\phi_{p^\beta}:V(\Z/p^\beta \Z)\to\C$ and 
$\beta:=v_p(n)$.  
We define the linear functionals $\cA_{p^\beta}$, $\cB_{p^\beta}$, and $\cC_{p^\beta}$ to be
\begin{equation}\label{eqcabc}
  \cA_{p^\beta}(\phi_{p^\beta}):=\widehat{\phi_{p^\beta}}(0),\;\;\;\;\;\;\;\;
  \cB_{p^\beta}(\phi_{p^\beta}):=\widehat{\phi_{p^\beta} \cdot b_p}(0),\;\;\;\;\;\;\;\;
  \cC_{p^\beta}(\phi_{p^\beta}):=\widehat{\phi_{p^\beta}\cdot c_p}(0),
\end{equation}
where $\phi_{p^\beta}\mapsto \widehat{\phi_{p^\beta}}$ is the Fourier transform of functions on 
$V(\Z/p^\beta \Z)$ from 
\S\ref{sec:pre} and 
where 
the 
functions 
\[
b_p,c_p:V(\Z/p^\beta\Z) \twoheadrightarrow V(\Z/p\Z)\to\R_{\ge 0}
\]
are $\GL_2(\Z/p^\beta\Z)$-invariant and defined in Table \ref{tabbc}.
\index{$b_p(f),c_p(f)$, densities of splitting types}
\index{$\cA_n(\phi),\cB_n(\phi),\cC_n(\phi)$, linear functionals for residues of $\xi^\pm(\phi,s)$}
  We
define $\cA_n(\phi)$, $\cB_n(\phi)$, and $\cC_n(\phi)$
multiplicatively as the product over $p^\beta\parallel n$ of $\cA_{p^\beta}(\phi_{p^\beta})$,
$\cB_{p^\beta}(\phi_{p^\beta})$, and $\cC_{p^\beta}(\phi_{p^\beta})$, respectively. By multilinearity,
the domain of definition of the functionals $\cA_n$, $\cB_n$, and
$\cC_n$ extends to all functions $\phi:V(\Z/n\Z)\to \C$.  Abusing
notation, we denote the lift of $\phi$ (resp.\ $\phi_{p^\beta}$) to
$V(\widehat{\Z})$ (resp.\ $V(\Z_p)$) also by $\phi$ (resp.\ $\phi_{p^\beta}$).
Note that $\cA_n(\phi)$ can be interpreted as the integral
\begin{equation*}
\cA_n(\phi)=\int_{V(\widehat{\Z})}\phi(f)df=\prod_{p}\int_{V(\Z_p)}\phi_{p^\beta}(f)df,
\end{equation*}
where $\phi_{p^\beta}$ is simply defined to be the function $1$ when $p\nmid
n$. This is true because, under our normalizations $\Vol(V(\Z_p))=1$.

\begin{table}[ht]
\centering
\begin{tabular}{|c | c| c| }
  \hline &&\\[-8pt] Splitting type of $f$ at $p$ & $b_p(f)$ & $(1-p^{-2})c_p(f)$
  \\[4pt] \hline\hline &&\\[-8pt]
  $(111)$ & 3&$(1-p^{-2/3})(1+p^{-1/3})^2$\\[4pt]
  $(12)$&1 &$(1-p^{-4/3})$\\[4pt]
  $(3)$&0 &$(1-p^{-1/3})(1+p^{-1})$\\[4pt]
  $(1^21)$&$\frac{p+2}{p+1}$ &$(1+p^{-1/3})(1-p^{-1})$\\[4pt]
  $(1^3)$&$\frac{1}{p+1}$ &$(1-p^{-4/3})$\\[4pt]
  $(0)$& 1 &$(1-p^{-2})p^{2/3}$\\
\hline
\end{tabular}
\caption{Densities of splitting types}\label{tabbc}
\end{table}
We then have the following expressions for the residues of Shintani 
zeta functions, see~\cite{Sato,DW,TT}.
\begin{proposition}\label{t:residues}
The functions $\xi^\pm(\phi,s)$ are holomorphic on $\C - \{1,5/6\}$ with 
at worst simple poles at $s=1,5/6$ and the residues are given by
\begin{equation*}
  \begin{array}{lcl}
    \operatorname*{Res}\limits_{s=1}\xi^\pm(\phi,s)&=&\alpha^\pm\cdot\cA_n(\phi)+\beta^\pm\cdot\cB_n(\phi),\\[.1in]
    \operatorname*{Res}\limits_{s=5/6}\xi^\pm(\phi,s)&=&\gamma^\pm\cdot\cC_n(\phi).
  \end{array}
\end{equation*}
\end{proposition}

The interpretation of these residues is that the term
$\alpha^\pm\cdot\cA_n(\phi)$ is the main term contribution from counting
irreducible binary cubic forms, the term $\beta^\pm\cdot\cB_n(\phi)$
is the main term contribution from counting reducible binary cubic forms, and
the term $\gamma^\pm\cdot\cC_n(\phi)$ is the secondary term
contribution from counting irreducible binary cubic forms, particularly arising
from cubic rings that are close to being \emph{monogenic}, i.e., that have an
element which generates a subring of small index.

\subsection{Uniform bound for Shintani zeta functions near the abscissa of 
convergence}

We recall the following \emph{tail estimate} due to
Davenport--Heilbronn \cite{DH}. See also \cite{BBP} for a streamlined
proof.
\begin{proposition}[Davenport--Heilbronn] \label{lemunif13}
Let $n$ and $m$ be positive squarefree integers. The number of
$\GL_2(\Z)$-orbits on the set of binary cubic forms having
discriminant bounded by $X$ and splitting type $(1^3)$ at every prime
dividing $n$ and splitting type $(0)$ at every prime dividing $m$ is
bounded by $O_\epsilon(X/(m^4n^{2-\epsilon}))$, where the implied constant is
independent of $X$, $m$, and $n$.
\end{proposition}

Let $p$ be a prime. Recall that for  \(p\neq 3\) the set of $\GL_2(\Z/p\Z)$-orbits on
$V^*(\Z/p\Z)$ (resp.\ $V(\Z/p\Z)$) is classified by the possible
splitting types, namely, $(111)$, $(12)$, $(3)$, $(1^21)$, $(1^3)$,
and $(0)$. For  \(p=3\), one could extend this classification, or, more simply, define  \(E_3(\psi_3):=||\psi_3||_\infty\), which will only affect the multiplicative constants in this paper.

\begin{definition}
For a prime $p$ and a $\GL_2(\Z/p\Z)$-invariant function
$\psi_p$ on $V^*(\Z/p\Z)$ (resp.\ $\phi_p$ on $V(\Z/p\Z)$), we define
\begin{equation*}
\begin{array}{rcl}
E_p(\psi_p)&:=&\displaystyle|\psi_p(111)|+
|\psi_p(12)|+|\psi_p(3)|+|\psi_p(1^21)|
+|\psi_p(1^3)|p^{-2}+|\psi_p(0)|p^{-4},
\end{array}
\end{equation*}
and similarly for $E_p(\phi_p)$.
\end{definition}

Let $n$ be a positive integer, and let $\psi:V^*(\Z/n\Z)\to\C$
(resp.\ $\phi:V(\Z/n\Z)\to\C$) be a $\GL_2(\Z/n\Z)$-invariant
function. If $\psi$ factors as $\psi=\prod_{p^\beta\parallel
  n}\psi_{p^\beta}$, where $\psi_{p^\beta}:V^*(\Z/p^\beta \Z)\to\C$
 are $\GL_2(\Z/p^\beta \Z)$-invariant
functions, then we define
\begin{equation*}
\begin{array}{rcl}
E_n(\psi)&:=&\displaystyle\prod_{\substack{p\parallel n}}E_p(\psi_p)\cdot \prod_{\substack{p^\beta \parallel 
n\\\beta \geq 
2}}\|\psi_{p^\beta }\|_\infty,
\end{array}
\end{equation*}
where $\|\cdot\|_\infty$ denotes the $L^\infty$-norm. 
%We extend $E_n$ to all 
%$\GL_2(\Z/n\Z)$-invariant functions as being the projective cross norm. 
We have a similar 
definition for $E_n(\phi)$.
\index{$E_n(\psi)$, norm of $\psi$ weighted by splitting types}
\begin{proposition}\label{properfirst}
Let $n$ be a positive integer.
Let $\psi$ be a
$\GL_2(\Z/n\Z)$-invariant function on $V^*(\Z/n\Z)$.  For every
$\epsilon>0$ and $t\in \R$, we have
\begin{equation}
\xi^{*\pm}(\psi,1+\epsilon+it)
\ll_{\epsilon}
n^\epsilon E_n(\psi).
\end{equation}
The same bound holds for $\xi^{\pm}(\phi,1+\epsilon+it)$ for a $\GL_2(\Z/n\Z)$-invariant function 
$\phi$ on $V(\Z/n\Z)$.
\end{proposition}

\begin{proof}
Let $q$ be a positive squarefree integer. We say that $\tau$ is a {\it
  splitting type modulo $q$} if $\tau=(\tau_p)_{p\mid q}$ is a
collection of splitting types $\tau_p$ for each prime $p$ dividing
$q$.  Let $q(\tau,1^3)$ $($resp.\ $q(\tau,0))$ denote the product of
primes $p$ dividing $q$, such that $\tau_p=(1^3)$
$($resp.\ $\tau_p=(0))$. That is,
\[
q(\tau,1^3) := \prod_{\substack{ p|q \\ \tau_p = (1^3) 
}} p,\quad
q(\tau,0) := \prod_{\substack{ p|q \\ \tau_p = (0) 
}} p.
\]
We write $n=q\ell$, where $q$ is squarefree, $\ell$ is powerful, and
$(q,\ell)=1$. Given an integral binary cubic form $f$, we have the
factorization $\psi(f)=\psi_q(f)\psi_\ell(f)$, where
$\psi_q:V(\Z/q\Z)\to\C$ and $\psi_\ell:V(\Z/\ell\Z)\to\C$ are
$\GL_2(\Z/q\Z)$-invariant and $\GL_2(\Z/\ell\Z)$-invariant functions,
respectively, and as usual, we are denoting the lifts of $\psi_q$ and
$\psi_\ell$ to $V^*(\Z)$ also by $\psi_q$ and $\psi_\ell$, respectively.
Let $S(q)$ denote the set of splitting types modulo $q$. For $f\in
V^*(\Z)$, the value of $\psi_q(f)$ is determined by the splitting type
$\tau$ modulo $q$ of $f$. For such a splitting type $\tau\in S(q)$, we
accordingly define $\psi_q(\tau):=\psi_q(f)$, where $f\in V^*(\Z)$ is
any element with splitting type $\tau$ modulo $q$.

Let $s=1+\epsilon+it$.
We have
\begin{equation*}
|\xi^{*\pm}(\psi,s)|
 \le 
\|\psi_\ell\|_\infty \cdot
\sum_{\tau\in S(q)}|\psi_q(\tau)|\sum_{m=1}^\infty 
\frac{c_\tau(m)}{m^{1+\epsilon}},
\end{equation*}
where $c_\tau(m)$ denotes the number of $\GL_2(\Z)$-orbits on
the set of elements in $V^*(\Z)$ having discriminant $m$ and splitting
type $\tau$ modulo $q$.  From partial summation, we obtain
\begin{equation*}
\begin{array}{rcl}
\displaystyle\sum_{m=1}^\infty \frac{c_\tau(m)}{m^{1+\epsilon}}&=&
\sum_{k=1}^{\infty}\Bigl(\frac{1}{k^{1+\epsilon}}-\frac{1}{(k+1)^{1+\epsilon}}\Bigr)\sum_{m=1}^kc_\tau(m)
\\[.2in]&\ll_\epsilon &
\sum_{k=1}^{\infty}
\frac{1}{k^{2+\epsilon}}
\sum_{m=1}^kc_\tau(m).
\end{array}
\end{equation*}

From Proposition \ref{lemunif13}, it follows that we have
\begin{equation*}
\sum_{m=1}^kc_\tau(m)\ll_\epsilon k\cdot 
q(\tau,1^3)^{-2+\epsilon}\cdot q(\tau,0)^{-4},
\end{equation*}
where the multiplicative constant is independent of $n$, $\tau$, and $k$.
Therefore, we have
\begin{equation*}
\begin{array}{rcl}
\xi^{*\pm}(\psi,s)&\ll_\epsilon&\displaystyle
\|\psi_\ell\|_\infty \cdot  \sum_{\tau\in S(q)}
|\psi_q(\tau)|q(\tau,1^3)^{-2+\epsilon}\cdot q(\tau,0)^{-4}
\Bigl(\sum_{k=1}^{\infty}\frac{1}{k^{1+\epsilon}}\Bigr)
\\[.2in]&\ll_\epsilon&\displaystyle
n^\epsilon E_n(\psi).
\end{array}
\end{equation*}
In the last equation, we used that 
\[E_n(\psi)=\|\psi_\ell\|_\infty \cdot \prod\limits_{p|q}E_p(\psi_p) = 
\|\psi_\ell\|_\infty \cdot
\sum\limits_{\tau\in 
S(q)} |\psi_q(\tau)| q(\tau,1^3)^{-2} q(\tau,0)^{-4}.
\qedhere
\]
\end{proof}

\subsection{Smooth counts of binary cubic forms satisfying congruence 
conditions}\label{sub:bound-shintani}

As in the previous subsection, let $n$ be a positive integer, and let
$\phi:V(\Z/n\Z)\to\C$ be a $\GL_2(\Z/n\Z)$-invariant function. Let
$\Psi:\R_{> 0}\to \C$ be a smooth function of compact support.  For a
real number $X\ge 1$, define the \emph{counting function}
$N^\pm_\Psi(\phi;X)$ to be
\begin{equation*}
N^\pm_\Psi(\phi;X):=\sum_{f\in\frac{V(\Z)^\pm}{\GL_2(\Z)}}
\frac{\phi(f)}{|\Stab(f)|}\Psi\Bigl(\frac{|\Delta(f)|}{X}\Bigr).
\end{equation*}
Applying the Mellin transform results from Section \ref{sec:dedekind}, and shifting the
line of integration from $\Re(s)=2$ to $\Re(s)=-\epsilon$, with $0<\epsilon<1$, we obtain

\begin{equation}\label{eqnpx52}
\begin{array}{rcl}
N^\pm_\Psi(\phi;X)
&=&\displaystyle\frac{1}{2\pi i}\int_{\Re(s)=2}X^s\xi^\pm(\phi,s)\widetilde{\Psi}(s)ds\\[.2in]
&=&\displaystyle\Res_{s=1}\xi^\pm(\phi,s)\cdot  \widetilde \Psi(1)\cdot X+\Res_{s=5/6}\xi^\pm(\phi,s)
\cdot  \widetilde \Psi(\frac 56)\cdot
X^{5/6}+\cE_\epsilon(\phi,\Psi)
\\[.2in]
&=&\displaystyle(\alpha^\pm\cA_n(\phi)+\beta^\pm\cB_n(\phi))
\cdot  \widetilde \Psi(1)\cdot X
+\gamma^\pm\cC_n(\phi)
\cdot  \widetilde \Psi(\frac 56)\cdot
X^{5/6}+\cE_\epsilon(\phi,\Psi).
\end{array}
\end{equation}
The \emph{error term} $\cE_\epsilon(\phi,\Psi)$ is defined below,
and bounded using the functional equation in Theorem
\ref{thshinfs} and Stirling's asymptotic formula in the form $\Gamma(\sigma+it)\ll_{\sigma} 
(1+|t|)^{\sigma-\frac12} e^{\frac{-\pi |t|}{2}} $ for every $\sigma\not\in \Z_{\le 0}$ and $t\in \R$:

\begin{equation}\label{eqnpx52er}
\cE_\epsilon(\phi,\Psi):=\int_{\Re(s)=-\epsilon}
X^s\xi^\pm(\phi,s)\widetilde{\Psi}(s)\frac{ds}{2\pi i}
\ll_{\epsilon} n^{4+\epsilon}
\operatorname{max}\limits_{t\in \R}
|\xi^{*\pm} (\widehat\phi,1+\epsilon+it)|
E_\infty(\widetilde{\Psi};\epsilon),
\end{equation}
where we define
$E_\infty(\widetilde{\Psi};\epsilon):=\int_{-\infty}^{\infty} \left|
\widetilde \Psi(-\epsilon+it) \right| (1+|t|)^{2+4\epsilon} dt.$
\index{$E_\infty(\widetilde \Psi;\epsilon)$, archimedean norm of $\widetilde \Psi$}

\begin{theorem}\label{countphi}
Let $\Psi:\R_{>0}\to\C$ be a smooth function with compact support and let $\epsilon>0$. Let $n$ be a
positive integer, and write $n=qm$, where $q$ is squarefree,
$(q,m)=1$, and $m$ is powerful.  For every real $X\ge 1$, and
$\GL_2(\Z/n\Z)$-invariant function $\phi:V(\Z/n\Z)\to\C$, we have
\begin{equation*}
\begin{array}{rcl}
\displaystyle  N^\pm_\Psi(\phi;X)=\displaystyle
\bigl(\alpha^\pm\cA_n(\phi)+\beta^\pm\cB_n(\phi)\bigr)\widetilde \Psi(1)\cdot X+
  \gamma^\pm\cC_n(\phi)\cdot  \widetilde \Psi(\frac 56)\cdot X^{5/6}+
  O_{\epsilon}\Bigl(
  n^{4+\epsilon} E_n(\widehat \phi) E_\infty(\widetilde{\Psi};\epsilon)\Bigr).
\end{array}
\end{equation*}
\end{theorem}
\begin{proof}
This follows from
\eqref{eqnpx52}, \eqref{eqnpx52er}, and
Proposition \ref{properfirst}.
\end{proof}

The following lemmas bound $E_n(\widehat \phi)$ for various functions $\phi$.
\begin{lemma}\label{lemEbound}
Let $n$ be a positive integer and $\phi$ be a
$\GL_2(\Z/n\Z)$-invariant function on $V(\Z/n\Z)$. Then we have, for every $\epsilon>0$,
\begin{equation*}
  E_n(\widehat{\phi})\ll_\epsilon 
  n^\epsilon\bigl(\prod_{p\parallel n}p\bigr)^{- 2}\|\phi\|_\infty.
\end{equation*}
\end{lemma}
\begin{proof}
This follows from the definitions
of $E_n$ and $E_p$, along with Corollary \ref{corphihatbound}. 
\end{proof}

Recall from~\S\ref{sec:lambdan} the
function $\lambda_n$, which is a $\GL_2(\Z/\rad(n)\Z)$-invariant
function on $V(\Z/\rad(n)\Z)$.

\begin{lemma}\label{lemE-lambda}
For every positive integer $n$ and 
every $\epsilon>0$,
\begin{equation*}
  E_n(\widehat{\lambda_n})\ll_\epsilon n^\epsilon
  \bigl(\prod_{p\parallel n}p\bigr)^{- 3}
  \bigl(\prod_{p^2\mid n}p\bigr)^{- 2}.
\end{equation*}
\end{lemma}

\begin{proof}
Recall that the functions $\lambda_{p^k}$ are defined modulo $p$ irrespective of $k$. Hence the claimed saving from the factors $p$ with $p^2\mid n$ follows from Lemma \ref{lemEbound}. The additional saving from the factors $p$ with $p\parallel n$ is a consequence of Proposition \ref{c_thetahat}.
\end{proof}

\begin{lemma}\label{lemAC-lambda}
For every prime $p\neq 3$,
\[
 \cA_p(\lambda_p) = \widehat{\lambda_p}(0) = \frac{p^2-1}{p^3}, \quad
\cB_p(\lambda_p) = \widehat{\lambda_pb_p}(0) = \frac{p^3-1}{p^3}, \quad
\cC_p(\lambda_p) = \widehat{\lambda_pc_p}(0) \ll \frac{1}{p^{1/3}}.
\]
\end{lemma}
\begin{proof}
The first equation is derived in Proposition \ref{c_thetahat}. The second equation is derived similarly: we have
\begin{equation*}
\widehat{\lambda_pb_p}(0)=6\cdot \frac{p(p+1)(p-1)^2}{6p^4}+\frac{p+2}{p+1}\cdot\frac{p(p+1)(p-1)}{p^4}=\frac{p^3-1}{p^3}.
\end{equation*}
To prove the final inequality, we write
\begin{equation*}
\begin{array}{rcl}
c_p(111)&=&
\displaystyle(1-p^{-1/3})(1+p^{-1/3})^3\Bigl(1-\frac{1}{p^2}\Bigr)^{-1};\\[.15in]
c_p(3)&=&
\displaystyle(1-p^{-1/3})(1+p^{-1})\Bigl(1-\frac{1}{p^2}\Bigr)^{-1};\\[.15in]
c_p(1^21)&=&
\displaystyle(1+p^{-1/3})(1-p^{-1})\Bigl(1-\frac{1}{p^2}\Bigr)^{-1}.
\end{array}
\end{equation*}
We compute $\widehat{\lambda_pc_p}(0)$ using Proposition \ref{thFT} and obtain
\begin{equation*}
\cC_p(\lambda_p)=
\frac13\Bigl(1-\frac{1}{p}\Bigr)(1-p^{-1/3})
\bigl((1+p^{-1/3})^3-(1+p^{-1})\bigr)+\frac{1}{p}
\Bigl(1-\frac{1}{p}\Bigr)(1+p^{-1/3}),
\end{equation*}
which concludes the proof of the lemma.
\end{proof}

\subsection{Application to cubic analogues of 
P\'olya--Vinogradov}\label{s_polya}

We sum the Artin character over isomorphism classes
of cubic rings. This is a cubic analogue of the P\'olya--Vinogradov
inequality~\cite[Thm.12.5]{IK}, which sums Artin characters over quadratic 
rings. There are
some substantial differences between quadratic and cubic cases: first,
in the cubic case we see the presence of second order terms which do
not occur in the quadratic case. Second, since the parameter space of
cubic rings is four dimensional (as opposed to one dimensional), the
trivial range for summing the Artin character $\lambda_n$ over cubic
rings with discriminant bounded by $X$ is $X\gg n^4$ (as opposed to
$X\gg n$ in the quadratic case).

\begin{theorem}[Cubic analogue of P\'olya--Vinogradov]\label{t_Polya}
Let $p$ be a prime and let $k\geq 2$ be an integer. Let
$\Psi:\R_{>0}\to\C$ be a smooth function with
compact support such that $\int_0^\infty \Psi(x)dx=1$. Then we have
\begin{equation*}
\begin{array}{rcl}
\displaystyle\sum_{f\in\frac{V(\Z)^\pm}{\GL_2(\Z)}}\frac{\lambda_p(f)}{|\Stab(f)|}\Psi\Bigl(\frac{|\Delta(f)|}{X}\Bigr)&=&
\displaystyle\Bigl(\alpha^\pm\frac{p^2-1}{p^3}+\beta^\pm\frac{p^3-1}{p^3}\Bigr)X+
\gamma^\pm \widehat{\lambda_pc_p}(0)\widetilde \Psi(\frac 56)\cdot X^{5/6}+O_{\epsilon,\Psi}(p^{1+\epsilon});
\\[.2in]
\displaystyle\sum_{f\in\frac{V(\Z)^\pm}{\GL_2(\Z)}}\frac{\lambda_{p^k}(f)}{|\Stab(f)|}\Psi\Bigl(\frac{|\Delta(f)|}{X}\Bigr)&=&
\displaystyle\Bigl(\alpha^\pm \widehat{\lambda_{p^k}}(0)+\beta^\pm \widehat{\lambda_{p^k}b_p}(0) \Bigr)X
+\gamma^\pm \widehat{\lambda_{p^k}c_p}(0)\widetilde \Psi(\frac 56)\cdot X^{5/6}+O_{\epsilon,\Psi}(kp^{2+\epsilon}).
\end{array}
\end{equation*}
\end{theorem}
\begin{proof}
This is a consequence of Theorem \ref{countphi} in
conjunction with Propositions~\ref{lemunif13} and \ref{properfirst} and 
Lemma~\ref{lemE-lambda}.
\end{proof}

\section{Sieving to the space of maximal binary cubic forms}\label{sec:switch}

In this section, we employ an inclusion-exclusion sieve to sum over
\emph{maximal} binary cubic forms. To set up this sieve, we need the
following notation. Denote the set of maximal integral binary cubic
forms by $V(\Z)^\max$. For a squarefree positive integer $q$, we let
\index{$q$, square-free integer entering into the sieve}
$\W_q$ denote the set of elements in $V(\Z)$ that are
\index{$\W_q$, elements in $V(\Z)$ nonmaximal at every prime dividing $q$}
\emph{nonmaximal at every prime dividing $q$}. Given a set $S$ with a  \(\GL_2(\Z)\)-action, we let
$\overline{S}:=\frac{S}{\GL_2(\Z)}$ denote the set of
$\GL_2(\Z)$-orbits on~$S$. 
\index{$\overline{T}$, set of $\GL_2(\Z)$-orbits on $T$}
Let $\Psi:\R_{>0}\to\C$ be a smooth
function with compact support, and let $\phi:V(\Z)\to\C$ be a
$\GL_2(\Z)$-invariant function. Then we have
\begin{equation}\label{eqinex}
\sum_{f\in\overline{V(\Z)^{\pm,\max}}}\frac{\phi(f)}{{|\Stab(f)|}}\Psi(|\Delta(f)|)
=\sum_{q\geq 1}\mu(q)\sum_{f\in\overline{\W^{\pm}_q}}\frac{\phi(f)}{{|\Stab(f)|}}\Psi\bigl(|\Delta(f)|\bigr).
\end{equation}
The difficulty in obtaining good estimates for the right-hand side of
\eqref{eqinex} is that the set $\W_q$ is defined via congruence
conditions modulo $q^2$, and a direct application of the results of
Section~\ref{secszf} yields not sufficiently precise error terms for
sums over such sets. We overcome this difficulty in
\S\ref{s_overrings} by using a ``switching trick'', developed in
\cite{BST}, which transforms the sum over $\W_p$ to a weighted sum
over $V(\Z)$, where the weights are defined modulo $p$. We then
combine the results of Section~\ref{secszf} and \S\ref{s_overrings} to
carry out the sieve and obtain improved bounds for the error
term. Finally, in \S\ref{s_switch_applications}, we derive several
applications; notably, we obtain a smoothed version of Roberts'
conjecture, and sum the Artin character $\lambda_K(n)$ over cubic fields $K$.

For a positive squarefree integer $m$ and an integral binary cubic
form $f\in V(\Z)$, denote the number of roots (resp.\ simple roots) in $\P^1(\Z/m\Z)$ of
the reduction of $f$ modulo $m$  by $\omega_m(f)$
(resp.\ $\omega^{(1)}_m(f)$). By the Chinese remainder theorem,
$\omega_m(f)$ and $\omega^{(1)}_m(f)$ are multiplicative in $m$.
\index{$\omega^{(1)}_m(f)$, number of simple roots of $f$ modulo $m$}
\begin{proposition}[{\cite[Eq.(70)]{BST}}]  \label{prop:sqitchclassic}
For every positive squarefree integer $q$ and every function
$\Psi:\R_{>0}\to \C$ of compact support,
\begin{equation*}
\sum_{f\in \overline{\W^{\pm}_q}}
\frac{\Psi\bigl(|\Delta(f)|\bigr)}{{|\Stab(f)|}}
=
\sum_{k\ell\mid q}\mu(\ell)
\sum_{f\in\overline{V(\Z)^{\pm}}}\frac{\omega_{k\ell}(f)}{{|\Stab(f)|}}
\Psi\Bigl(\frac{q^4|\Delta(f)|}{k^2}\Bigr).
\end{equation*}
\end{proposition}

The above identity was proved using the following procedure in
\cite[\S9]{BST}. Every element $f\in \W_q$ corresponds to a ring $R_f$
that is nonmaximal at every prime dividing $q$, hence $R_f$ is
contained in a certain ring $R'$, such that the \emph{index}
$\ind(f):=[R':R_f]$ satisfies $q\mid\ind(f)$ and $\ind(f)\mid q^2$. In
particular, the discriminant of $R'$ is smaller than that of
$R_f$. Then elements in $\overline{\W_q}$ can be counted by counting the
rings $R'$ instead of $R_f$. In what follows, we formalize this
procedure, and adapt it so that we may sum congruence functions over
$\overline{\W_q}$ (Theorem~\ref{thswitch} which is a strenghtening of
Proposition~\ref{prop:sqitchclassic}).

\subsection{Switching to overrings}\label{s_overrings}

We begin with a bijection which allows us to replace sums over
$\overline{\W_q}$ with sums over $\overline{\W_{q_1}}$, for various
integers $q_1\mid q$ with $q_1<q$. Given a set $S\subset V(\Z)$ and an
element $\alpha\in\P^1(\F_p)$, let $S^{(\alpha)}$ denote the set of
elements $f\in S$ such that $f(\alpha)\equiv 0\pmod p$. Then we
have the following result.

\begin{lemma}\label{lembij}
Let $q$ be a positive squarefree number, and let $p$ be a prime such
that $p\mid q$. Then there is a bijection between the following
two sets:
\begin{equation}\label{eqbij}
\Bigl\{f\in \overline{\W_q}\backslash \overline{p\W_{q/p}}\Bigr\}
\bigcup
\overline{\Bigl\{(f,\gamma):f\in p\W_{q/p}^{(\gamma)},\,\gamma\in\P^1(\F_p)\Bigr\}}
\longleftrightarrow
\overline{\Bigl\{
(g,\alpha): g\in \W_{q/p}^{(\alpha)},\,\alpha\in\P^1(\F_p)
\Bigr\}},
\end{equation}
uniquely characterized as follows.
Both
sets are in natural bijection with the set of isomorphism classes of pairs $(R,R')$ with
$R\subset R'$, where $R$ is an index-$p$ subring of the cubic ring
$R'$. The two bijections are given via $(R_f=R,R')\mapsto f$ and
$(R,R'=R_g)\mapsto g$. Moreover, we have $|\Stab(f)|=|\Stab(g)|$.
\end{lemma}
\index{$f\leftrightarrow (g,\alpha)$ switch, $R_f$ is  an index-$p$ subring of $R_g$}
\begin{proof}
The set $\overline{\W_q}$ is in bijection with the set of cubic rings
that are nonmaximal at every prime $p$ dividing $q$. As in
\cite[\S9]{BST}, we consider the set of pairs of cubic rings $R\subset
R'$, such that $R$ is nonmaximal at every prime dividing $q$, and the
index of $R$ in $R'$ is $p$.  Let $f$ and $g$ be representatives for
the $\GL_2(\Z)$-orbits on $V(\Z)$ corresponding to $R$ and $R'$,
respectively. If $f\in W_q$ is not a multiple of $p$, then there
exists a unique index-$p$ overring $R'$ of $R$ by Proposition
\ref{subsupring}. On the other hand, if $f$ is a multiple of $p$,
then the set of index-$p$ overrings $R'$ of $R$ are in natural
bijection with the roots of $f/p$ in $\P^1(\F_p)$ (also by Proposition
\ref{subsupring}). Therefore, the set of pairs $(R,R')$ is in natural
bijection with $\GL_2(\Z)$-orbits on the following set:
\begin{equation}\label{eqs1}
  \Bigl\{f\in \W_q\backslash p \W_{q/p}) \Bigr\}
\bigcup
  \Bigl\{(f,\gamma):f\in p\W_{q/p}^{(\gamma)},\;\gamma\in\P^1(\F_p)
  \Bigr\},
\end{equation}
and every form  \(f\) in the above set corresponds to the ring $R=R_f$. 

On the
other hand, the set of index-$p$ subrings of the ring $R_g$ is in
natural bijection with the set of roots of $g$ in $\P^1(\F_p)$ by
Proposition \ref{subsupring}.  Therefore, the set of pairs $(R,R')$ is
also in natural bijection with $\GL_2(\Z)$-orbits on the set
\begin{equation}\label{eqs2}
\Bigl\{(g,\alpha):\alpha\in\P^1(\F_p),\;g\in \W_{q/p}^{(\alpha)}\Bigr\},
\end{equation}
and every form  \(g\) in the above set corresponds to the ring $R'=R_g$.  It
follows that $\GL_2(\Z)$-orbits on the sets \eqref{eqs1} and
\eqref{eqs2} are in natural bijection.
\end{proof}

We will also need the following lemma determining how the above
bijection changes the splitting types of the binary cubic forms.
\begin{lemma}\label{l_sigmap_f}
Let $g\in \W_{q/p}$ and $\alpha\in\P^1(\F_p)$ be a root of $g$ modulo
$p$.  Let $f\in \overline{\W_q}$ correspond to the $\GL_2(\Z)$-orbit
of $(g,\alpha)$ under the bijection of Lemma \ref{lembij}. Then
\begin{equation*}
  \sigma_p(f)=\left\{\begin{array}{ccl}
  (1^21) && \mbox{if }\alpha \mbox{ is a simple root of }$g$;\\[.1in]
  (1^3)\mbox{ or } (0) && \mbox{otherwise.}
  \end{array}\right.
\end{equation*}
Moreover, for every prime $\ell\neq p$, we have
$\sigma_\ell(f)=\sigma_\ell(g)$.  And more generally, for every
integer $n$ coprime with $p$, the reduction of $f$ modulo $n$ and the
reduction of $g$ modulo $n$ are $\GL_2(\Z/n\Z)$-conjugates.
\end{lemma}
\begin{proof}
  By translating $g$ with an element of $\GL_2(\Z)$ if necessary, we
  can assume that $\alpha=[0:1]$. In that case, we have
  $g(x,y)=ax^3+bx^2y+cxy^2+dy^3$, where $p\mid d$. Furthermore, we
  have $p\nmid c$ if and only if $\alpha$ is a simple root. Then, the
  element $f(x,y)$ is given by $f(x,y)=p^2ax^3+pbx^2y+cxy^2+d/py^3$,
  and has splitting type $(1^21)$ if and only if $p\nmid c$. The first part of the lemma
  follows. 

To prove the second part of the lemma, note that by tensoring the exact sequence $0\to R_f\to R_g\to\Z/p\Z\to 0$ by the flat $\Z$-module $\Z_\ell$, we obtain the isomorphism $R_f\otimes\Z_\ell\cong R_g\otimes\Z_\ell$. Reducing modulo $\ell$ yields $R_f\otimes\F_\ell\cong
  R_g\otimes\F_\ell$ (and also $R_f \otimes \Z/n\Z \cong R_g \otimes \Z/n\Z$) which implies the desired conclusion.
\end{proof}

Let $n$ be a positive integer, and let $\phi:V(\Z/n\Z)\to\C$ be a
$\GL_2(\Z/n\Z)$-invariant function such that $\phi$ is given by
\begin{equation*}
\phi=\prod_{p^\beta\| n}\phi_{p^\beta},
\end{equation*}
where $f\mapsto \phi_{p^\beta}(f)$ is $\GL_2(\Z/p^\beta\Z)$-invariant. When
$\beta=1$, we have that $\phi_p(f)$ is determined by the splitting type of $f$ at
$p$. For any positive integer $k$ dividing $n$, such that $(k,n/k)=1$,
we denote $\prod_{p^\beta\|k}\phi_{p^\beta}$ by $\phi_k$.  Let $d\geq 1$
be a squarefree integer dividing $n$ such that $(d,n/d)=1$. 
\begin{definition}
We say
that such a function $\phi_n$ is {\it simple} at $d$, if for all
$p\mid d$, we have $\phi_p(f)= \phi_p(0)$ when the splitting type of $f$ at $p$
is $(1^3)$. 
\end{definition}
\index{$\phi_p(1^3)=\phi_p(0)$, simple congruence function at $p$}
\noindent Note that the functions of interest in the rest of the paper, namely $\lambda_{p^k}$ and $\theta_{p^k}$ for primes $p$ and positive integers $k$, are all simple.

We are now ready to prove the main result of this
subsection.

\begin{theorem}\label{thswitch}
Let $\Psi:\R_{>
  0}\to \C$ be a compactly supported function, $n$ be a positive integer and $q$ be a 
  positive squarefree
integer. Let $\phi$ be a $\GL_2(\Z/n\Z)$-invariant function on
$V(\Z/n\Z)$. Denote $(q,n)$ by $de$, where $d$ is the product of
primes dividing $(q,n)$ at which $\phi$ is simple, and assume that
$\phi_p(0)=0$ for every prime $p|d$. Write $n=dm$ and $\phi=\phi_d\phi_m$. Then
\begin{equation*}
  \displaystyle\sum_{f\in\overline{\W_q^\pm}}
  \frac{\phi(f)}{{|\Stab(f)|}}\Psi\bigl(|\Delta(f)|\bigr)=
  \phi_d(1^21)\displaystyle\sum_{k\ell \mid
    \frac{q}{de}}\mu(\ell)\sum_{g\in\overline{\W_e^\pm}}
\omega_d^{(1)}(g)\omega_{k\ell}(g)
  \frac{\phi_m(g)}{|\Stab(g)|}
  \Psi\Bigl(\frac{q^4|\Delta(g)|}{e^4d^2k^2}\Bigr).
\end{equation*}
\end{theorem}

\begin{proof}
We prove Theorem \ref{thswitch} by induction on the number of prime
factors of $d$. First we consider the case $d=1$ which we establish by
induction on the number of prime factors of $q/e$. Let $p$ be a prime
dividing $q/e$.  We again use the bijection of Lemma~\ref{lembij} to
relate the sum over $f\in\overline{\W_q}$ to a sum over
$g\in\overline{\W_{q/p}}$.  If $f\in \overline{p\W_{q/p}}$, then
$\phi(f/p)=\phi(f)$ because $\phi$ is $\GL_2(\Z/n\Z)$-invariant and
$(p,n)=1$ implies $1/p\in Z(\GL_2(\Z/n\Z))$ which acts by scalar
multiplication on $V(\Z/n\Z)$.  Suppose $f\in\overline{\W_q}\backslash
\overline{p\W_{q/p}}$ corresponds to the $\GL_2(\Z)$-orbit of $g\in
\W_{q/p}$ and a root $\alpha\in \P^1(\F_p)$ of $g$ modulo $p$ under
the surjection of Lemma \ref{lembij}.  Then since $(n,p)=1$, we have
$\phi(f)=\phi(g)$ from Lemma \ref{l_sigmap_f}.  Thus,
\begin{align*}
 \sum_{f\in\overline{\W_q^\pm}}
\frac{\phi(f)}{|\Stab(f)|}
\Psi\bigl(|\Delta(f)|\bigr)
&=
\sum_{k_1\ell_1 \mid p}\mu(\ell_1)\sum_{f\in\overline{\W_{q/p}^\pm}}
\omega_{k_1\ell_1}(f)
\frac{\phi(f)}{|\Stab(f)|}
\Psi\Bigl(\frac{p^4|\Delta(f)|}{k_1^2}\Bigr)\\
&=
\sum_{k_1\ell_1 \mid p}\mu(\ell_1)
\sum_{k_2\ell_2 \mid \frac{q}{pe}}\mu(\ell_2)
\sum_{f\in\overline{\W_{e}^\pm}}
\omega_{k_1\ell_1}(f)
\omega_{k_2\ell_2}(f)
\frac{\phi(f)}{|\Stab(f)|}
\Psi\Bigl(\frac{q^4|\Delta(f)|}{e^4k_1^2k_2^2}\Bigr)\\
&=
\sum_{k\ell \mid \frac{q}{e}}\mu(\ell)
\sum_{f\in\overline{\W_{e}^\pm}}
\omega_{k\ell}(f)
\frac{\phi(f)}{|\Stab(f)|}
\Psi\Bigl(\frac{q^4|\Delta(f)|}{e^4k^2}\Bigr),
\end{align*}
where the second equality follows by induction on the sum over
$\overline{\W_{q/p}^\pm}$ of the $\GL_2(\Z/pn\Z)$-invariant function
$\omega_{k_1\ell_1}\cdot \phi$.

It remains to prove the inductive step on the number of prime factors
of $d$. Let $p$ be a prime dividing $d$. We use the bijection of
Lemma~\ref{lembij} to relate the sum over $f\in\overline{\W_q}$ to sums
over $f\in\overline{\W_{q/p}}$. Suppose $f\in\overline{\W_q}$
corresponds to $g\in\overline{\W_{q/p}^{(\alpha)}}$ under the bijection
of Lemma \ref{lembij}, then by Lemma \ref{l_sigmap_f}, we have
$\phi_p(g)=\phi_p(1^21)$ if $\alpha$ is a simple root and
$\phi_p(g)=\phi_p(1^3)=0$ otherwise. Also, we have
$\phi_{n/p}(g)=\phi_{n/p}(f)$. Therefore, we have
\begin{equation*}
\begin{array}{rcl}
\displaystyle\sum_{f\in\overline{\W_q^\pm}}
\frac{\phi(f)}{|\Stab(f)|}
\Psi\bigl(|\Delta(f)|\bigr)&=&
\displaystyle\sum_{g\in\overline{\W^\pm_{q/p}}}
\omega^{(1)}_p(g)
\phi_p(1^21)
\frac{\phi_{n/p}(g)}{|\Stab(g)|}
\Psi\Bigl(\frac{|\Delta(g)|}{p^2}\Bigr)
\\[.2in]&=&\displaystyle
\phi_d(1^21)
\sum_{g\in\overline{\W^\pm_{q/d}}}
\omega^{(1)}_p(g)\omega^{(1)}_{d/p}(g)
\frac{\phi_{n/d}(g)}{|\Stab(g)|}
\Psi\Bigl(\frac{|\Delta(g)|}{d^2}\Bigr),
\end{array}
\end{equation*}
where the second equation follows by induction on the sum over
$\overline{\W^\pm_{q/p}}$ of the (simple at $d/p$) function
$\phi_{n/p}\cdot\omega^{(1)}_p$. The result now follows since
$\omega^{(1)}_k$ is multiplicative in $k$.
\end{proof}

\subsection{Summing congruence functions over 
$\overline{\W^\pm_q}$}\label{s_sieve_max}

Let $n$ be a positive integer and let $\phi:V(\Z/n\Z)\to \C$ be of the
form $\phi=\prod_{p^\beta\parallel n}\phi_{p^\beta}$, where $\phi_{p^\beta}:V(\Z/p^\beta \Z)\to \C$ and
$\beta:=v_p(n)$. 
\index{$\cA^{(q)}_n,\cC^{(q)}_n$, residue functionals with nonmaximality condition at $q$}
Let $V(\Z_p)^{\nm}$ be the subset of $V(\Z_p)$ of nonmaximal cubic forms. It is the closure 
of $\W_p^\pm$ inside $V(\Z_p)$.
\index{$V(\Z_p)^{\nm}$, subset of $V(\Z_p)$ of nonmaximal cubic forms}
Given a positive squarefree integer $q$, we define
\[
 \begin{array}{rcl}
\displaystyle\cA^{(q)}_{n}(\phi)
&:=&\displaystyle
\prod_{p|q}
\int_{ V(\Z_p)^\nm}\phi_{p^\beta}(f)df
\cdot 
\prod_{\substack{p\mid n\\ p\nmid q}}
\cA_{p^\beta}(\phi_{p^\beta}),
\\[.2in]
\cC^{(q)}_{n}(\phi)
&:=&\displaystyle
\prod_{p|q}
\int_{ V(\Z_p)^\nm}\phi_{p^\beta}(f)c_p(f)df
\cdot 
\prod_{\substack{p\mid n\\ p\nmid q}}
\cC_{p^\beta}(\phi_{p^\beta}),
\end{array}
 \]
where $df$ denotes the probability Haar measure on $V(\Z_p)$, and the
values of $c_p(f)$ are given in Table \ref{tabbc}. When $p|q$ but $p\nmid n$, we assume by 
convention that $\phi_{p^\beta}=1$ in the integral above.
Note that if $q=1$ then $\cA_n^{(q)}=\cA_n$ and
$\cC_n^{(q)}=\cC_n$, and more generally if $(n,q)=1$, then $\cA_n^{(q)}$ is  equal to $\cA_n$ times the probability that $f$ is nonmaximal at every prime dividing $q$ (with something similar holding for $\cC_n^{(q)}$).

\begin{theorem}\label{propWqerror}
Let $\Psi:\R_{> 0}\to \C$ be a smooth function
with compact support, let $n$ be a positive
integer, let $q$ be a positive squarefree integer, and let $d:=(q,n)$. Let $\phi = \prod_{p^\beta \parallel 
n} 
\phi_{p^\beta}$ be
a $\GL_2(\Z/n\Z)$-invariant function on $V(\Z/n\Z)$ such that $\phi$ is simple at $d$ and 
$\phi_p(0)=0$ for every
prime $p\mid d$.
Finally assume that there exists a prime $p$ dividing $n/d$ such that
$\phi_{p^\beta}$ is supported on elements $f\in V(\Z/p^\beta\Z)$ with splitting type
$\sigma_p(f)=(3)$. Then for every $X\ge 1$,
\begin{equation*}
\begin{array}{rcl}
  \displaystyle\sum_{f\in \overline{\W_q^\pm}} 
\frac{\phi(f)}{|\Stab(f)|}
\Psi
  \Bigl(\frac{|\Delta(f)|}{X}\Bigr)
  &=&
  \displaystyle\alpha^\pm \cdot \cA^{(q)}_n(\phi)\cdot\widetilde\Psi(1)\cdot X 
+\gamma^\pm \cdot \cC^{(q)}_n(\phi)
\cdot \widetilde \Psi(\frac 56) \cdot
X^{5/6}
\\[.2in]&&\displaystyle +
O_{\epsilon}\Bigl( 
d^{1+\epsilon}q^{1+\epsilon} \cdot \bigl(\frac{n}{d}\bigr)^{4+\epsilon} E_{n/d}
\bigl(\widehat{\phi_{n/d}}\bigr)
\cdot E_\infty(\widetilde{\Psi},\epsilon)
\Bigr).
\end{array}
\end{equation*}
\end{theorem}
\begin{proof}

The values of the constants in front of the primary and secondary main
terms follow from Theorem \ref{countphi}.  The term $\mathcal
B_n(\phi)$ vanishes because there exists a prime $p$ dividing $n$ such
that $\phi_{p^\beta}$ is supported on elements in $V(\Z/p^\beta\Z)$ with splitting
type $(3)$, which implies $\mathcal B_{p^\beta}(\phi_{p^\beta})=0$ because $\phi_{p^\beta}
\cdot b_p$ vanishes on $V(\Z/p^\beta\Z)$ in view of Table~\ref{tabbc}.

It remains to justify the size of the error term. For this, we first
use Theorem \ref{thswitch} to write
\begin{equation*}
\sum_{f\in \overline{\W_q^\pm}}
\frac{\phi_n(f)}{|\Stab(f)|}
\Psi\Bigl(\frac{|\Delta(f)|}{X}\Bigr)
=
\phi_d(1^21)\sum_{k\ell\mid\frac{q}{d}}\mu(\ell)\sum_{g\in\overline{V(\Z)^\pm}}
\omega^{(1)}_{d}(g)\omega_{k\ell}(g)
\frac{\phi_{n/d}(g)}{|\Stab(g)|}
\Psi\Bigl(\frac{q^4|\Delta(g)|}{Xd^2k^2}\Bigr).
\end{equation*}
For each $k$ and $\ell$, we apply Theorem
\ref{countphi} to the inner sum, obtaining
\begin{equation*}
\begin{array}{rcl}
\displaystyle\sum_{g\in\overline{V(\Z)^\pm}}
\omega^{(1)}_{d}(g)\omega_{k\ell}(g)
\frac{\phi_{n/d}(g)}{|\Stab(g)|}
\Psi\Bigl(\frac{q^4|\Delta(g)|}{Xd^2k^2}\Bigr)
&=&\displaystyle
c^{(1)}_{k,\ell}X+c^{(2)}_{k,\ell}X^{5/6}
\\[.2in]&&\displaystyle
+O_\epsilon\Bigl(
(nkl)^{4+\epsilon}\cdot
E_{d}(\widehat{\omega^{(1)}_{d}})E_{k\ell}(\widehat{\omega_{k\ell}})
E_{n/d}(\widehat{\phi_{n/d}})E_\infty(\widetilde{\Psi},\epsilon)
\Bigr)\\[.2in]&=&\displaystyle
c^{(1)}_{k,\ell}X+c^{(2)}_{k,\ell}X^{5/6}
\\[.2in]&&\displaystyle
+O_\epsilon\Bigl(
d^{2+\epsilon}(k\ell)^{1+\epsilon}\cdot \Bigl(\frac{n}{d}\Bigr)^{4+\epsilon}
E_{n/d}(\widehat{\phi_{n/d}})E_\infty(\widetilde{\Psi},\epsilon)
\Bigr).
\end{array}
\end{equation*}
The second estimate above follows since we have the bounds
\begin{equation*}
E_{d}(\widehat{\omega^{(1)}_d})\ll\frac{1}{d^{2-\epsilon}},\qquad\qquad 
E_{k\ell}(\widehat{\omega_{k\ell}})\ll\frac{1}{k^{3-\epsilon}\ell^{3-\epsilon}},
\end{equation*}
where the bounds follow from Lemmas \ref{lemEbound} and \ref{lemE-lambda} since
$\omega_{k\ell}=\lambda_{k\ell}+1$.  Summing over $k\ell$ dividing
$q/d$, we obtain
\begin{equation*}
\sum_{k\ell\mid\frac{q}{d}}d^{2+\epsilon}(k\ell)^{1+\epsilon}\ll (dq)^{1+\epsilon},
\end{equation*}
which yields the result.
\end{proof}

Recall that for a finite collection $\Sigma$ of local specifications,
we defined a \emph{family of fields} $\FF_\Sigma$. The finite collection
$\Sigma$ can also be used to cut out subsets of binary cubic
forms. Namely, for a set $S$ of integral binary cubic forms, let
$S(\Sigma)$ denote the subset of elements $f\in S$ such that
$R_f\otimes\Q_v\in\Sigma_v$ for each place $v$ associated to
$\Sigma$. Here, as usual, $R_f$ denotes the cubic ring associated to
$f$. Henceforth, we will always assume that $\Sigma_\infty$ is a
singleton set. That is, it is either $\R\oplus\R\oplus\R$,
corresponding to cubic fields and forms with positive discriminant, or
it is $\R\oplus\C$, corresponding to cubic fields and forms with
negative discriminant. Accordingly the sign $\pm$ in $\alpha^\pm$, $\gamma^\pm$, $V(\Z)^\pm$, 
$\W_q^\pm$, and so 
on, with be 
$+$ for the former case and $-$ for the latter case.
\index{$\pm$, $+$ is for totally real fields and $-$ is for complex fields}

Let $\chi_\Sigma$ be the characteristic
function of the set of elements
$(f_p) \in \prod_{p}V(\Z_p)$
such that
$R_{f_p}\otimes\Q_{p}\in \Sigma_{p}$ for every prime $p$.
\index{$\chi_\Sigma$, characteristic function of forms with specification $\Sigma$}
We have that $\chi_\Sigma$ factors through the quotient $
\prod_{p}V(\Z_p) \twoheadrightarrow V(\Z/r_\Sigma \Z)$ to a 
$\GL_2(\Z/r_\Sigma \Z)$-invariant function which we also denote by the same letter 
$\chi_\Sigma$. Here $r_\Sigma$ is a positive integer which is the product of $p$ over all primes 
$p\neq 2,3$ such that $\Sigma_p$ is specified at $p$ and of $16$ (resp. $27$) for the prime $2$ 
(resp. $3$). 
\index{$r_\Sigma$, product of primes $p$ such that $\Sigma_p$ is specified at $p$}

\begin{corollary}\label{lemlamer}
Let $\Sigma$ be a finite collection of local 
specifications and assume
that $\Sigma_p=\{\Q_{p^3}\}$ for at least one prime $p$. For every
positive integer $n$, positive squarefree integer $q$ and $X\ge 1$, we
have

\begin{equation}\label{eqlamer}
\begin{array}{rcl}
\displaystyle\sum_{f\in\overline{\W_q(\Sigma)}} \frac{\lambda_n(f)}{|\Stab(f)|} \Psi\Bigl(\frac{|\Delta(f)|}{X}\Bigr)
&=&\displaystyle
\alpha^\pm
\cA^{(q)}_{[n,r_\Sigma]}
(\lambda_n\chi_\Sigma)\widetilde\Psi(1)\cdot X
+ \gamma^\pm
\cC^{(q)}_{[n,r_\Sigma]}
(\lambda_n\chi_\Sigma) 
\cdot \widetilde \Psi(\frac 56) \cdot
X^{5/6}
\\[.2in]&&\displaystyle
+O_{\epsilon,\Sigma}\bigl((nq)^{1+\epsilon}
\cdot
E_\infty(\widetilde{\Psi},\epsilon)
\bigr).
\end{array}
\end{equation}
\end{corollary}

\begin{proof}
The two main terms of the asymptotic follow from Theorem
\ref{propWqerror}, and it is only necessary to analyze the size of the
error term. We write $n=n_1n_2$, where $n_1$ is squarefree, $n_2$
is powerful and $(n_1,n_2)=1$. Let $m$ denote the radical of $n_2$.
Recall that $\lambda_n$ is defined modulo $n_1m$, the radical of $n$. (Indeed, $\lambda_{p^k}$ only depends on the reduction of $f$ modulo $p$.) Thus, Theorem
\ref{propWqerror} yields an error term of
\begin{equation*}
O_{\epsilon,\Sigma}\Bigl(\frac{(n_1m)^{4+\epsilon}q^{1+\epsilon}}{(n,q)^3}\cdot
E_{\frac{n}{(n,q)}}\bigl(\widehat {\lambda_{\frac{n}{(n,q)}}}\bigr) E_\infty(\widetilde \Psi,\epsilon) \Bigr).
\end{equation*}
For a prime $p$ and and integer $k\geq 2$, it follows from Lemma
\ref{lemE-lambda} that we have
\begin{equation*}
E_p(\widehat{\lambda_p})\ll\frac{1}{p^3};\hspace{.3in}
E_{p^k}(\widehat{\lambda_{p^k}})\ll\frac{k}{p^2}.
\end{equation*}
Therefore, we obtain
\begin{equation*}
E_{\frac{n}{(n,q)}}\bigl(\widehat{\lambda_\frac{n}{(n,q)}}\bigr)\ll_\epsilon \frac{(n,q)^3n^\epsilon}{n_1^3m^{2}}.
\end{equation*}
The theorem now follows since $n_1m^2\leq n$.
\end{proof}

We end with two results.  The first is a uniform estimate, proved in
\cite{DH}, on the number of elements in $\overline{\W_q}$ with bounded
discriminant. This estimate will be used to bound the tail of the sum
in the right-hand side of \eqref{eqinex}.
\begin{proposition}[Davenport~\cite{DH}]\label{propunif}
For every $\epsilon>0$, $X\ge 1$, and squarefree integer $q$,
\begin{equation*}
  \#\bigl\{f\in\overline{\W^\pm_q}:|\Delta(f)|<X\bigr\}
  \ll_\epsilon\frac{X}{q^{2-\epsilon}}.
\end{equation*}
The multiplicative constant is independent of $q$ and $X$ (it depends
only on $\epsilon$).
\end{proposition}

\begin{proof}
With the notation we have set up, Davenport's proof can be expressed
as follows: We use Proposition~\ref{prop:sqitchclassic} with $\Psi$
the indicator function of the interval $[\frac12,X]$. Then, instead of
applying Theorem~\ref{countphi} as above, we apply the more direct
upper bound $\omega_{k\ell}(f)\ll q^{\epsilon}$ and estimate from
above the sum over $f\in \overline{V(\Z)^{\pm}}$ by $\frac{X
  k^2}{q^4}$.
\end{proof}

Second, we add up the functionals of Theorem \ref{propWqerror} over
squarefree numbers $q$. Let $\phi:V(\Z/n\Z)\to\C$ be a function of the
form $\phi=\prod_{p^\beta \parallel n} \phi_{p^\beta}$, where $\phi_{p^\beta}:V(\Z/p^\beta \Z)\to \C$ and
$\beta=v_p(n)$. For every prime $p\nmid n$, we define
$\phi_{p^\beta}:V(\Z_p)\to\C$ to simply be the constant $1$ function. We now
define the functionals
\begin{equation*}
\begin{array}{rcl}
\displaystyle\cA^\max(\phi)&:=&
\displaystyle\prod_{p}\int_{f\in V(\Z_p)^\max}\phi_{p^\beta}(f)df;\\[.1in]
\displaystyle\cC^\max(\phi)&:=&
\displaystyle\prod_{p}\int_{f\in V(\Z_p)^\max}c_p(f)\phi_{p^\beta}(f)df,
\end{array}
\end{equation*}
where the values of $c_p(f)$ are given in Table \ref{tabbc}.  By
multilinearity, the domain of definition of the functionals
$\cA^{\max}$ and $\cC^\max$ extends to all functions
$\phi:V(\Z/n\Z)\to \C$.

\begin{lemma}\label{lemmtcompmax}
For every integer $n$, the following identity between functionals
defined on functions from $V(\Z/n\Z)$ holds:
\begin{equation*}
\begin{array}{rcl}
\displaystyle\sum_{q\ge 1}\mu(q)\cA^{(q)}_n&=&\displaystyle\cA^\max;\\[.1in]
\displaystyle\sum_{q\ge 1}\mu(q)\cC^{(q)}_n&=&\displaystyle\cC^\max.
\end{array}
\end{equation*}
\end{lemma}
\index{$\cA^{\max},\cC^\max$, residue functionals with maximality condition}

\begin{proof}
This follows from the partition
\[
 V(\Z_p) = V(\Z_p)^\max \sqcup V(\Z_p)^\nm
\]
for every prime $p$, and the inclusion-exclusion principle.
\end{proof}

\subsection{Application to smooth counts of cubic 
fields with prescribed local specifications}\label{s_switch_applications}

In this subsection, we use \eqref{eqinex}, Theorem \ref{propWqerror}, Proposition~\ref{propunif} and Lemma \ref{lemmtcompmax} to sum congruence
functions over the space of cubic fields.
We denote the set of all cubic fields $K$ with $\pm\Delta(K)>0$ by
$\FF^\pm$. We say that $\theta:\FF^\pm\to\C$ is a {\it simple function
  defined modulo $n$} if there exists a simple
$\GL_2(\Z/n\Z)$-invariant function $\phi: V(\Z/n\Z)\to \C$ such that
for every cubic field $K$, whose ring of integers corresponds to a
maximal binary cubic form $f$, we have $\theta(K)=\phi(\bar{f})$, where
$\bar{f}$ denotes the reduction of $f$ modulo $n$. For example
$\lambda_K(n)$ is a simple function defined modulo $n$ corresponding
to the function $\lambda_n(f)$.

\begin{theorem}\label{thfieldscount}
Let $\Psi:\R_{> 0}\to \C$ be a smooth function
with compact support such that $\int \Psi=1$. Let
$\Sigma$ be a finite set of local
specifications, such that $\Sigma_p=\{\Q_{p^3}\}$ for at least one
prime $p$.  For every real $X\ge 1$ and integer $n\ge 1$,
\begin{equation*}
\sum_{K\in\FF_\Sigma}\lambda_K(n)\Psi\Bigl(\frac{|\Delta(K)|}{X}\Bigr)
=  \alpha^\pm \cA^\max(\lambda_n \chi_\Sigma) X+
\gamma^\pm \cC^\max(\lambda_n \chi_\Sigma)
\widetilde\Psi\bigl(\frac56\bigr)
X^{5/6}+
O_{\epsilon,\Sigma,\Psi}\bigl(X^{2/3+\epsilon}n^{1/3}\bigr).
\end{equation*}
\end{theorem}

\noindent Before we turn to the proof of Theorem \ref{thfieldscount},
we make the following two observations. First, the quadratic analogue
of the above result is the question of summing the Legendre symbol
$\bigl(\frac{\cdot}{n}\bigr)$ over the set of fundamental discriminants (or squarefree
integers). 

Second, the case $n=1$ of the above result (with the simplifying assumption that 
$\Sigma_p=\{\Q_{p^3}\}$ for at least one prime $p$) is a smoothed version (instead of a sharp 
version counting $K\in \FF_\Sigma(X)$ without the $\Psi$-smoothing) of the
refined Roberts' conjecture, proved independently in \cite{BST} and
\cite{TaTh1}. Those works obtain the error terms
$O_\epsilon(X^{5/6-1/48+\epsilon})$ and
$O_\epsilon(X^{7/9+\epsilon})$ for the sharp version of the refined Roberts'
conjecture. Independently from the present article, the recent work
\cite{BTT} obtains an improved error term of $O_\epsilon(X^{2/3+\epsilon})$ for the sharp count.
This seems to indicate that $X^{2/3+\epsilon}$ is the natural exponent both for our present 
purposes of smoothly summing the Artin character of cubic fields and for the problem of 
sharp 
counting of cubic fields.

\begin{proof}[Proof of Theorem \ref{thfieldscount}]
We start with the sieve \eqref{eqinex} to write
\begin{equation*}
  \sum_{K\in\FF_\Sigma}\lambda_K(n)\Psi\Bigl(\frac{|\Delta(K)|}{X}\Bigr)=
  \sum_{q\geq 1}\mu(q)\sum_{f\in\overline{\W_q^\pm}}
\frac{\lambda_n(f)}{|\Stab(f)|}
\chi_\Sigma(f)\Psi\Bigl(\frac{|\Delta(f)|}{X}\Bigr) + O_\epsilon\bigl(X^{\frac12}n^\epsilon\bigr).
\end{equation*}
Note that the sum over $K$ is \emph{not} weighted by the size of the automorphism group. 
On the right-hand side, the difference is accounted by the number of cyclic cubic fields 
which is $O(X^{\frac12})$.

Pick a real number $Q$ to be optimized later. Using Corollary \ref{lemlamer} for $q\leq Q$, 
Proposition \ref{propunif} for $q>Q$, and Lemma
\ref{lemmtcompmax} to evaluate the main terms, we obtain
\begin{equation*}
  \sum_{K\in\FF_\Sigma}\lambda_K(n)\Psi\Bigl(\frac{|\Delta(K)|}{X}\Bigr)=
\alpha^\pm\cA^\max(\lambda_n\chi_\Sigma)X+
\gamma^\pm\cC^\max(\lambda_n\chi_\Sigma)
\widetilde\Psi\bigl(\frac56\bigr)X^{5/6}+
O_{\epsilon,\Sigma,\Psi}\bigl((nQ^2)^{1+\epsilon}\bigr)+O_{\epsilon,\Psi}\Bigl(\frac{X}{Q^{1-\epsilon}}\Bigr).
\end{equation*}
Optimizing, we pick $Q=(X/n)^{1/3}$ which yields Theorem
\ref{thfieldscount}.
\end{proof}

Finally, we have a result estimating smoothed sums over cubic fields,
where we sum over arbitrary congruence functions defined modulo a
squarefree integer. (We could allow more general specifications, but this situation seems 
to be the most common in applications).
\begin{theorem}\label{thKcountsimple}
Let $\Psi:\R_{> 0}\to \C$ be a smooth function with compact support
such that $\int \Psi=1$. Let $n$ be a positive squarefree integer, and
let $\theta$ be a simple function on the family $\FF^+$
(resp. $\FF^-$) of totally real cubic fields (resp. complex cubic
fields) corresponding to a $\GL_2(\Z/n\Z)$-invariant congruence
function $\phi:V(\Z/n\Z)\to\C$ which is simple at $n$ and such that
$\phi_p(0)=0$ for every prime $p|n$.  Namely
$\theta(K_f)=\phi(\overline f)$ for every $f\in
V(\Z)^{\pm,\irr,\max}$. Assume that for at least one prime $p|n$,
$\theta(K)\neq 0$ forces $K$ to be inert at $p$.  Then we have
\begin{equation*}
\sum_{K\in\FF^\pm}\theta(K)\Psi\Bigl(\frac{|\Delta(K)|}{X}\Bigr)
=\alpha^\pm\cA^\max(\phi)X+
\gamma^\pm\cC^\max(\phi)\widetilde\Psi\bigl(\frac56\bigr)
X^{5/6}+O_{\epsilon,\Psi}\bigl(X^{2/3+\epsilon}n^{2/3+\epsilon}||\theta||_\infty\bigr).
\end{equation*}
\end{theorem}
\begin{proof}
As before, we begin with the inclusion-exclusion sieve. Pick $Q>1$ to
be optimized and write
\begin{equation*}
\sum_{K\in\FF^\pm}\theta(K)\Psi\Bigl(\frac{|\Delta(K)|}{X}\Bigr)=
\sum_{q\leq Q}\mu(q) \sum_{f\in\overline{\W_q^\pm}}\frac{\phi(f)}{|\Stab(f)|}\Psi\Bigl(\frac{|\Delta(K)|}{X}\Bigr)
+O_{\epsilon,\Psi}\Bigl(\frac{X}{Q^{1-\epsilon}}\Bigr)+O\bigl(X^{1/2}||\theta||_\infty\bigr).
\end{equation*}
For $q\leq Q$, we use Theorem \ref{propWqerror} to write
\begin{equation*}
\sum_{f\in\overline{\W_q^\pm}}\frac{\phi(f)}{|\Stab(f)|}\Psi\Bigl(\frac{|\Delta(K)|}{X}\Bigr)=
\alpha^\pm\cA^{(q)}_n(\phi)X+\gamma^\pm\cC^{(q)}_n(\phi)
\widetilde\Psi\bigl(\frac56\bigr)X^{5/6}
+O_{\epsilon,\Psi}\Bigl(\frac{n^{4+\epsilon}q^{1+\epsilon}}{(n,q)^3}\cdot
E_{\frac{n}{(n,q)}}\bigl(\widehat{\phi_{\frac{n}{(n,q)}}}\bigr)\Bigr).
\end{equation*}
From the definition of the error term $E_{\frac{n}{(n,q)}}$ and Corollary
\ref{corphihatbound}, we obtain the bound
\begin{equation*}
  E_{\frac{n}{(n,q)}}\bigl(\widehat{\phi_{\frac{n}{(n,q)}}}\bigr)\ll_\epsilon
  \frac{(n,q)^2}{n^{2-\epsilon}} ||\theta||_\infty.
\end{equation*}
Using Lemma \ref{lemmtcompmax} to evaluate the main term, we therefore
obtain
\begin{equation*}
\sum_{K\in\FF^\pm}\theta(K)\Psi\Bigl(\frac{|\Delta(K)|}{X}\Bigr)
=\alpha^\pm\cA^\max(\phi)X+
\gamma^\pm\cC^\max(\phi)
\widetilde\Psi\bigl(\frac56\bigr)X^{5/6}+O_{\epsilon,\Psi}\Bigl(\frac{X}{Q^{1-\epsilon}}\Bigr)
+O_{\epsilon,\Psi}\bigl(n^{2+\epsilon}Q^{2+\epsilon}||\theta||_\infty\bigr).
\end{equation*}
Optimizing, we pick $Q=X^{1/3}/n^{2/3}$, which yields the result.
\end{proof}

\section{Low-lying zeros of Dedekind zeta functions of cubic 
fields}\label{sec:low-lying}

We follow the setup of \cite[\S2.4]{SST1} and of the previous Section~\ref{sec:switch}.
For every function $\eta:\FF_\Sigma\to\C$, we define
\begin{equation*}
\CS_\Sigma(\eta,X):=\sum_{K\in \FF_\Sigma}\eta(K)
\Psi\Bigl(\frac{|\Delta(K)|}{X}\Bigr)
\end{equation*}
to be the smoothed average of $\eta(K)$ over fields $K$ in
$\FF_\Sigma$ with discriminant close to $X$. Note in particular that
$\CS_\Sigma(1,X)$ denotes a smooth count of elements in
$\FF_\Sigma$.

For a cubic field $K$, recall from Proposition~\ref{p:Hecke} that the function $L(s,\rho_K)$ is
known to be entire and that the Artin conductor of
$L(s,\rho_K)$ is equal to $|\Delta(K)|$.
We define the quantity $\L_X$ to be
the average value of $\log |\Delta(K)|$ over $K\in\FF_\Sigma(X)$, i.e., we
define
\begin{equation*}
\L_X:=\frac{\CS_\Sigma(\log |\Delta(K)|,X)}{\CS_\Sigma(1,X)}.
\end{equation*}
The Davenport--Heilbronn theorem implies that we have
\begin{equation*}
\L_X=\log X+O(1).
\end{equation*}

We write the \emph{nontrivial zeros} of $L(s,\rho_K)$ as
$1/2+i\gamma_K^{(j)}$, where the imaginary part of $\gamma_K^{(j)}$ is
bounded in absolute value by $1/2$.  We pick $\Phi:\R\to\C$ to be a
smooth and even function such that the Fourier transform
$\widehat{\Phi}:\R\to\C$ has compact support contained in the open interval
$(-a,a)$. It is then known that $\Phi$ can be extended to an
entire function of exponential type $a$.  Define $Z_K(X)$ by
\begin{equation*}
Z_K(X):=\sum_j \Phi\Bigl(\frac{\gamma^{(j)}_K\L_X}{2\pi}\Bigr).
\end{equation*}
We work with the following variant of the \emph{$1$-level density}
$\cD(\FF_\Sigma(X),\Phi)$ of the family of Artin $L$-functions $L(s,\rho_K)$ (equivalently, of
the family of Dedekind zeta functions $\zeta_K(s)$) of $K\in\FF_\Sigma$:
\begin{equation*}
\cD(\FF_\Sigma(X),\Phi):= 
\frac{\CS_\Sigma\bigl(Z_K(X),X\bigr)}{\CS_\Sigma(1,X)}.
\end{equation*}

Recall that $\theta_K(n)$ was defined in \eqref{def:Dirichlet} so that the $n$th Dirichlet
coefficient of the logarithmic derivative of $L(s,\rho_K)$ is
$\theta_K(n)\Lambda(n)$.  We use the explicit formula
\cite[Prop.2.1]{SST1} to evaluate $Z_K(X)$:
\begin{equation*}\label{eqexplicit}
        \sum_j \Phi\big(\gamma^{(j)}_K\big) 
=\frac1{2\pi}\int_{-\infty}^\infty \Phi(t)\bigl(\log
|\Delta(K)|+O(1)\bigr)dt -\frac{1}{\pi}\sum_{n=1}^\infty
\frac{\theta_K(n)\Lambda(n)}{n^{1/2}}
\widehat{\Phi}\Bigl(\frac{\log n}{2\pi}\Bigr).
\end{equation*}
It yields
$Z_K(X)=Z_K^{(1)}(X)+Z_K^{(2)}(X)$, where
\begin{equation*}
\begin{array}{rcl}
\displaystyle Z_K^{(1)}(X)&=&
\displaystyle\frac{1}{2\pi}\int_{-\infty}^\infty \Phi\Bigl(\frac{t\L_X}{2\pi}
\Bigr)\bigl(\log |\Delta(K)|+O(1)\bigr)dt;
\\[.2in] \displaystyle Z_K^{(2)}(X)&=&
\displaystyle-\frac{2}{\L_X}\sum_{n=1}^\infty
\frac{\theta_K(n)\Lambda(n)}{\sqrt{n}}
\widehat{\Phi}\Bigl(\frac{\log n}{\L_X}\Bigl).
\end{array}
\end{equation*}
A computation identical to \cite[Eq.(17)]{SST1} gives
\begin{equation}\label{eqDZ1}
\lim_{X\to\infty}
\frac{\CS_\Sigma\bigl(Z_K^{(1)}(X),X\bigr)}{\CS_\Sigma(1,X)}
=\widehat{\Phi}(0).
\end{equation}
To compute the $1$-level density, we need to compute the asymptotics
of $\CS_\Sigma(Z_K^{(2)}(X),X)$. We write

\begin{equation}\label{eqLLZZ2}
\begin{array}{rcl}
\displaystyle\CS_\Sigma(Z_K^{(2)}(X),X)&=&
\displaystyle-\frac{2}{\L_X}
\CS_\Sigma\Bigl(\sum_{n=1}^\infty\frac{\theta_K(n)\Lambda(n)}{\sqrt{n}}
\widehat{\Phi}\Bigl(\frac{\log n}{\L_X}\Bigr),X\Bigr)
\\[.2in]&=&
\displaystyle-\frac{2}{\L_X}
\sum_{p,m}\frac{\log p}{p^{m/2}}
\widehat{\Phi}\Bigl(\frac{m\log p}{\L_X}\Bigr)
\CS_\Sigma\bigl(\theta_K(p^m),X\bigr).
\end{array}
\end{equation}

We now have the following result estimating the ratios
$\CS_\Sigma\bigl(\theta_K(p^m),X\bigr)/\CS_\Sigma(1,X)$.

\begin{proposition}\label{propthetacount}
Let $p$ be a prime number, and let $X\ge 1$ be a real number. Then,
for integers $m\geq 3$, we have
\begin{equation}\label{eqllzb}
\begin{array}{rcl}
\displaystyle \frac{\CS_\Sigma(\theta_K(p),X)}{\CS_\Sigma(1,X)}&\ll_\epsilon&
\displaystyle \frac{1}{p}+\frac{1}{p^{1/3}X^{1/6}}+\frac{p^{1/3}}{X^{1/3-\epsilon}};
\\[.2in]
\displaystyle \frac{\CS_\Sigma(\theta_K(p^2),X)}{\CS_\Sigma(1,X)}-1&\ll_\epsilon&
\displaystyle \frac{1}{p^2}+\frac{1}{X^{1/6}}+
\frac{p^{2/3}}{X^{1/3-\epsilon}};
\\[.2in]  
\displaystyle \frac{\CS_\Sigma(\theta_K(p^m),X)}{\CS_\Sigma(1,X)}&\ll&
\displaystyle 1.
\end{array}
\end{equation}
\end{proposition}
\begin{proof}
From Lemma \ref{l_thetap2} we have that $\theta_K(p)=\lambda_K(p)$.
The left-hand side of the first line of \eqref{eqllzb} can be computed
from Theorem \ref{thfieldscount}, yielding
\begin{equation*}
  \frac{\CS_\Sigma(\theta_K(p),X)}{\CS_\Sigma(1,X)}\ll_\epsilon
  \cA^\max(\lambda_p \chi_\Sigma)+X^{-1/6}\cC^\max(\lambda_p \chi_\Sigma)
  +X^{-1/3+\epsilon}p^{1/3}.
\end{equation*}
Note that the first summand in the right-hand side is bounded by $O(\widehat{u_p\cdot\lambda_p}(0))$, where $u_p$ (defined in Section 3) is the characteristic function of the set of elements in $V(\Z/p^2\Z)$ that lift to binary cubic forms which are maximal at $p$. The required bound then follows from the first part of Proposition
\ref{propmaxden}. Similarly, the second term is bounded by $O(X^{-1/6}\widehat{u_p\cdot c_p\lambda_p}(0))$. We prove in Lemma \ref{lemAC-lambda} that $\widehat{c_p\lambda_p}(0)\ll p^{-1/3}$. The same bound holds for $\widehat{u_p\cdot c_p\lambda_p}(0)$ since $u_p$ differs from $1$ only at a density $1/p^2$ subset.

The proof of the second inequality is similar: this time, we use Theorem
\ref{thKcountsimple} to deduce the estimate
\begin{equation*}
  \frac{\CS_\Sigma(\theta_K(p^2),X)}{\CS_\Sigma(1,X)}-1\ll_\epsilon
  \cA^\max((\theta_{p^2}-1) \chi_\Sigma)+X^{-1/6}\cC^\max(\theta_{p^2} \chi_\Sigma)
  +X^{-1/3+\epsilon}p^{2/3}.
\end{equation*}
The third part of Proposition \ref{propmaxden} implies the required bound on the first summand on the right-hand side above, while the required bound on the second summand follows immediately since $\theta_{p^2}$ is absolutely bounded.
Finally, Lemma \ref{l_thetap2} states that $|\theta_K(p^m)|\leq 2$,
from which the third inequality follows immediately.
\end{proof}

We are now ready to prove the main result of this section.

\begin{proof}[Proof of Theorem \ref{thmllz}]
From \eqref{eqLLZZ2} and Proposition \ref{propthetacount}, we obtain
\begin{equation*}
\begin{array}{rcl}
\displaystyle-\frac{\CS_\Sigma(Z_K^{(2)}(X),X)}{\CS_\Sigma(1,X)}&=&
\displaystyle\frac{2}{\L_X}
\sum_{p}\frac{\log p}{p}
\widehat{\Phi}\Bigl(\frac{2\log p}{\L_X}\Bigr)
\frac{\CS_\Sigma(\theta_K(p^2),X)}{\CS_\Sigma(1,X)}+O\biggl(\frac{1}{\log X}
\sum_{\substack{p^m\ll X^a\\m\neq 2}}\frac{\log p}{p^{m/2}}
\frac{\CS_\Sigma(\theta_K(p^m),X)}{\CS_\Sigma(1,X)}\biggr)
\\[.25in]&=&
\displaystyle\frac{2}{\L_X}
\sum_{p}\frac{\log p}{p}\widehat{\Phi}\Bigl(\frac{2\log p}{\L_X}\Bigr)
+O_\epsilon\Bigl(\frac{1}{\log X}+X^{\frac{a-1}{6}}
+X^{\frac{5a-2}{6}+\epsilon}\Bigr)
\\[.25in]&&+
\displaystyle O_\epsilon\Bigl(\frac{1}{\log X}+X^{-\frac{1}{6}+\epsilon}
+X^{\frac{a-1}{3}+\epsilon}\Bigr)
+O\Bigl(\frac{1}{\log X}\Bigr),
\end{array}
\end{equation*}
where the three error terms respectively arise from the three
estimates of Proposition \ref{propthetacount}. Assuming that
$a<\frac{2}{5}$, and using the above computation in conjunction
with \eqref{eqDZ1}, gives
\begin{equation*}
\begin{array}{rcl}
\displaystyle
\lim_{X\to \infty} \cD(\FF_\Sigma(X),\Phi)&=&
\displaystyle\widehat{\Phi}(0)-\lim_{X\to \infty}
\displaystyle\frac{2}{\L_X}
\sum_{p}\frac{\log p}{p}\widehat{\Phi}\Bigl(\frac{2\log p}{\L_X}\Bigr)
\\[.2in]&=&
\displaystyle\widehat{\Phi}(0)-\frac12\int_{-1}^{1}\widehat{\Phi}(t)dt,
\end{array}
\end{equation*}
where the final equality follows from the prime number
theorem. This concludes the proof of Theorem~\ref{thmllz}.
\end{proof}

\section{Main term for the average central values}\label{sec:average}

Let $\Sigma=(\Sigma_v)$ be a finite set of local
specifications.  Without loss of generality we assume that
$\Sigma_\infty$ is a singleton set, which is to say that either the
cubic fields prescribed by $\Sigma_\infty=\{\R\times \R \times \R\}$
are all totally real, or the cubic fields prescribed by $\Sigma_\infty
= \{\R\times \C\}$ are all complex.  We also assume (by adding a prime
if necessary) that there exists a prime $p$
such that $\Sigma_{p}=\{\Q_{p^3}\}$. Let $\FF_\Sigma$ denote the
family of cubic fields $K$ prescribed by the set $\Sigma$ of specifications, 
namely such that for each place $v$ we have $K
\otimes_\Q \Q_v\in\Sigma_v$.

We let $V(\Z)(\Sigma)$ denote the set of elements $f\in V(\Z)$ such
that $\chi_\Sigma(f)=1$ and such that $\Delta(f)>0$ if
$\Sigma_\infty = \{\R\times \R\times \R\}$ (resp.  $\Delta(f)<0$ if
$\Sigma_\infty=\{\R\times \C\}$).  For each prime $p$,
let $\W_p(\Sigma)$ denote the set of elements in $V(\Z)(\Sigma)$ that
are nonmaximal at $p$. If $q$ is a squarefree positive integer, we set
$\W_q(\Sigma)=\cap_{p\mid q}\W_p(\Sigma)$. In particular
$\W_1(\Sigma)=V(\Z)(\Sigma)$.

Thanks to the condition $\Sigma_{p}=\{\Q_{p^3}\}$, we have that every form $f\in V(\Z)(\Sigma)$ 
is irreducible. This implies that the set  $\overline{V(\Z)(\Sigma)^{\max}}$ of $\GL_2(\Z)$-orbits 
parametrizes under the Delone--Faddeev correspondence 
the family 
$\FF_\Sigma$ of cubic fields prescribed by the finite set $\Sigma$ of specifications.

Let $\Psi:\R_{>0} \to \C$ be a smooth function of compact support with
$\int \Psi=1$. 

For the rest of this paper, we automatically assume that every sum of binary cubic
forms $f$ is weighted by $1/|\Stab(f)|$. For a real number $X\ge 1$,
the inclusion-exclusion principle in conjunction with
Proposition~\ref{propAFE} yields:
\begin{equation}\label{eqnvzsieve}
A_\Sigma(X):=
\sum_{K\in \FF_\Sigma}\frac{L(\tfrac 12,\rho_K)}{|\Aut(K)|}\Psi\Bigl(\frac{|\Delta(K)|}{X}\Bigr)=
2\sum_{q\geq 1}\mu(q)\sum_{f\in 
\overline{\W_q(\Sigma)}} S(f)
\Psi\Bigl(\frac{|\Delta(f)|}{X}\Bigr),
\end{equation}
\index{$A_\Sigma(X)$, smoothed first moment of $L(\tfrac12,\rho_K)$}
where $S(f)$ was defined in~\eqref{defSf} to be
\begin{equation}\label{eqL12f}
  S(f)=\sum_{n=1}^\infty\frac{\lambda_n(f)}{n^{1/2}}
  V^\pm\Bigl(\frac{n}{\sqrt{|\Delta(f)|}}\Bigr),
\end{equation}
with $V^\pm$ as in Proposition \ref{propAFE} and where the sign is $+$
if $\Sigma_\infty=\{\R\times \R\times \R\}$ and $-$ if $\Sigma_\infty
= \{\R\times \C\}$.
The identity holds because for a maximal irreducible binary
cubic form $f\in V(\Z)^{\irr,\max}$ corresponding to the ring of
integers of a cubic field $K_f$, we have $2S(f)=L(\tfrac12,\rho_{K_f})$ by
Corollary~\ref{c_Euler} and Proposition~\ref{propAFE}, and we also have $|\Aut(K_f)|=|\Stab(f)|$.

In this section, we will prove two results. First, we will prove an
upper bound on $A_\Sigma(X)$, which improves on the pointwise bound
coming from summing the best known upper bounds on $|L(\tfrac12,\rho_K)|$
over the associated fields $K$.  Second, assuming a sufficiently
strong upper bound on $|L(\tfrac12,\rho_K)|$, we obtain asymptotics for
$A_\Sigma(X)$.

\subsection{Asymptotics for the terms with $q<Q$}

For $Q\in \R_{\ge 1}$ to be chosen later, we split the right-hand side of 
\eqref{eqnvzsieve} into two parts,
\begin{equation*}
\sum_{q<Q} \quad \text{and} \quad  \sum_{q\ge Q}.
\end{equation*}
This section is concerned with the first part:
\begin{equation}\label{eqlongest}
\displaystyle 2
\sum_{q<Q}\mu(q)
\sum_{f\in \overline{\W_q(\Sigma)}}
\sum_{n=1}^\infty
\frac{\lambda_n(f)}{n^{1/2}}
\Psi\Bigl(\frac{|\Delta(f)|}{X}\Bigr)
V^\pm \Bigl(\frac{n}{\sqrt{|\Delta(f)|}}\Bigr).
\end{equation}

It will be convenient for us to set some notation surrounding the
smooth functions above and their Mellin transforms.  For any
positive real number $y\in \R_{> 0}$, let $\H_y:\R_{>0}\to \C$
denote the compactly supported function
\begin{equation}\label{eqimpfunc}
\H_y(t):=
\Psi(t) \cdot
V^\pm\Bigl(\frac{y}{\sqrt{t}}\Bigr).
\end{equation}
The relevance of $\H_y(t)$ is that we have the equality
\begin{equation*}
\Psi\Bigl(\frac{|\Delta(f)|}{X}\Bigr)
V^\pm\Bigl(\frac{n}{\sqrt{|\Delta(f)|}}\Bigr)
=\H_{\frac{n}{\sqrt{X}}}\Bigl(\frac{|\Delta(f)|}{X}\Bigr).
\end{equation*}
\index{$\H_y$, compactly supported function on $\R_{>0}$}

\begin{lemma}\label{bound-H}
$($i$)$ There exists a constant $C>0$ depending only on $\Psi$ 
  such that for every $\epsilon \in [-1,1]$ and $y\in \R_{>0}$,
\begin{equation*}
E_\infty(\widetilde{\cH_y};\epsilon)= \int^\infty_{-\infty}
|\widetilde{\mathcal H_y}(-\epsilon + ir)| (1+|r|)^{2+4\epsilon} dr \le C.
\end{equation*}

$($ii$)$ There exists a constant $C>0$ depending only on $\Psi$ such that for every $y\in 
\R_{>0}$, 
$|\widetilde{\mathcal H_y}(\frac 56)| \le C$.
\end{lemma}
\begin{proof}
We have by definition~\eqref{defV},
\[
 \widetilde{V^\pm}(s) = \frac{G(s)}{s} \frac{\gamma^\pm(\tfrac12+s)}{\gamma^\pm(\tfrac12)}.
\]
We deduce that the Mellin transform of $t\mapsto V^\pm(\frac{y}{\sqrt{t}})$ is equal to
\[
 2 y^{2s} \widetilde{V^\pm}(-2s) = -y^{2s} \frac{G(-2s)}{s} \frac{\gamma^\pm(\tfrac12-2s)}{\gamma^\pm(\tfrac12)}.
\]
Since $\mathcal H_y$ is the product of the two functions $\Psi$ and  $t\mapsto 
V^\pm(\frac{y}{\sqrt{t}})$, its Mellin transform is the convolution of the Mellin transforms of 
the respective functions:
\begin{equation}\label{Mellin-convolution}
 \widetilde{\mathcal H_y}(\sigma + ir) =
2 \int_{\Re(u)=\eta}
\widetilde \Psi(\sigma+ir + u)
y^{-2u} \widetilde{V^\pm}(2u) \frac{du}{2\pi i},
\end{equation}
where $0 <\eta < \tfrac 12$ is fixed.
Indeed, to establish~\eqref{Mellin-convolution} it suffices to compute the inverse Mellin transform of the right-hand side with a translation of the integration of the \(v\)-variable:
\begin{equation*}
\begin{aligned}
2 \int_{\Re(v)=0} t^{-v} \int_{\Re(u)=\eta}
\widetilde{\Psi}(v+u)y^{-2u}\widetilde{V^{\pm}}(2u) \frac{du}{2\pi i} \frac{dv}{2\pi i}
&=
\int_{\Re(v)=\eta} 
t^{-v}
\widetilde{\Psi}(v)
\frac{dv}{2\pi i}
\int_{\Re(u)=\eta}
2 t^u
y^{-2u}\widetilde{V^{\pm}}(2u) \frac{du}{2\pi i} \\
&= \Psi(t) V^{\pm}\Bigl(\frac{y}{\sqrt{t}}\Bigr) = \mathcal H_y(t),
\end{aligned}
\end{equation*}
which coincides with the inverse Mellin transform of the left-hand side.

We deduce from~\eqref{Mellin-convolution} the following inequality:
\[
 |\widetilde{\mathcal H_y}(\sigma + ir)| \le
\frac{y^{-2\eta}}{\pi}   \int^\infty_{-\infty}
|\widetilde \Psi(\sigma +ir +\eta  + i\tau )| \cdot |\widetilde{V^\pm}(2\eta + 2i\tau )| d\tau.
\]
We shall use this inequality for $y\in [1,+\infty)$, in which case $y^{-2\eta}\le 1$.

On the other hand, if we shift the contour of~\eqref{Mellin-convolution} to $\Re(u)=-\eta$, picking up a 
simple pole at $u=0$ of  \(\widetilde{V^{\pm}}(2u)\), we then obtain the following inequality:
\[  
 |\widetilde{\mathcal H_y}(\sigma + ir)| \le
\frac{y^{2\eta}}{\pi}   \int^\infty_{-\infty}
|\widetilde \Psi(\sigma +ir - \eta  + i\tau )| \cdot |\widetilde{V^\pm}(-2\eta + 2i\tau )| d\tau + |\widetilde \Psi(\sigma + 
ir)|\cdot |G(0)|.
\]
We shall use this inequality for the other interval $y\in (0,1]$, in which case $y^{2\eta}\le 
1$.

Assertion (ii) follows immediately by inserting $\sigma=\tfrac 56$ and $r=0$.
Assertion (i) follows by inserting $\sigma=-\epsilon$ and integrating over $r$ because 
$E_\infty(\widetilde{\cH_y};\epsilon)$ for $y\ge 1$ is bounded by
\[
 \frac{1}{\pi}   \int^\infty_{-\infty}
|\widetilde \Psi(-\epsilon+\eta +ir)| (1+|r|)^{2+4\epsilon} dr 
\cdot \int^\infty_{-\infty} |\widetilde{V^\pm}(2\eta + 2i\tau )| (1+|\tau|)^{2+4\epsilon} d\tau \le C,
\]
where $C$ depends only on $\Psi$. The estimate for $y\le 1$ is similar.
\end{proof}

We are now ready to prove the main result of this subsection.
\begin{proposition}\label{p_main-term}
For every $\epsilon>0$ and $Q,X\ge 1$, the sum \eqref{eqlongest} is
asymptotic to
\[
 C_\Sigma \cdot X \cdot \bigl(\log X + \widetilde \Psi'(1)\bigr) 
+
 C'_{\Sigma} \cdot X
+
O_{\epsilon,\Sigma,\Psi}\Bigl( \frac{X^{1+\epsilon}}{Q} + X^{11/12+\epsilon} +  Q^{2+\epsilon}X^{3/4+\epsilon}\Bigr),
\]
where $C_\Sigma>0$ and $C'_\Sigma\in \R$ depend only on the finite set
$\Sigma$ of local specifications.
\index{$C_\Sigma,C'_\Sigma$, main terms for the first moment}
\end{proposition}
\begin{proof}
Since $V^\pm$ is a function rapidly decaying at infinity, we may truncate the $n$-sum in
the definition of $S(f)$ to $n < X^{1/2+\epsilon}$ with negligible
error term.  To estimate \eqref{eqlongest} we switch order of
summation and consider
\[
\displaystyle 2\sum_{n< X^{1/2+\epsilon}}
\sum_{q<Q}\mu(q)
\sum_{f\in \overline{\W_q(\Sigma)}}
\frac{\lambda_n(f)}{n^{1/2}}
\H_{\frac{n}{\sqrt{X}}}\Bigl(\frac{|\Delta(f)|}{X}\Bigr).
\]
Recall that by convention, the sum over $f$ is weighted by $1/|\Stab(f)|$. We may then use Corollary \ref{lemlamer}, to estimate the inner sum over $f$:

\begin{multline}\label{qlessQ}
2\sum_{n<X^{1/2+\epsilon}}\frac{1}{\sqrt{n}} 
\sum_{q<Q}\mu(q)
\left (\alpha^\pm \cA^{(q)}_{[n,r_\Sigma]}(\lambda_n \chi_\Sigma)
\cdot 
\widetilde{\H_{\frac{n}{\sqrt{X}}}}
 (1)
X+
\gamma^\pm 
\cC^{(q)}_{[n,r_\Sigma]}(\lambda_n \chi_\Sigma)
\cdot \widetilde{\H_{\frac{n}{\sqrt{X}}}}
 (\frac 56) \cdot 
 X^{5/6}
\right )
\\
+\displaystyle O_{\epsilon,\Sigma,\Psi}\Bigl(\sum_{n< X^{1/2+\epsilon}}\frac{1}{\sqrt{n}}
\sum_{q<Q}
(nq)^{1+\epsilon}
\cdot 
E_\infty(\widetilde{\H_{\frac{n}{\sqrt{X}}}},\epsilon)
\Bigr).
\end{multline}
The error term above is seen to be bounded by
$O_{\epsilon,\Sigma,\Psi}(Q^{2+\epsilon}X^{3/4+\epsilon})$ thanks to
Lemma~\ref{bound-H}. 

Next, we bound the secondary term in~\eqref{qlessQ}. Since $r_\Sigma$ is fixed, the 
contribution to $\cC^{(q)}_{[n,r_\Sigma]}(\lambda_n \chi_\Sigma)$ from primes $p\mid r_\Sigma$ is bounded.
Therefore, we consider without further mention in the remainder of this paragraph only 
the primes $p\nmid r_\Sigma$.
We begin with the primes $p$ dividing $q$. The contribution to $\cC^{(q)}_{[n,r_\Sigma]}(\lambda_n \chi_\Sigma)$ from
primes $p\mid q$ and $p\nmid n$ is given in \cite[Thm.2.2]{TaTh1} and \cite[Cor.8.15]{TT} to be
$O(p^{-5/3})$. (Note that our quantity $\cC^{(p)}_1(1)$ defined in \S\ref{s_sieve_max} 
corresponds to the quantity denoted $\cC_{p^2}(\Phi_p,1)$ in~\cite{TT}.) 
The contribution to $\cC^{(q)}_{[n,r_\Sigma]}(\lambda_n \chi_\Sigma)$ from primes $p\mid 
q$ and $p\mid n$ is estimated 
from~\cite[Prop.8.16]{TT} to also be $O(p^{-5/3})$. (If $a=(1^21_*)$, then $\mathcal 
C_{p^2}(a,1)\asymp p^{1/3}$, and the cardinality of the orbit 
$\GL_2(\Z/p^2\Z)\cdot 
a$ inside $V(\Z/p^2\Z)$ is equal to $p^4(p^2-1)$ by \cite[Lem.5.6]{TT}, which yields $p^{1/3} 
p^6 / p^8 = p^{-5/3}$, whereas the other nonmaximal types $a=(1^3_*)$, 
$(1^3_{**})$, $(0)$ have a smaller contribution.)

We turn to the primes $p$ not dividing $q$. The contribution to $\cC^{(q)}_{[n,r_\Sigma]}(\lambda_n \chi_\Sigma)$ from primes $p\nmid q$ and $p\nmid n$ is a convergent infinite product that is uniformly bounded.
The contribution to $\cC^{(q)}_{[n,r_\Sigma]}(\lambda_n \chi_\Sigma)$ from 
primes $p\nmid q$ and $p\parallel n$ is computed from~\eqref{def_lambda} and Table \ref{tabbc} to be
$O(p^{-1/3})$ (see also Lemma~\ref{lemAC-lambda}).
The contribution to $\cC^{(q)}_{[n,r_\Sigma]}(\lambda_n \chi_\Sigma)$ from primes $p\nmid q$ and $p^2\mid n$ 
is bounded by $O_\epsilon(n^\epsilon)$ since $c_p$ is absolutely bounded and $|\lambda_n|\ll_\epsilon n^\epsilon$. Therefore, letting $n_1:=\prod_{p||n} 
p$ and writing $n=n_1n_2$, we see that 
the 
secondary term 
in~\eqref{qlessQ} is $\ll_{\epsilon,\Sigma,\Psi}$
\begin{equation*}
  X^{\frac{5}{6}} \sum_{n< X^{1/2+\epsilon}}\frac{1}{\sqrt{n}}\sum_{q<
  Q}\frac{(n,q)^{1/3}}{q^{5/3-\epsilon}n_1^{1/3-\epsilon}}\ll_{\epsilon,\Sigma,\Psi}X^{\frac56+\epsilon}\sum_{\substack{n_1<X^{1/2+\epsilon}\\|\mu(n_1)|=1}}\frac{1}{n_1^{5/6}}\sum_{\substack{n_2<X^{1/2+\epsilon}}}\frac{1}{\sqrt{n_2}}\ll_{\epsilon,\Sigma,\Psi} X^{\frac{11}{12}+\epsilon},
\end{equation*}
where the final estimate follows since the inner sum is over powerful integers $n_2$ and hence is $\ll_\epsilon X^\epsilon$.

Finally, to express the first main term in a more convenient form, we define the function $g(y)$ to be
\begin{equation}\label{eqgXn}
g(y):=\widetilde{\H_y}(1) = \int_0^\infty \H_y(t)dt.
\end{equation}
\index{$g(y)$, equal to $\widetilde{H_y}(1)$}
From Lemma \ref{lemmtcompmax}, we see that for a fixed $n$, we have
\begin{equation*}
\begin{array}{rcl}
\displaystyle\sum_{q<Q}\mu(q)\cA^{(q)}_{[n,r_\Sigma]}(\lambda_n\chi_\Sigma)&=&
\displaystyle\cA^{\max}(\lambda_n\chi_\Sigma)
+O\Big(\sum_{q\geq Q}\cA^{(q)}_{[n,r_\Sigma]}(\lambda_n\chi_\Sigma)\Big)
\\[.2in]&=&\displaystyle
\cA^{\max}(\lambda_n\chi_\Sigma)+O_\epsilon\Bigl(\sum_{q\geq Q}\frac{(n_1,q)}{q^2n_1^{1-\epsilon}}\Bigr),
\end{array}
\end{equation*}
where as before $n_1:=\prod_{p\parallel n}p$. We omit the details of the bound on $\cA^{(q)}_{[n,r_\Sigma]}(\lambda_n\chi_\Sigma)$, since it is 
similar (and simpler) to the bound on $\cC^{(q)}_{[n,r_\Sigma]}(\lambda_n\chi_\Sigma)$.
Thus, writing $n=n_1n_2$,
the first term in~\eqref{qlessQ}
is equal to
\begin{equation*}
2\alpha^\pm \cdot X \cdot \sum_{n<X^{1/2+\epsilon}}\frac{g(\frac{n}{\sqrt 
X})}{\sqrt{n}}\cA^\max(\lambda_n \chi_\Sigma)
+O_{\epsilon,\Sigma,\Psi}\Bigl(\sum_{n< X^{1/2+\epsilon}}
\frac{X}{Qn_1^{3/2-\epsilon}n_2^{1/2}}\Bigr).
\end{equation*}
The result now follows with the values of the constants being
\begin{equation}\label{defCSigma}
  C_\Sigma:=\alpha^\pm\Res_{s=\frac12}
  T_\Sigma(s),\quad C'_\Sigma:= 2\alpha^\pm C',
\end{equation}
as is shown in 
Proposition~\ref{propconstant} below, and where $T_\Sigma$ is defined in \eqref{defTSigma}.
\end{proof}

\subsection{Computing the leading constants}\label{s_Burgess}

We compute the constants $C_\Sigma,C'_\Sigma$ arising in Proposition
\ref{p_main-term}. We begin with the
following lemma.
\begin{lemma}\label{lemMelgX}
The Mellin transform of the function $g$ in~\eqref{eqgXn} is
\begin{equation*}
  \widetilde{g}(s)=\widetilde{\Psi}(1+s/2)\frac{G(s)}{s}
  \frac{\gamma^\pm(1/2+s)}{\gamma^\pm(1/2)},
\end{equation*}
where $G$ is as in~\eqref{defV}. In particular, $\widetilde{g}(s)$ is
meromorphic on the half-plane $\Re(s)>-1/2$ with only a simple pole at $s=0$.
\end{lemma}
\begin{proof}
Unwinding definitions~\eqref{eqimpfunc} and~\eqref{eqgXn}, we see that
\begin{equation*}
\begin{array}{rcl}
\displaystyle\widetilde{g}(s)&=&
\displaystyle\int_{0}^\infty \Psi(t)\int_{0}^\infty
V^\pm\Bigl(\frac{y}{\sqrt{t}}\Bigr)y^s\frac{dy}{y}dt
\\[.2in]&=&
\displaystyle \int_{0}^\infty t^{s/2+1}
\Psi(t)\frac{dt}{t}\int_{0}^\infty V^\pm(u)u^s
\frac{du}{u}
\\[.2in]&=&
\displaystyle \widetilde{\Psi}(1+s/2)\widetilde{V^\pm}(s).
\end{array}
\end{equation*}
The lemma follows from the expression~\eqref{eqMelV} for
$\widetilde{V^\pm}(s)$.
\end{proof}

\index{$T_\Sigma(s)$, Dirichlet series of $t_\Sigma(n)$}
Define the Dirichlet series
\begin{equation}\label{defTSigma}
T_\Sigma(s):=\sum_{n=1}^\infty\frac{t_\Sigma(n) }{n^s},\end{equation}
where 
\index{$t_\Sigma(n)$, average of $\lambda_K(n)$ over $K$ in $\FF_\Sigma$}
$t_\Sigma(n)=\cA^{\max}(\lambda_n\chi_\Sigma)$ is the average of $\lambda_K(n)$ over $K$ in
$\FF_\Sigma$ (note that this is actually a finite average, since the
value of $\lambda_K(n)$ is determined by the splitting type of $K$ at
the primes dividing $n$.)

\begin{proposition}\label{propanalcont}
The Dirichlet series $T_\Sigma(s)$ has a meromorphic continuation to
the half-plane $\Re(s)>1/3$ with a simple pole at $s=\frac
12$. Moreover, this simple pole has a positive residue.
\end{proposition}
\begin{proof}
For every integer $n\ge 1$, we have
\begin{equation*}
  t_\Sigma(n)=\prod_{p^k\parallel n}\sum_{\sigma}
  \frac{\lambda_{p^k}(\sigma)}{\# \mathcal O_\sigma},
\end{equation*}
where $\mathcal O_\sigma \subset V(\F_p)$ is the $\GL_2(\F_p)$-orbit attached to $\sigma$, and  $\sigma$ 
ranges over all 
splitting types that are
compatible with $\Sigma_p$. The quantity $t_\Sigma(n)$ is clearly
multiplicative, and so $T_\Sigma(s)$ has an Euler product decomposition
\begin{equation*}
T_\Sigma(s):=\prod_p\sum^\infty_{k=0}\frac{t_\Sigma(p^k)}{p^{ks}}.
\end{equation*}
If $p\neq 3$ and there is no specification $\Sigma_p$ at $p$, then
Proposition~\ref{propmaxden} asserts that
$t_\Sigma(p)=\frac{(p-1)(p^2-1)}{p^4}$ and that
$t_\Sigma(p^2)=\frac{(p^2-1)^2}{p^4}$.  Therefore, the 
Dirichlet
series $T_\Sigma(s)\zeta(2s)^{-1}$ converges absolutely for
$\Re(s)>1/3$.

\medskip

It follows that the residue at $s=\tfrac 12$ is given by the following
convergent product
\[
 \Res_{s=\frac12} T_\Sigma(s) = \frac{1}{2} \prod_p (1-p^{-1}) \sum^\infty_{k=0} 
 \frac{t_\Sigma(p^k)}{p^{k/2}}.
\]
We claim that each factor in the product is positive:
\[
 \sum^\infty_{k=0} 
 \frac{t_\Sigma(p^k)}{p^{k/2}} >0
\quad
\text{for every prime $p$.}
\]
Indeed, $\lambda_{p^m}(f)$ is only negative if $\sigma_p(f)=(3)$ and
$m\equiv 1\pmod{3}$, in which case $\lambda_{p^m}(f)=-1$. 
Therefore, the minimum possible value of
$ \sum^\infty_{k=0}\frac{t_\Sigma(p^k)}{p^{k/2}}$ occurs when $\Sigma_p=\{(3)\}$. In
this case
\begin{equation}\label{eqpostest}
\begin{array}{rcl}
\displaystyle \sum^\infty_{k=0}
\frac{t_\Sigma(p^k)}{p^{k/2}}&=&
\displaystyle \sum_{k\equiv 0\pmod 3}\frac{1}{p^{k/2}}-
\sum_{k\equiv 1\pmod{3}}\frac{1}{p^{k/2}},
\end{array}
\end{equation}
which is clearly positive since the $n$th term of the sum on the left
is greater than the $n$th term of the sum of the right.
\end{proof}

\begin{proposition}\label{propconstant}
As $X\to \infty$, we have the asymptotic
\[
\displaystyle\sum_{n=1}^\infty \frac{t_\Sigma(n)}{\sqrt{n}}  g(\frac{n}{\sqrt X})
=\frac{1}{2}\Res_{s=\frac12}T_\Sigma(s)\cdot
\bigl(\log X + \widetilde \Psi'(1)\bigr) + 
C'
+O_{\epsilon,\Sigma,\Psi}\bigl(X^{-\frac1{12} + \frac \epsilon 2}
\bigr),
\]
where
\[
 C':=
{\frac{d}{ds}}\Bigl|_{s=0}
sT_\Sigma(\tfrac 12+s)
  \frac{\gamma^\pm\bigl(\frac12+s\bigr)}{\gamma^\pm(1/2)}.
\] 
\end{proposition}
\begin{proof}
From Lemma \ref{lemMelgX}, we obtain
\begin{equation}\label{eqconstantint}
\begin{array}{rcl}
  \displaystyle\sum_{n=1}^\infty
\frac{t_\Sigma(n)}{\sqrt{n}}
g(\frac{n}{\sqrt X})&=&\displaystyle
\frac{1}{2\pi i}\int_{\Re(s)=2}
T_\Sigma(\tfrac 12+s)\widetilde{g}(s) X^{s/2}ds\\[.2in]
&=&
  \displaystyle
  \frac{1}{2\pi i}\int_{\Re(s)=2}s T_\Sigma(\tfrac 12+s)\widetilde{\Psi}(1+s/2)
  G(s)\frac{\gamma^\pm\bigl(\frac12+s\bigr)}{\gamma^\pm(1/2)}X^{s/2}\frac{ds}{s^2}
  \\[.2in]
&=&
  \displaystyle
  \frac{1}{2\pi i}\int_{\Re(s)=2} J(s)X^{s/2}\frac{ds}{s^2},
\end{array}
\end{equation}
where the above equation serves as a definition of $J(s)$.

Since $\widetilde{\Psi}(1)=G(0)=1$, it follows 
that $J(s)$ is holomorphic in $\Re(s)>-\frac16$, and
$J(0)=\Res_{s=\frac12}T_\Sigma(s)$. Expanding in Taylor series, we write
\begin{equation*}
J(s)X^{s/2}=J(0)+\Bigl(\frac{J(0)\log X}{2}+J'(0)\Bigr)s+\cdots
\end{equation*}
Shifting the integral to $\Re(s)=-\tfrac16 +\epsilon$ for some
$0<\epsilon<\frac16$, we therefore obtain
\begin{equation*}
\displaystyle\sum_{n=1}^\infty \frac{t_\Sigma(n)}{\sqrt{n}}  g(\frac{n}{\sqrt X})
=\frac{1}{2}\Res_{s=1/2}T_\Sigma(s)\cdot \log X+J'(0)+O_{\epsilon,\Sigma,\Psi}\bigl(X^{-\frac {1}{12}+\frac \epsilon 2}
\bigr).
\end{equation*}
Calculating $J'(0)$, we obtain, using that $G(s)$ is even:
\[
J'(0)
=\frac{1}{2}\Res_{s=\frac 12}T_\Sigma(s)\cdot  \widetilde \Psi'(1) + 
C'.
\]
This concludes the proof of the proposition.
\end{proof}

\subsection{Upper bound for the first moment}

In this subsection we investigate pointwise bounds for the tail of the
sieve when $q\ge Q$.

\begin{proposition}\label{lem_Eq}
For every $Q,X\ge 1$ and $\epsilon >0$,
\[
\sum\limits_{q\ge Q} 
\sum_{\substack{f\in \overline{\W_q^\irr}\\
|\Delta(f)| < X}}
|S(f)| 
= 
O_{\epsilon}
\Bigl(\frac{X^{5/4-\delta+\epsilon}}{Q^{3/2-2\delta}}\Bigr),
\]
for $\delta=\frac{1}{128}$ as in
Theorem \ref{p_Wu}.
\end{proposition}
\begin{proof}
Let $f\in V(\Z)^\irr$ be an irreducible binary cubic form, and denote the
field of fractions of the ring associated to $f$ by $K_f$.

Note that for $f\in \mathcal W_q$ with $|\Delta(f)| < X$, we have
$|\Delta(K_f)|< X/q^2$, and recall from Proposition \ref{propunif} that
\begin{equation*}
\#\bigl\{f\in
\overline{\W_q}:\ |\Delta(f)|<X\bigr\}\ll_\epsilon\frac{X}{q^{2-\epsilon}}.
\end{equation*}
Therefore, we deduce from~\eqref{Lf_bound} the estimate
\[
\sum\limits_{q\ge Q} 
\sum_{\substack{f\in \overline{\W_q^\irr}\\
|\Delta(f)| < X}}
|S(f)|  \ll_\epsilon
\sum_{q\geq Q} (X/q^2)^{\theta+\epsilon}\cdot X/q^{2-\epsilon},
\]
where we recall that $\theta=1/4-\delta$. The result follows.
\end{proof}

Optimizing, we pick $Q=X^{\frac{1-2\delta}{7-4\delta}}$ in
\eqref{eqlongest}. We have now established the following by combining
the two Propositions \ref{p_main-term} and \ref{lem_Eq}.
\begin{theorem}\label{thuncondupbd}
For every $X\ge 1$ and $\epsilon >0$,
\begin{equation}\label{eqLfavgupbound}
A_\Sigma(X)
\ll_{\epsilon,\Sigma,\Psi}
X^{\frac{29-28\delta}{28-16\delta}+\epsilon}.
\end{equation}
Numerically,
\[
\frac{29-28\delta}{28-16\delta} = \frac{921}{892} = 1.0325\ldots
\]
for the best known value of $\delta=\frac{1}{128}$ of Theorem \ref{p_Wu}.
\end{theorem}
The exponent is smaller than $5/4-\delta = \tfrac{159}{128} = 1.2421875$, 
thus~\eqref{eqLfavgupbound} 
is 
an improvement on the exponent arising from summing the pointwise bound on
$|L(\tfrac 12,\rho_K)|$ over cubic fields $K$ with discriminant bounded by
$X$.

\section{Conditional computation of the first moment of 
$L(\tfrac 12,\rho_K)$}\label{sConditional}

In this section, we shall compute the first moment of $L(\tfrac 12,\rho_K)$
assuming one of two hypotheses. More precisely, we prove the following
result.
\begin{theorem}\label{propSN}
Assume {\rm one} of the following two hypotheses:
\begin{itemize}
\item[{\rm (S)}] {\bf Strong Subconvexity:} For every $K\in\FF_\Sigma$, we
  have $|L(\frac12,\rho_K)|\ll |\Delta(K)|^{\frac{1}{6}-\vartheta}$
  for some $\vartheta>0$.
\item[{\rm (N)}] {\bf Nonnegativity:} For every $K\in\FF_\Sigma$, we
  have $L(\frac12,\rho_K)\geq 0$.
\end{itemize}
Then we have for small enough $\epsilon>0$,
\begin{equation*}
\sum_{K\in\FF_\Sigma}L\bigl(\tfrac12,\rho_K\bigr)
\Psi\Bigl (\frac{|\Delta(K)|}{X}\Bigr)
=
C_\Sigma \cdot X \cdot \bigl(\log X + \widetilde \Psi'(1) \bigr)+
C'_\Sigma \cdot X + 
O_{\epsilon,\Sigma,\Psi}(X^{1-\epsilon}),
\end{equation*}
where $C_\Sigma$ and $C_\Sigma'$ are the constants arising in
Proposition \ref{propconstant}.
\end{theorem}

Compared to Section~\ref{sec:average}, the proof is significantly more difficult, and will 
require
several new inputs.  Indeed, recall that we have
\begin{equation}\label{eqincexcrep}
A_\Sigma(X)= 2 \sum_{q\ge 1}\mu(q)
\sum_{f\in \overline{\W_q(\Sigma)}}
S(f)
\Psi\Bigl(\frac{|\Delta(f)|}{X}\Bigr).
\end{equation}
Pick a small $\kappa_\downarrow>0$. Proposition \ref{p_main-term} provided
an estimate for the above sum with $q$ in the range
$[1,X^{1/8-\kappa_\downarrow}]$.
\index{$q\in [X^{1/8-\kappa_\downarrow},X^{1/8+\kappa_\uparrow}]$, border range of the sieve}
\index{$q\geq X^{1/8+\kappa_\uparrow}$, large range of the sieve}

For $q\geq X^{1/8-\kappa_\downarrow}$, our approach is to approximate the
smoothed sum of $S(f)$ with a smoothed sum of $D(\tfrac 12,f)$.  We do this
by breaking up these $q$ into two ranges: the ``large range'' and the
``border range''. Namely, we pick a small $\kappa_\uparrow>0$. Then the range
$q\geq X^{1/8+\kappa_\uparrow}$ is the large range while the range
$[X^{1/8-\kappa_\downarrow},X^{1/8+\kappa_\uparrow}]$ is the border range.  For $q$ in
both of these ranges we want to prove
\begin{equation}\label{eqSDCloseIntro9}
\sum_{f\in\overline{\W_q(\Sigma)}}S(f)\Psi\Bigl(\frac{|\Delta(f)|}{X}\Bigr)
\approx
\sum_{f\in\overline{\W_q(\Sigma)}}D(\tfrac12,f)\Psi\Bigl(\frac{|\Delta(f)|}{X}\Bigr).
\end{equation}
On average over $f\in\overline{\W_q(\Sigma)}$, this is an unbalanced approximation of the central 
value $D(\tfrac12,f)$ by the Dirichlet sum $S(f)$ of the coefficients $\lambda_n(f)$.

In \S\ref{sLargerange}, we establish \eqref{eqSDCloseIntro9} with $q$ in the large
range, which is straightforward. The bulk of the section is devoted to proving
\eqref{eqSDCloseIntro9} in the border range. This is proved in
\S\ref{sPreparations} and \S\ref{sBorder} using the
unbalanced approximate functional equation of Proposition~\ref{thm_AFE2}. 
The crux of the proof is to estimate the average of the coefficients $e_k(f)$ of the 
unbalanced Euler 
factors $E_p(s,f)$ over the forms $f\in \overline{\W_q(\Sigma)}$.
Finally, in 
\S\ref{sSuborders}, we compute the average
of $D(\tfrac 12,f)$ (assuming either nonnegativity or strong subconvexity of $L(\tfrac 
12,\rho_K)$), 
thereby obtaining the average of $S(f)$ and finishing the proof of Theorem
\ref{propSN}.

\subsection{Estimates for the large range}\label{sLargerange}

We begin by estimating $S(f)$ for integral binary cubic forms with large index.
\begin{lemma}\label{lemeasybound}
For every integral binary cubic form $f\in V(\Z)^{\rm irr}$ and every
$\epsilon>0$, we have
\begin{equation*}
S(f)=D(\tfrac 12,f)+O_\epsilon\Bigl(\frac{|\Delta(f)|^{1/4+\epsilon}}{\ind(f)}\Bigr).
\end{equation*}
\end{lemma}
\begin{proof}
Recall the computation of $\widetilde{V^\pm}(s)$ in \eqref{eqMelV}, and note that by definition, we have
\begin{equation*}
S(f)=\frac{1}{2\pi i}\int_{\Re(s)=2}D(\tfrac12+s,f)\widetilde{V^\pm}(s)
|\Delta(f)|^{s/2}ds.
\end{equation*}
Shifting to the line $s=-1/2+\epsilon$, we pick up the pole of
$\widetilde{V^\pm}(s)$ at $0$ (with residue $1$), to obtain
\begin{equation*}
\begin{array}{rcl}
\displaystyle  S(f)&=&\displaystyle D(\tfrac 12,f)+
\frac{1}{2\pi i}\int_{\Re(s)=-1/2+\epsilon}D(\tfrac12+s,f)\widetilde{V^\pm}(s)
|\Delta(f)|^{s/2}ds
\\[.2in]&=&\displaystyle
D(\tfrac12,f)+O_\epsilon\bigl(|\Delta(f)|^{-1/4+\epsilon}|\Delta(K)|^{1/2+\epsilon}\bigr),
\end{array}
\end{equation*}
where the final estimate follows since $D(s,f)$ is within
$|\Delta(f)|^\epsilon$ of $L(s,\rho_K)$ for $\Re(s)$ close to $0$. The
lemma now follows since $\Delta(f)=\ind(f)^2\Delta(K)$.
\end{proof}

Adding up the above estimate for $f\in \overline{\W_q(\Sigma)}$, we
immediately obtain the following result.
\begin{proposition}\label{proplargerange}
For every square-free $q$, and $X\ge 1$, we have
\begin{equation*}
\sum_{f\in \overline{\W_q(\Sigma)}}S(f)\Psi\Bigl(\frac{|\Delta(f)|}{X}\Bigr)
=\sum_{f\in \overline{\W_q(\Sigma)}}D(\tfrac 12,f)\Psi\Bigl(\frac{|\Delta(f)|}{X}\Bigr)
+O_{\epsilon,\Sigma,\Psi}\Bigl(\frac{X^{5/4+\epsilon}}{q^3}\Bigr).
\end{equation*}
\end{proposition}
\begin{proof}
The proposition follows from Lemma \ref{lemeasybound} and the tail estimate in 
Proposition \ref{propunif}.
\end{proof}

An immediate consequence of the previous result is the following
estimate for $q$ in the large range.

\begin{corollary}\label{corlargerange}
For every small $\kappa_\uparrow>0$, square-free $q>X^{1/8+\kappa_\uparrow}$ and $X\ge 1$, we have
\begin{equation*}
  \sum_{f\in \overline{\W_q(\Sigma)}}\bigl(S(f)-D(\tfrac 12,f)\bigr)
  \Psi\Bigl(\frac{|\Delta(f)|}{X}\Bigr)\ll_{\epsilon,\kappa_\uparrow,\Sigma,\Psi}
  \frac{X^{1-2\kappa_\uparrow+\epsilon}}{q}.
\end{equation*}
\end{corollary}

\subsection{Preparations and strategy for the border range}\label{sPreparations}

In this subsection, we shall introduce spaces, notation, and
some preliminary results that will be useful subsequently in handling
the border range. One of the key tools in comparing $S(f)$ and
$D(\tfrac 12,f)$ is the unbalanced approximate functional equation of
Proposition~\ref{thm_AFE2}. To apply this result, it is not possible
to only work with the information that forms $f \in \W_q(\Sigma)$ are nonmaximal at
primes dividing $q$. Rather, we shall work with the additional information of the index of $f$, including at primes not dividing $q$.

To this end,
for a positive (not necessarily squarefree) integer $b$, let
$\U_{b}(\Sigma)$ denote the set of binary cubic forms $f\in
V(\Z)(\Sigma)$ such that $\ind(f)=b$.  \index{$\U_b$, set of cubic
  forms $f$ with $\ind(f)=b$}  Note the inclusion $\U_b(\Sigma) \subset
\W_{\rad(b)}(\Sigma)$, and in fact we have
\begin{equation*}
\W_q(\Sigma) =\bigsqcup_{m\geq 1} \U_{mq}(\Sigma),
\end{equation*}
where the union is disjoint and $q$ is square-free. Let
$\overline{\U_{b}(\Sigma)}$ denote the set of $\GL_2(\Z)$-orbits on
$\U_{b}(\Sigma)$.

\index{$\Y_{b,r}$, subset of cubic forms $f\in \W_r$ with $b\parallel \ind(f)$}
Let $b$ be a positive integer, and let $r$ be a positive squarefree
integer such that $(b,r)=1$. Finally, we define the set $\Y_{b,r}(\Sigma)$ to
be the subset of elements in $\W_{r}(\Sigma)$ whose index at primes $p$ dividing
$b$ is exactly $p^{v_p(b)}$. As usual we let
$\overline{\Y_{b,r}(\Sigma)}$ denote the set of $\GL_2(\Z)$-orbits on
$\Y_{b,r}(\Sigma)$. The significance of these subsets $\Y_{b,r}(\Sigma)$
is the following disjoint union
\[
 \Y_{b,r}(\Sigma) = \bigsqcup_{(b,s)=1} \U_{brs}(\Sigma),
\]
hence for any function $\phi:\overline{\U_b(\Sigma)}\to\C$, we have
\begin{equation*}
\sum_{f\in \overline{\U_b(\Sigma)}}
\phi(f)\Psi\Bigl(\frac{|\Delta(f)|}{X}\Bigr)
=\sum_{(b,r)=1}\mu(r)\sum_{f\in \overline{\Y_{b,r}(\Sigma)}}
\phi(f)\Psi\Bigl(\frac{|\Delta(f)|}{X}\Bigr).
\end{equation*}

Recall that the border range is what we are calling
$q\in[X^{1/8-\kappa_\downarrow},X^{1/8+\kappa_\uparrow}]$, where $\kappa_\downarrow,\kappa_{\uparrow}$ 
are
positive constants that can eventually be taken to be arbitrarily
small. We next estimate the sum of $S(f)-D(\tfrac 12,f)$ over $f$ in
$\overline{\U_{mq}(\Sigma)}$, where $m$ is somewhat large.

We begin by bounding the number of elements in $\overline{\U_{mq}(\Sigma)}\subset \overline{\U^{\rm 
irr}_{mq}}$
that have discriminant less than $X$.
\begin{lemma}\label{lemUunif}
For every positive integer $m$ and square-free $q$,
write $mq=m_1q_1$, where $m_1$ is powerful,
$(m_1,q_1)=1$, and $q_1$ is squarefree. Then for every $X\ge 1$,
\begin{equation}\label{eqUunif}
|\{f\in
\overline{\U^{\rm irr}_{mq}}:|\Delta(f)|<X\}|\ll_\epsilon \frac{X^{1+\epsilon}}{m_1^{5/3}q_1^2}.
\end{equation}
The multiplicative constant depends only on $\epsilon$ (it is independent of $m,q,X$).
\end{lemma}
\begin{proof}
Elements $f$ in the left-hand side of \eqref{eqUunif} are in bijection
with rings $R_f$ that have index $mq=m_1q_1$ in the maximal orders
$\O_{K_f}$ of their fields of fractions $K_f$. It follows that the
discriminants of these fields $K_f$ are less than $X/(m_1^2q_1^2)$. It
follows that the total number of such fields that can arise is bounded
by $O(X/m_1^2q_1^2)$.

To estimate the total number of rings $R_f$ that can arise, it
suffices to estimate the number of such rings $R_f$ within a single
$K_f$. This can be done prime by prime, for each prime dividing the
index $m_1q_1$. Let $p$ be a prime dividing $q_1$. Since $q_1$ is
squarefree, it follows that the index of $R_f$, at the prime $p$, is
$p$. Given the index $p$ overorder $R$ of $R_f$, it follows from
Proposition \ref{subsupring}, that the number of index $p$ suborders
of $R$ is bounded by $3$.

For primes dividing $m_1$, this procedure is more complicated since
there can be many more subrings with prime power index. However, this
question is completely answered by work of Shintani \cite{Shintani}
and Datskovsky--Wright~\cite{DW} (see \cite[\S1.2]{KNT}), who give
an explicit formula for the counting function of suborders $R$ of a fixed
cubic field $K$, which we state as Proposition \ref{propSDW}. They
show that the number of suborders of index $m$, for $m\geq 1$, is the
$m$th Dirichlet coefficient of
\begin{equation*}
\frac{\zeta_K(s)}{\zeta_K(2s)}\zeta_\Q(3s)\zeta_\Q(3s-1).
\end{equation*}
To verify the lemma, it suffices to bound the Dirichlet coefficients of the Euler factor of primes $p$ having splitting type $(111)$, since these coefficients majorize those of primes with all other splitting types.
For such a prime, the $p$th Euler factor of the above Dirichlet
series is:
\begin{equation*}
(1-p^{-s})^{-3}(1-p^{-2s})^{3}(1-p^{-3s})^{-1}(1-p^{-3s+1})^{-1}=
(1+p^{-s})^3\Bigl(\sum_{k=0}^\infty p^{-3ks}\Bigr)
\Bigl(\sum_{k=0}^\infty p^{k-3ks}\Bigr).
\end{equation*}
It is thus clear that the $k$th Dirichlet coefficient is bounded by
$O(p^{k/3})$. Therefore, the number of possible suborders of index
$p^k$ is bounded by $O(p^{k/3})$.

Putting this together, it follows that the number of suborders of $K$,
having index $q_1m_1$ is bounded by $O(q_1^\epsilon m_1^{1/3})$. Multiplying
this quantity by $X/(m_1q_1)^2$ yields the result.
\end{proof}

\begin{lemma}\label{lemborderrange}
For $X\ge 1$, square-free $q$, and small enough $\eta>0$
\begin{equation*}
\sum_{m> X^\eta}\sum_{f\in \overline{\U_{mq}(\Sigma)}}
\bigl( S(f) - D(\tfrac 12,f) \bigr) \Psi\Bigl(\frac{|\Delta(f)|}{X}\Bigr)=
O_{\epsilon,\Sigma,\Psi}\Bigl(\frac{X^{5/4-\eta+\epsilon}}{q^3}\Bigr).
\end{equation*}
\end{lemma}
\begin{proof}
From Lemma \ref{lemeasybound}, it follows that for
$f\in\U_{mq}(\Sigma)$ with $|\Delta(f)|\asymp X$, we have
$S(f)-D(\tfrac 12,f)=O(\frac{X^{1/4+\epsilon}}{mq})$.  We write $mq$ as
$m_1q_1$, where $q_1$ is squarefree with $(q_1,m_1)=1$, and $m_1$ is
powerful.  We now have
\begin{equation*}
\begin{array}{rcl}
\displaystyle\sum_{m> X^\eta} \sum_{f\in \overline{\U_{mq}(\Sigma)}}
\bigl(S(f)-D(\tfrac 12,f)\bigr)
\Psi\Bigl(\frac{|\Delta(f)|}{X}\Bigr) \ll_{\epsilon,\Sigma,\Psi}
\displaystyle\sum_{m>X^\eta}\frac{X^{1/4+\epsilon}}{mq}\cdot
\frac{X^{1+\epsilon}}{m_1^{5/3}q_1^2},
\end{array}
\end{equation*}
where the final estimate follows from Lemma \ref{lemUunif}.
\end{proof}

We then have the following corollary.

\begin{corollary}\label{corborderrange}
Let $X\ge 1$, squarefree $q>X^{1/8-\kappa_\downarrow}$, and $\eta>0$ be
such that $\eta-2\kappa_\downarrow>0$. Then we have
\begin{equation*}
\sum_{m> X^\eta}\sum_{f\in \overline{\U_{mq}(\Sigma)}}
\bigl( S(f) - D(\tfrac 12,f) \bigr) \Psi\Bigl(\frac{|\Delta(f)|}{X}\Bigr)=
O_{\epsilon,\kappa_\downarrow,\Sigma,\Psi}\Bigl(\frac{X^{1+2\kappa_\downarrow-\eta+\epsilon}}{q}\Bigr).
\end{equation*}
\end{corollary}

\noindent Furthermore, $\kappa_\downarrow$ and hence $\eta$ can be taken to
be arbitrarily small. Therefore, a consequence of the above lemma is
that when $q$ is in the border range, sums over $\overline{\U_{mq}(\Sigma)}$ only have to
be considered for $m$ less than arbitrarily small powers of $X$.

\medskip

Let $q\in [X^{1/8-\kappa_\downarrow},X^{1/8+\kappa_\uparrow}]$ be fixed for the rest of this 
subsection. For 
a 
positive
integer $m$, we write $mq=m_1q_1$, where $m_1$ is powerful,
$(m_1,q_1)=1$, and $q_1$ is squarefree. Note that since $m$ will be
taken to be very small ($\ll X^\eta$), $q_1$ will be quite close in
size to $q$. We restate Proposition~\ref{thm_AFE2} for convenience:
for $f \in \overline{\U_{m_1q_1}(\Sigma)}$, we have
\begin{equation}\label{eqinex2}
S(f)=D(\tfrac 12,f)-
\sum_{k=1}^\infty
\frac{e_k(f)k^{1/2}}{q_1\rad(m_1)}
\sum_{n=1}^\infty
\frac{\lambda_n(f)}{n^{1/2}}
V^\pm
\left(
\frac{m_1^2 kn}{\rad(m_1)^2 |\Delta(f)|^{1/2}}
\right),
\end{equation}
where $e_k(f)k^{1/2}$ is the $k$th Dirichlet coefficient of the
  series
\begin{equation*}
  \sum_{k=1}^\infty \frac{e_k(f)k^{1/2}}{k^s}
  =q_1^{1-2s}\rad(m_1)^{1-2s}\frac{E(\frac12-s,f)}{E(\frac12+s,f)}.
\end{equation*}

Our next and final goal of this subsection is to perform a switching
trick, analogous to Theorem \ref{thswitch}, in which our sums over
$\overline{\U_{m_1q_1}(\Sigma)}$ are replaced with sums over
$\overline{\U_{m_1}(\Sigma)}$.  We thus need to understand how the
quantity $e_k(f)$ behaves under such a switch.  The next lemma does
just that: more precisely, if $f$ is nonmaximal and switches to the
pair $(g,\alpha)$ with prime index $p$, then the next lemma determines
$e_k(f)$ in terms of $(g,\alpha)$.

As recalled in Proposition~\ref{subsupring}, the proof of \cite[Prop.16]{BST} implies 
that there is a
bijection between the zeros in $\P^1(\F_p)$ of the reduction modulo
$p$ of $g(x,y)$ and the set of cubic rings that are index-$p$ subrings
of $R_g$. Thus, $f$ corresponds uniquely to a pair $(g,\alpha)$, where
$\alpha\in \P^1(\F_p)$ is a root of $g(x,y)$ modulo $p$. Then the
following lemma determines $E_p(s,f)$ given this pair $(g,\alpha)$.
\begin{lemma}\label{gDetermineE}
Let $g\in V(\Z)$ be a binary cubic form that is maximal at $p$. Let
$\alpha\in\P^1(\F_p)$ be a root of the reduction of $g$ modulo
$p$. Let $f\in V(\Z)$ be a binary cubic form corresponding to the index-$p$
subring of $R_g$ associated to the pair $(g,\alpha)$. Then $E_p(s,f)$,
and hence $e_k(f)$ for every $k$, is determined by the pair
$(g,\alpha)$.  More precisely, we have
\begin{itemize}
\item[{\rm (a)}] If $\sigma_p(g)=(111)$, then $\sigma_p(f)=(1^21)$ and
  $E_p(s,f)=1-p^{-s}$;
\item[{\rm (b)}] If $\sigma_p(g)=(12)$, then $\sigma_p(f)=(1^21)$ and
  $E_p(s,f)=1+p^{-s}$;
\item[{\rm (c)}] If $\sigma_p(g)=(1^21)$ and $\alpha$ is the single
  root, then $\sigma_p(f)=(1^21)$ and $E_p(s,f)=1$;
\item[{\rm (d)}] If $\sigma_p(g)=(1^21)$ and $\alpha$ is the double
  root, then $\sigma_p(f)=(1^3)$ and $E_p(s,f)=1-p^{-s}$;
\item[{\rm (e)}] If $\sigma_p(g)=(1^3)$, then $\sigma_p(f)=(1^3)$ and
  $E_p(s,f)=1$.\qedhere
\end{itemize}
\end{lemma}
\begin{proof}
The procedure to compute $f(x,y)$ given the pair $(g,\alpha)$ is as
follows: use the action of $\GL_2(\Z)$ to move $\alpha$ to the point
$[1:0]\in\P^1(\F_p)$. This yields the binary cubic form
$ax^3+bx^2y+cxy^2+dy^3$, where $p\mid a$. Moreover, since $g$ is
maximal at $p$, we see that $p\mid b$ implies that $p^2\nmid a$. Then
$f(x,y)$ can be taken to be $(a/p)x^3+bx^2y+pcxy^2+p^2dy^3$. Running
this procedure for the different splitting types of $g$ immediately
shows that the corresponding $f$ has the splitting type listed in the
lemma.

For example, if $g$ has splitting type $(111)$ or $(12)$, then we may
bring one of the single roots (using a $\GL_2(\Z)$-transformation) to
infinity. Then we may write $g(x,y)=ax^3+bx^2y+cxy^2+dy^3$, where $p\mid a$ and
$p\nmid b$ since $g$ is unramified.  Then the procedure gives
$f(x,y)=(a/p)x^3+bx^2y+pcxy^2+p^2dy^3$. Since $p\nmid b$, the splitting
type of $f(x,y)$ is $(1^21)$ as claimed. The other cases are similar,
and we omit them.

Finally, $e_{\ell,\beta}(f)$ is determined for $p\neq \ell|\ind(f)$ and all $\beta\ge 0$ as follows from 
Lemma~\ref{l_sigmap_f}.
\end{proof}

The final result of this subsection is to determine what happens to
the quantity $e_k(f)\lambda_n(f)$ after the switch.

\begin{lemma}\label{switchek}
Let $m_1$ and $q_1$ be positive integers, where $m_1$ is powerful,
$(m_1,q_1)=1$, and $q_1$ is squarefree. Let $k$ be a positive integer divisible only
by primes dividing $m_1q_1$. Let $n$ be a positive integer
and write $n=n_1\ell_1$ where $(\ell_1,q_1)=1$ and $n_1$ is divisible only
by primes dividing $q_1$. Then we have
\begin{equation*}
\sum_{f\in \overline{\U_{m_1q_1}(\Sigma)}}e_k(f)\lambda_n(f)\Psi(|\Delta(f)|)=
\sum_{g\in \overline{\U_{m_1}(\Sigma)}}c_{q_1}(g)d_{m_1}(g)
\lambda_{\ell_1}(g)\Psi(q_1^2|\Delta(g)|),
\end{equation*}
where $c_{q_1}$ and $d_{m_1}$ are congruence functions on $V(\Z)$
defined modulo $q_1$ and $m_1^3$, respectively. Furthermore, we have
$c_{q_1}(g)\ll_\epsilon q_1^\epsilon$ and
$d_{m_1}(g)\ll_\epsilon m_1^\epsilon$ uniformly for every $g\in V(\Z)$.
\end{lemma}
\begin{proof}
As in Section~\ref{sec:switch}, we will write sums over $\overline{\U_{m_1q_1}(\Sigma)}$ in terms
of sums over $\overline{\U_{m_1}(\Sigma)}$. In this case, we have the simple
bijection
\begin{equation*}
  \U_{m_1q_1}(\Sigma) \leftrightarrow \bigl\{(g,\alpha):g\in\U_{m_1}(\Sigma),
  \alpha\in\Z/q_1\Z,g(\alpha)\equiv 0\pmod{q_1}\bigr\},
\end{equation*}
which follows by an argument similar to that of Lemma~\ref{lembij}.

Since the functions $e_k$ and $\lambda_n$ are multiplicative, we may write
\begin{equation*}
e_k(f)=e_{k_1}(f)e_{k_2}(f);\quad \lambda_n(f)=\lambda_{n_1}(f)\lambda_{\ell_1}(f),
\end{equation*}
where $k_1$ is only divisible by primes dividing $q_1$, and $k_2$ is
only divisible by primes dividing $m_1$. To prove the lemma, we need to express 
$e_{k_1}(f)$, $e_{k_2}(f)$, $\lambda_{n_1}(f)$, and $\lambda_{\ell_1}(f)$ in terms of congruence functions on the $(g,\alpha)$ corresponding to $f$ under the above bijection. We begin by noting that we have $e_{k_2}(f)=e_{k_2}(g)$ and $\lambda_{\ell_1}(f)=\lambda_{\ell_1}(g)$; the function $e_{k_2}(g)$ is defined modulo $m_1^3$ (since $g$ has index $m_1$) and of course the function $\lambda_{\ell_1}(g)$ is defined modulo $\ell$, the radical of $\ell_1$.

Next, since $\lambda_{n_1}(f)=0$ if $\alpha$ corresponds to a double root of $g$
modulo some prime $p\mid (q_1,n_1)$, and $\lambda_{n_1}(f)=1$ otherwise, it is easy to see that $\lambda_{n_1}(f)$ can be expressed as a congruence function on $g$ defined modulo $(q_1,n_1)$.
Finally, we have seen in Lemma~\ref{gDetermineE} that the value $e_{k_1}(f)$
depends only on the splitting type of $g$ modulo all the primes
dividing $q_1$, and on whether $\alpha$ is a single or double root
modulo all the primes dividing $q_1$.  It is thus clear that $e_{k_1}(f)$
can also be expressed as a congruence function on $g$ defined via congruence conditions modulo $q_1$. The first claim
of the lemma now follows.

The bounds in the second claim of the lemma are immediate since
$\lambda_{n_1}$, $e_{k_1}$ and $e_{k_2}$, each are bounded by $\ll_\epsilon n_1^\epsilon$, $\ll_\epsilon 
k_1^\epsilon$ and $\ll_\epsilon k_2^\epsilon$, respectively 
(see 
Proposition~\ref{ekbound} and the examples just before Proposition 
\ref{Efs} for
the claims regarding $e_{k_i}$).
\end{proof}

\subsection{Estimates for the border range}\label{sBorder}

In this subsection, we assume that our integers $q$ lie in the border
range $[X^{1/8-\kappa_\downarrow},X^{1/8+\kappa_\uparrow}]$ with small enough 
$\kappa_\downarrow,\kappa_\uparrow>0$. 
Our goal is to bound
\begin{equation*}
\sum_{f\in \overline{\W_{q}(\Sigma)}}
\bigl( S(f) - D(\tfrac 12,f) \bigr) \Psi\Bigl(\frac{|\Delta(f)|}{X}\Bigr),
\end{equation*}
for $q$ in this range. Recall that we have a disjoint union
\begin{equation*}
\overline{\W_q(\Sigma)}=\bigsqcup_{m\geq 1}\overline{\U_{mq}(\Sigma)},
\end{equation*}
and that we will be summing $S(f)-D(\tfrac 12,f)$ over $\overline{\U_{mq}(\Sigma)}$ (and then
summing over $m$) rather than simply summing over $\overline{\W_q(\Sigma)}$. From Lemma
\ref{lemborderrange}, it follows that we may restrict the sum to
$m\leq X^\eta$, where $\eta$ may be taken to be arbitrarily small.
All multiplicative constants are understood to depend on the initial
choices of $\kappa_\downarrow,\kappa_\uparrow,\eta>0$.

We write $mq=m_1q_1$, where $m_1$ is powerful, $(m_1,q_1)=1$ and $q_1$
is squarefree. Note then that $m_1\leq m^2\leq X^{2\eta}$, and thus
$q_1\geq q/m\geq X^{1/8-\eta-\kappa_\downarrow}$.  We begin by fixing $k$ and $n$
in \eqref{eqinex2}, and bounding the sum over $f\in \overline{\U_{m_1q_1}(\Sigma)}$.
\begin{proposition}\label{propafeeb}
For every small enough $\kappa_1>0$, the following estimate holds.
Let $m_1$, $q_1$, $k$, and $n$ be positive integers and $X\ge
1$. Assume that $m_1$ is powerful, $(m_1,q_1)=1$, and $q_1$ is
squarefree. Write $n=n_1\ell_1$ where $(\ell_1,q_1)=1$ and $n_1$ is
divisible only by primes dividing $q_1$. Denote the radical of
$\ell_1$ by $\ell$. Then
\begin{equation*}
\sum_{f\in \overline{\U_{m_1q_1}(\Sigma)}}e_k(f)
\lambda_n(f)
V^\pm\Bigl(\frac{nkm_1^2}{\rad(m_1)^2|\Delta(f)|^{1/2}}
\Bigr)
\Psi\Bigl(\frac{|\Delta(f)|}{X}\Bigr)\ll_{\epsilon,\Sigma,\Psi} 
X^\epsilon\cdot
H(n,m_1,q_1;X),
\end{equation*}
where
\begin{equation*}
H(n,m_1,q_1;X)=\frac{X}{q_1^2m_1^{5/3}\ell} +
\frac{X^{5/6+\kappa_1/3}}{q_1^{5/3}\ell^{1/3}}+
\ell q_1^2 m_1^{12} X^{9\kappa_1} + \frac{X^{1-\kappa_1}}{q_1^2 m_1^{5/3}}.
\end{equation*}
\end{proposition}

\begin{proof}Applying the preceding Lemma~\ref{switchek},
we obtain
\begin{equation*}
\sum_{f\in \overline{\U_{m_1q_1}(\Sigma)}}e_k(f)
\lambda_n(f)V\Bigl(\frac{nkm_1^2}{\rad(m_1)^2\sqrt{|\Delta(f)|}}
\Bigr)\Psi\Bigl(\frac{|\Delta(f)|}{X}\Bigr)=
\sum_{f\in\overline{\U_{m_1}(\Sigma)}}c_{q_1}(f)d_{m_1}(f)\lambda_{\ell_1}(f)
\Psi_1\Bigl(\frac{q_1^2|\Delta(f)|}{X}\Bigr),
\end{equation*}
where $c_{q_1}$ is defined modulo $q_1$, $d_{m_1}$ is defined modulo
$m_1^3$, and $\Psi_1=\H_{\frac{nkm_1^2}{\sqrt{X}\rad(m_1)^2}}$. Recall
that in Corollary~\ref{bound-H}, we bound
$E_\infty(\widetilde{\Psi_1};-\epsilon)$ by an absolute constant.  For
brevity in this proof, we will write $\ll$ as a shorthand for
$\ll_{\epsilon,\Sigma,\Psi}$.

We perform an inclusion-exclusion principle to write the sum over
$\overline{\U_{m_1}(\Sigma)}$ in terms of sums over
$\overline{\Y_{m_1, r}(\Sigma)}$. This yields
\begin{equation*}
\sum_{f\in\overline{\U_{m_1}(\Sigma)}}c_{q_1}(f)d_{m_1}(f)\lambda_{\ell_1}(f)
\Psi_1\Bigl(\frac{q_1^2|\Delta(f)|}{X}\Bigr)=
\sum_{(m_1,r)=1}\mu( r )\sum_{f\in\overline{\Y_{m_1, r }(\Sigma)}}
c_{q_1}(f)d_{m_1}(f)\lambda_{\ell_1}(f)
\Psi_1\Bigl(\frac{q_1^2|\Delta(f)|}{X}\Bigr).
\end{equation*}
We split up the above sum into two sums, corresponding to the ranges
$ r < B$ and $ r \geq B$, for some $B>1$.

We estimate each summand in the range $ r <B$ using Theorem
\ref{countphi}, and each summand in the range $ r \geq B$ using Lemma
\ref{lemUunif}, to respectively obtain
\begin{equation*}
\begin{array}{rcl}
\displaystyle\sum_{f\in\overline{\Y_{m_1, r }(\Sigma)}}
c_{q_1}(f)d_{m_1}(f)\lambda_{\ell_1}(f)
\Psi_1\Bigl(\frac{q_1^2|\Delta(f)|}{X}\Bigr)
&\ll&
\displaystyle\frac{X^{1+\epsilon}(\ell, r )}{q_1^{2}m_1^{5/3} r ^2\ell}+
\frac{X^{5/6+\epsilon}(\ell,r)}{q_1^{5/3}r^{5/3}\ell^{1/3}}+\ell q_1^2 m_1^{12} r ^{8}X^\epsilon
\\[.2in]&\ll&
\displaystyle\frac{X^{1+\epsilon}}{q_1^{2}m_1^{5/3} r \ell}+
\frac{X^{5/6+\epsilon}}{q_1^{5/3}r^{2/3}\ell^{1/3}}+\ell q_1^2 m_1^{12} r ^{8}X^\epsilon;
\\[.2in]
\displaystyle\sum_{f\in\overline{\Y_{m_1, r }(\Sigma)}}
c_{q_1}(f)d_{m_1}(f)\lambda_{\ell_1}(f)
\Psi_1\Bigl(\frac{q_1^2|\Delta(f)|}{X}\Bigr)
&\ll&
\displaystyle\frac{X^{1+\epsilon}}{q_1^2m_1^{5/3} r ^2}.
\end{array}
\end{equation*}
The second bound is simply an application of the tail estimate of
Lemma~\ref{lemUunif}. The first bound is more complicated, and
we explain how it is derived. Summing over $\overline{\Y_{m_1, r
  }(\Sigma)}$ can be replaced by summing a function $\phi\chi_\Sigma$
over $\overline{V(\Z)}$, where $\phi$ is defined modulo $m_1^2r^2$ and
$\chi_\Sigma$ is the indicator function defined in \S\ref{s_sieve_max}
before Corollary~\ref{lemlamer}. In the above equation, we are
therefore summing a function defined over $r^2m_1^3q_1\ell r_\Sigma$
(here, we also use Lemma \ref{switchek}).  Moreover $q_1$ is
squarefree, and the function defined modulo $\ell$ is
$\lambda_{\ell_1}$. Therefore, the error term with applying Theorem
\ref{countphi} is bounded by $\ll \ell q_1^2m_1^{12}r^8X^\epsilon$.

We now estimate the first and second main terms. The density of the
first main term follows from the uniformity estimates and the bound
$\cA_{\ell_1}(\lambda_{\ell_1}) \ll\frac{1}{\ell}$ from
Lemma~\ref{lemAC-lambda}. The second main term computation follows
similarly using the bound $\cC_{\ell_1}(\lambda_{\ell_1})\ll
\frac{1}{\ell^{1/3}}$ from Lemma~\ref{lemAC-lambda}.

Adding the above bounds over the appropriate ranges of $r$ yields
\begin{equation*}
\begin{array}{rcl}
&&\displaystyle\sum_{f\in \overline{\U_{m_1q_1}(\Sigma)}}e_k(f)
\lambda_n(f)V\Bigl(\frac{nkm_1^2}{\rad(m_1)^2\sqrt{|\Delta(f)|}}
\Bigr)\Psi\Bigl(\frac{|\Delta(f)|}{X}\Bigr)
\\[.2in]
&&\displaystyle\quad
\ll \frac{X^{1+\epsilon} \log B}{q_1^2m_1^{5/3}\ell} +
\frac{X^{5/6+\epsilon} B^{1/3}}{q_1^{5/3}\ell^{1/3}}+
\ell q_1^2 m_1^{12} B^9X^{\epsilon} + \frac{X^{1+\epsilon}}{q_1^2 m_1^{5/3} B}.
\end{array}
\end{equation*}
Choosing $B=X^{\kappa_1}$ concludes the proof of the proposition.
\end{proof}

Let notation be as in the beginning of this section. We have
\begin{equation*}
\begin{array}{rcl}
&&\displaystyle\sum_{f\in\overline{\W_q(\Sigma)}}\bigl(D(\tfrac 12,f)-S(f)\bigr)
\Psi\Bigl(\frac{\Delta(f)}{X}\Bigr)\\[.2in]&=&
\displaystyle\sum_{f\in\overline{\W_q(\Sigma)}}\Psi\Bigl(\frac{\Delta(f)}{X}\Bigr)
\sum_{k=1}^\infty\frac{e_k(f)k^{1/2}}{\rad(\ind(f))}\sum_{n=1}^\infty
\frac{\lambda_n(f)}{n^{1/2}}V^{{\rm sgn}(\Delta(f))}\Bigl(
\frac{\ind(f)^2kn}{\rad(\ind(f))^2|\Delta(f)|^{1/2}}\Bigr)
\\[.2in]&=&
\displaystyle\sum_{m=1}^{\infty}\sum_{f\in\overline{\U_{mq}(\Sigma)}}
\Psi\Bigl(\frac{\Delta(f)}{X}\Bigr)
\sideset{}{^\flat}\sum_{k\geq 1}
\frac{e_k(f)k^{1/2}}{q_1\rad(m_1)}\sum_{n=1}^\infty
\frac{\lambda_n(f)}{n^{1/2}}V^{{\rm sgn}(\Delta(f))}\Bigl(
\frac{m_1^2kn}{\rad(m_1)^2|\Delta(f)|^{1/2}}\Bigr)
\\[.2in]&=&
\displaystyle\sum_{m=1}^{X^\eta}\sum_{f\in\overline{\U_{mq}(\Sigma)}}
\Psi\Bigl(\frac{\Delta(f)}{X}\Bigr)
\sideset{}{^\flat}\sum_{k\geq 1}\frac{e_k(f)k^{1/2}}{q_1\rad(m_1)}
\sum_{n\le \frac{X^{1/2+\epsilon}}{k}}
\frac{\lambda_n(f)}{n^{1/2}}V^{{\rm sgn}(\Delta(f))}\Bigl(
\frac{m_1^2kn}{\rad(m_1)^2|\Delta(f)|^{1/2}}\Bigr)
\\[.2in]
&&\displaystyle+O_{\epsilon,\kappa_\downarrow,\Sigma,\Psi}\Bigl(\frac{X^{1-\eta+2\kappa_\downarrow+\epsilon}}{q}\Bigr),
\end{array}
\end{equation*}
where the final estimate follows from Corollary \ref{corborderrange}, and the rapid decay 
of $V^{\pm}$ to truncate the $n$-sum, and where the $\flat$ above indicates that the sum 
over $k$ is supported on
multiples of $q_1$ and ranges only over integers whose prime factors
are all divisors of $mq$.

Next, we truncate the sum over $k$ as follows. For the next two results, we will write $k=k_1k_2$, where $k_1$ is cubefree, $k_2$ is cubeful, and $(k_1,k_2)=1$.
\begin{lemma}
For every small enough $\kappa_2>0$, $X\ge 1$, and $q_1,m_1$ as above (i.e., satisfying $m_1\leq 
X^{2\eta}$ and $q_1 \ge X^{1/8-\eta-\kappa_{\downarrow}}$), we have
\begin{equation}\label{eqkupbd}
\begin{array}{rcl}
\displaystyle\sum_{m=1}^{X^\eta}\sum_{\substack{f\in\overline{\U^{\rm irr}_{mq}}\\ |\Delta(f)|<X}}
\sideset{}{^\flat}\sum_{\substack{k\\k_2>X^{3\kappa_2}}}\frac{|e_k(f)|k^{1/2}}{q_1\rad(m_1)}
\sum_{n\le \frac{X^{1/2+\epsilon}}{k}}
\frac{|\lambda_n(f)|}{n^{1/2}}
&\ll_{\epsilon,\kappa_2}&
\displaystyle\frac{X^{1-\kappa_2+4\eta+2\kappa_\downarrow+\epsilon}}{q}.
\end{array}
\end{equation}
\end{lemma}
\begin{proof}
The integers $k$ that arise range over products of powers of primes
dividing $mq$.  It follows from
Proposition~\ref{ekbound} that
\[
e_{k}(f) \ll_\epsilon \frac{\rad(k_2)^2}{k_2} X^\epsilon
\le k_2^{-1/3} X^{\epsilon} < X^{-\kappa_2 + \epsilon}.
\] 
Hence the sums over $n$ and $k$ can be bounded as follows: we have
\begin{equation*}
\begin{array}{rcl}
\displaystyle \sideset{}{^\flat}\sum_{\substack{k\\k_2>X^{3\kappa_2}}}\frac{|e_k(f)|k^{1/2}}{q_1\rad(m_1)}
\sum_{n\le \frac{X^{1/2+\epsilon}}{k}}
\frac{|\lambda_n(f)|}{n^{1/2}}&\ll_{\epsilon}&
\displaystyle
\frac{X^{1/4+\epsilon}}{q_1\rad(m_1)}
\sum_{\substack{k\\k_2>X^{3\kappa_2}}}
|e_k(f)|
\\[.2in]&\ll_{\epsilon,\kappa_2}&
\displaystyle\frac{X^{1/4-\kappa_2+2\epsilon}}{q_1\rad(m_1)}.
\end{array}
\end{equation*}
We already know from Lemma~\ref{lemUunif} that
\begin{equation*}
\sum_{\substack{f\in\overline{\U^{\rm irr}_{mq}}\\ |\Delta(f)|<X}} 1
\ll_\epsilon  \frac{X^{1+\epsilon}}{m_1^{5/3}q_1^2}.
\end{equation*}
Therefore, the left-hand side of \eqref{eqkupbd} is bounded by
\begin{equation*}
\ll_{\epsilon,\kappa_2}
X^{5/4 -\kappa_2 + 3\epsilon} \cdot
\sum^{X^\eta}_{m=1}
\frac{1}{m_1^{5/3}\rad(m_1) q_1^3}
\ll_{\epsilon,\kappa_2}
\frac{X^{5/4-\kappa_2+\eta+3\epsilon}}{q_1^3}\le
\frac{X^{1-\kappa_2+3\eta+2\kappa_\downarrow+3\epsilon}}{q_1},
\end{equation*}
which is sufficient because $q_1\ge q/m$ and $m\le X^\eta$.
\end{proof}

We input Proposition \ref{propafeeb}, which bounds the
sum over $f$, and obtain with Corollary~\ref{corborderrange} and \eqref{eqkupbd}:
\begin{equation}\label{eqafterfsum}
\begin{array}{rcl}
&&\displaystyle\sum_{f\in\overline{\W_q(\Sigma)}}\bigl(D(\tfrac 12,f)-S(f)\bigr)
\Psi\Bigl(\frac{\Delta(f)}{X}\Bigr)\ll_{\epsilon,\kappa_2,\Sigma,\Psi}
\\[.25in]
&&
\displaystyle\sum_{m_1=1}^{X^{2\eta}}
\displaystyle \sideset{}{^\flat}\sum_{\substack{k\\k_2\leq X^{3\kappa_2}}}
\sum_{n\le \frac{X^{1/2+\epsilon}}{k}}
\frac{k^{1/2}}{n^{1/2}q_1\rad(m_1)}
X^{\epsilon} H(n,m_1,q_1;X)
%\\[.2in]&&\displaystyle 
+\frac{X^{1-\eta+2\kappa_\downarrow+\epsilon}}{q}
+\frac{X^{1-\kappa_2+4\eta+2\kappa_\downarrow+\epsilon}}{q}.
\end{array}
\end{equation}

In our next result, we estimate the triple sum in \eqref{eqafterfsum}:
\begin{proposition}\label{propimpest}
For every square-free $q\in
[X^{1/8-\kappa_\downarrow},X^{1/8+\kappa_\uparrow}]$ and $X\ge 1$, we have
\begin{equation*}
\sum_{m_1=1}^{X^{2\eta}}
\displaystyle \sideset{}{^\flat}\sum_{\substack{k\\k_2\leq X^{3\kappa_2}}}
\sum_{n\le \frac{X^{1/2+\epsilon}}{k}}
\frac{k^{1/2}}{n^{1/2}q_1\rad(m_1)}
H(n,m_1,q_1;X)\ll_{\epsilon,\kappa_1,\kappa_2} H(q;X),
\end{equation*}
where $H(q;X)$ is the sum of the final terms in Equations
\eqref{eqest1}, \eqref{eqest2}, \eqref{eqest3}, and \eqref{eqest4}.
\end{proposition}
\begin{proof}
In this proof we shall write $\ll$ as a
shorthand for $\ll_{\epsilon,\kappa_1,\kappa_2}$. 
As before, we write  $n=n_1\ell_1$, where $n_1$ is only divisible by primes
dividing $q$ and $(\ell_1, q)=1$, and denote the radical of $\ell_1$ by $\ell$.
For convenience, we recall the definition of $H(n,m_1,q_1;X)$:
\begin{equation*}
H(n,m_1,q_1;X)=\frac{X}{q_1^2m_1^{5/3}\ell} +
\frac{X^{5/6+\kappa_1/3}}{q_1^{5/3}\ell^{1/3}}+
\ell q_1^2 m_1^{12} X^{9\kappa_1} + \frac{X^{1-\kappa_1}}{q_1^2 m_1^{5/3}}.
\end{equation*}
To prove the proposition, we take each term in $H(n,m_1,q_1;X)$ by turn, and sum it over $n$,
$k$, and $m_1$. The sum over $n$ is broken up into sums over $n_1$ and $\ell_1$. Note that since $n_1$ is only divisible by primes dividing $q$, the presence of $1/n^{1/2}$ in the sum (and no $n_1$'s in $H(n,m_1,q_1;X)$) means that the sum over $n_1$ can be ignored, at the cost of the harmless factor $O(X^\epsilon)$. Indeed, we have
\begin{equation*}
\sum_{n_1}\frac{1}{n_1^{1/2}}\leq \prod_{p\mid q}\Bigl(1+\frac{1}{p^{1/2}}+\frac{1}{p}+\cdots\Bigr)\ll
2^{\omega(q)}\ll_\epsilon X^\epsilon.
\end{equation*}

Next note that $k=k_1k_2$, where $k_1$ is cubefree, and $k$ is only divisible by primes dividing $mq=m_1q_1$. Hence, we have $k_1\leq q_1^2\rad(m_1)^2$, and in conjunction with $k_2\leq X^{3\kappa_2}$, we also have $k\leq q_1^2X^{2\eta+3\kappa_2}$.
We begin with the first term: in this case, the
sums over $\ell_1$ and $m_1$ converge, and so we have
\begin{equation}\label{eqest1}
\begin{array}{rcl}
\displaystyle\frac{X^{1+\epsilon}}{q_1^2}\sum_{m_1=1}^{X^{2\eta}}
\sideset{}{^\flat}\sum_{\substack{k\\k_2\leq X^{3\kappa_2}}}
\sum_{\ell_1\le \frac{X^{1/2+\epsilon}}{k}}
\frac{k^{1/2}}{\ell_1^{1/2}q_1\rad(m_1)}
\frac{1}{m_1^{5/3}\ell}
&\ll& \displaystyle
\frac{X^{1+\epsilon}}{q_1^3}\sideset{}{^\flat}\sum_{\substack{k\\k_2\leq X^{3\kappa_2}}}
k^{1/2}
\\[.2in]&\ll& \displaystyle
\frac{X^{1+\epsilon}}{q_1^3}\cdot q_1X^{\eta+3\kappa_2/2}\sum_{k_2\leq X^{3\kappa_2}}1
\\[.2in]&\ll& \displaystyle
\frac{X^{7/8+3\eta+(9/2)\kappa_2+\kappa_\downarrow+\epsilon}}{q},
\end{array}
\end{equation}
where the final estimate follows because $q_1\gg qX^{-\eta}$ and $q\gg X^{1/8-\kappa_\downarrow}$.
Similarly, for the second term, we have
\begin{equation}\label{eqest2}
\frac{X^{5/6+\kappa_1/3+\epsilon}}{q_1^{5/3}}
\sum_{m_1=1}^{X^{2\eta}}
\sideset{}{^\flat}\sum_{\substack{k\\k_2\leq X^{3\kappa_2}}}
\sum_{\ell_1\le \frac{X^{1/2+\epsilon}}{k}}
\frac{k^{1/2}}{\ell_1^{1/2}q_1\rad(m_1)}
\frac{1}{\ell^{1/3}}\ll \frac{X^{11/12+(8/3)\eta+\kappa_1/3+4\kappa_2+\epsilon}}{q^2}.
\end{equation}
To estimate the third term, we write
\begin{equation}\label{eqest3}
\begin{array}{rcl}
\displaystyle q_1^2X^{9\kappa_1+\epsilon}
\sum_{m_1=1}^{X^{2\eta}}
\sideset{}{^\flat}\sum_{\substack{k\\k_2\leq X^{3\kappa_2}}}
\sum_{\ell_1\le \frac{X^{1/2+\epsilon}}{k}}
\frac{k^{1/2}}{\ell_1^{1/2}q_1\rad(m_1)}
\ell m_1^{12}&\ll&
\displaystyle q_1X^{3/4+9\kappa_1+26\eta+\epsilon}
\sideset{}{^\flat}\sum_{\substack{k\\k_2\leq X^{3\kappa_2}}}\frac{1}{k}
\\[.2in]&\ll&\displaystyle
\frac{X^{7/8+9\kappa_1+26\eta+\kappa_\uparrow+\epsilon}}{q},
\end{array}
\end{equation}
where the final estimate follows because non-zero values of $k$ are
all multiples of the squarefree $q_1$; see Proposition
\ref{propklarge}. Finally, we have
\begin{equation}\label{eqest4}
\frac{X^{1-\kappa_1+\epsilon}}{q_1^2}
\sum_{m_1=1}^{X^{2\eta}}
\sideset{}{^\flat}\sum_{\substack{k\\k_2\leq X^{3\kappa_2}}}
\sum_{\ell_1\le \frac{X^{1/2+\epsilon}}{k}}
\frac{k^{1/2}}{\ell_1^{1/2}q_1\rad(m_1)}
\frac{1}{m_1^{5/3}}\ll \frac{X^{1-\kappa_1+3\eta+3\kappa_2+2\kappa_\downarrow+\epsilon}}{q}.
\end{equation}
This concludes the proof of Proposition \ref{propimpest}.
\end{proof}

We are now ready to prove the main result of this subsection.

\begin{proposition}\label{propmediumrange}
There exist positive constants $\kappa_\uparrow,\kappa_\downarrow,\kappa_3$ such that
the following holds. For every $X \ge 1$ and every squarefree $q\in
[X^{1/8-\kappa_\downarrow},X^{1/8+\kappa_\uparrow}]$, we have
\begin{equation*}
\sum_{f\in \overline{\W_q(\Sigma)}}
\Bigl(S(f)-D(\tfrac 12,f)\Bigr)
\Psi\Bigl(\frac{|\Delta(f)|}{X}\Bigr)= 
O_{\Sigma,\Psi}\Bigl(\frac{X^{1-\kappa_3}}{q}\Bigr).
\end{equation*}
\end{proposition}
\begin{proof}
We apply \eqref{eqafterfsum} and then apply Proposition
\ref{propimpest}. It is only necessary to ensure that the exponent of
$X$ is less than $1$ for each of the 6 different error terms. This is
easily done. First, we temporarily pick any positive $\kappa_\uparrow$
and $\kappa_\downarrow$. Next we pick $\eta>2\kappa_\downarrow$. Then
we pick $\kappa_2>4\eta+2\kappa_\downarrow$ and $\kappa_1>3\eta+2\kappa_\downarrow+3\kappa_2$.  This takes care
of~\eqref{eqest4} and of the last two terms of \eqref{eqafterfsum}.

Finally, to ensure that the exponents of $X$ in the final terms of
\eqref{eqest1}, \eqref{eqest2}, and \eqref{eqest3} are less than $1$,
we simply divide our constants
$\kappa_\uparrow,\kappa_\downarrow,\eta,\kappa_1,\kappa_2$ by the same
sufficiently large number.
\end{proof}

We now put together our results for the border range and the large
range.

\begin{theorem}\label{thmStoD}
There exists an absolute constant $\varkappa>0$ such that for every
$X\ge 1$ and every squarefree $q\ge X^{1/8-\varkappa}$, we have
\begin{equation*}
\sum_{f\in\overline{\W_q(\Sigma)}} 
\Bigl(
 S(f) -  D(\tfrac 12,f)
\Bigr)
\Psi\Bigl(\frac{|\Delta(f)|}{X}\Bigr)
=
O_{\Sigma,\Psi} \Bigl(\frac{X^{1-\varkappa}}{q}\Bigr).
\end{equation*}
\end{theorem}

\begin{proof}
We combine Corollary \ref{corlargerange} and Proposition \ref{propmediumrange}, where we 
choose 
$\varkappa 
= \min(\kappa_{\uparrow},\kappa_3)$.
\end{proof}

\begin{corollary}\label{corStoD}
There exists an absolute constant $\mu>0$ such that for every
$X\ge 1$, we have
\begin{equation}\label{eqStoD}
\sum_{\substack{q \;{\rm squarefree}\\
    q\geq X^{1/8-\mu}}}
\left|
\sum_{f\in\overline{\W_q(\Sigma)}} 
\Bigl(
 S(f) -  D(\tfrac 12,f)
\Bigr)
\Psi\Bigl(\frac{|\Delta(f)|}{X}\Bigr)
\right|
=
O_{\Sigma,\Psi} \bigl(X^{1-\mu}\bigr).
\end{equation}
\end{corollary}

\begin{proof}
Adding up the above result for $q\geq X^{1/8-\varkappa}$, we note that
$\{f\in \overline{\W_q(\Sigma)}:|\Delta(f)|<X\}$ is empty for $q\ge X^{1/2}$ because $\Delta(f) = \ind(f)^2 
\Delta(K_f) \ge q^2 \Delta(K_f) \ge q^2$ for $f\in \W_q$.
\end{proof}

\begin{remark}
\begin{comment}
Choosing $X^{\kappa_1}=\frac{X^{1/10}}{\ell^{1/10}q_1^{2/5}m_1^{41/30}}$
in the proof of Proposition~\ref{propafeeb} yields the equality of the
last two terms and the following value
\begin{equation*}
H(n,m_1,q_1;X)=\frac{X}{q_1^2m_1^{5/3}\ell} +
\frac{X^{26/30}}{q_1^{27/15}m_1^{41/19}\ell^{11/30}}+\frac{\ell^{1/10} X^{9/10}}{q_1^{8/5}m_1^{4/15}}.
\end{equation*}
Then the remainders \eqref{eqest3} and \eqref{eqest4} are equal, and
we may further choose $\kappa_\uparrow$ so that they are also equal to
the remainder in Corollary~\ref{corlargerange} which yields a specific
choice of the constant $\kappa_\uparrow$.  
\end{comment}
An admissible set of
values of the constants is as follows: $ \kappa_\downarrow =
\tfrac{1}{9000}$, $\kappa_{\uparrow}=\tfrac{1}{300}$,
$\eta=\tfrac{1}{3000}$ $\kappa_1=\tfrac{1}{100}$,
$\kappa_2=\tfrac{1}{600}, \varkappa = \tfrac{1}{10000}$.  To verify the
admissibility of these numerical values, it suffices to insert them in
each of the remainder terms of Proposition~\ref{p_main-term},
Corollary~\ref{corlargerange}, Corollary~\ref{corborderrange},
\eqref{eqkupbd}, \eqref{eqest1}, \eqref{eqest2}, \eqref{eqest3}, and
\eqref{eqest4}.
\end{remark}

\subsection{Counting suborders}\label{sSuborders}

In this subsection we prove Theorem \ref{propSN} by conditionally
bounding
\begin{equation*}
\sum_{q>X^{1/8-\varkappa}}\sum_{f\in\overline{\W_q(\Sigma)}}S(f).
\end{equation*}
Note that by Corollary \ref{corStoD}, we may replace $S(f)$ in the
above sum by $D(\tfrac 12,f)$. The advantage of using $D(\tfrac 12,f)$ over
$S(f)$ is that the values of $D(\tfrac 12,f)$ for binary cubic forms $f$
corresponding to suborders of a fixed cubic field $K$ can be
simultaneously controlled in terms of $L(\tfrac 12,\rho_K)$. To this
end, we start by recalling the following result, due to works of
Shintani \cite{Shintani} and Datskovsky--Wright~\cite{DW} (see
\cite[\S1.2]{KNT}), giving an explicit formula for the counting
function of suborders $R$ of a fixed cubic field $K$.
\begin{proposition}\label{propSDW}
  Let $K$ be a cubic field with ring of integers $\O_K$. For an order
  $R\subset\O_K$, let $\ind(R)$ denote the index of $R$ in $\O_K$. Then
  \begin{equation*}
\sum_{R\subset 
\O_K}\frac{1}{\ind(R)^s}=\frac{\zeta_K(s)}{\zeta_K(2s)}\zeta_\Q(3s)\zeta_\Q(3s-1).
  \end{equation*}
\end{proposition}
We thus obtain the following corollary regarding the number $N_K(Z)$
of orders of $\O_K$ with index less than $Z$ for a cubic field $K$.
\begin{corollary}\label{suborders}
For every $\epsilon>0$, $Z\ge 1$ and cubic field $K$, we have
\begin{equation*}
  \begin{array}{rcl}
N_K(Z)\ll_\epsilon Z^{1+\epsilon}|\Delta(K)|^\epsilon.
  \end{array}
\end{equation*}
The implied constant is independent of $K$ and $Z$.
\end{corollary}
\begin{proof}
This follows from Perron's formula integrating along the 
vertical line $\Re(s)=1+\epsilon$.
\end{proof}

The above result can be used to give a very useful bound on the sum of
$D(\tfrac 12,f)$, over $f\in\overline{\W_q(\Sigma)}$ for $q$ greater than some
positive $Q$.

\begin{lemma}
For every $Q,X\ge 1$ and $\epsilon>0$,
\begin{equation}\label{eqsumK}
 \sum_{q\ge Q} \sum_{
\substack{f\in \overline{\W_q(\Sigma)}\\ |\Delta(f)|< X}}
|D(\tfrac12,f)|
\ll_{\epsilon,\Sigma}
X^{\frac12 + \epsilon} \sum_{2^{\mathbb N}\ni Y\le X/Q^2}
{Y^{-\frac12}}
\sum_{\substack{K\in \FF_\Sigma \\
    Y\leq|\Delta(K)|<2Y}}|L(\tfrac 12,\rho_K)|.
\end{equation}
\end{lemma}
\begin{proof}
Consider a real number $Y$ with $Y\ll X/Q^2$ and a cubic field
$K$ such that $Y\leq |\Delta(K)|<2Y$. Then the number of binary cubic
forms $f\in \cup_{q\ge Q} \overline{W_q(\Sigma)}$ such that $|\Delta(f)|<X$ and $K_f = K$ is bounded by
\[
N_K\Bigl(\frac{X^{\frac12}}{Y^{\frac12}}\Bigr) = O_\epsilon\bigl(X^{1/2+\epsilon}/Y^{1/2}\bigr),
\]
using Corollary~\ref{suborders}. 

Summing over all $K$ in the discrimant range $Y\leq|\Delta(K)|< 2Y$,
and then summing over $Y\in 2^{\mathbb N}$ such that the dyadic ranges
$[Y,2Y)$ cover (more than) the interval $[1,X/Q^2]$, we capture the
  sum over $f\in \overline{\W_q(\Sigma)}$, for all $q>Q$, such that
  $|\Delta(f)|< X$.

Recall from~\eqref{boundEf} that we have $D(\tfrac12,f) =
L(\tfrac12,\rho_{K_f}) E(\tfrac12,f)$ and $E(\tfrac12,f)
=\prod_{p|\Delta(f)} (1+O(p^{-\frac12})) = |\Delta(f)|^{o(1)}$, which
concludes the proof of the lemma.
\end{proof}

The above lemma yields the following consequence, which clarifies how
nonnegativity is used by us.
\begin{corollary}\label{cornnest}
For every cubic field $K\in\FF_\Sigma$, assume that $L(\tfrac12,\rho_K)\geq
0$. Then for $Q,X\geq 1$, we have
\begin{equation}\label{eqcornonnegbound}
\sum_{q\geq Q}
\sum_{\substack{f\in\overline{\W_q(\Sigma)}\\|\Delta(f)|<X}} D(\tfrac12,f)
\ll_{\epsilon,\Sigma} X^{29/28+\epsilon}Q^{-15/14}.
\end{equation}
\end{corollary}
\begin{proof}
First note that the assumption $L(\tfrac12,\rho_K)\geq 0$ for all cubic
fields $K$ implies that $D(\tfrac12,f)\geq 0$ for all irreducible integral
binary cubic forms. Thus, we may apply the previous lemma to estimate
the left-hand side of \eqref{eqcornonnegbound}.

From Theorem \ref{thuncondupbd} (using a smooth function which
dominates the characteristic function of $[1,2]$), we obtain
\begin{equation*}
\sum_{\substack{K\in \FF_\Sigma \\ Y\leq|\Delta(K)|<2Y}}|L(\tfrac
12,\rho_K)|\ll_{\epsilon,\Sigma} Y^{\frac{29-28\delta}{28-16\delta}+\epsilon}, 
\end{equation*}
for $\delta = 1/128$. Even the bound with $\delta=0$ in conjunction
with \eqref{eqsumK}, yields the result.
\end{proof}

We are now ready to prove Theorem \ref{propSN}.

\begin{proof}[Proof of Theorem \ref{propSN}]
Proof assuming strong subconvexity: The hypothesis (S) would imply
that the central value in the right-hand side of \eqref{eqsumK} is bounded by $Y^{\frac 16
  -\vartheta}$.  Hence the bound in \eqref{eqsumK} becomes $X^{\frac
  12 + \epsilon} (X/Q^2)^{\frac{2}{3}-\vartheta}$.  We pick
$Q=X^{\frac{1}{8}-\kappa_{\downarrow}}$ with
$\epsilon,\kappa_\downarrow >0$ sufficiently small such that $\tfrac12
+ \epsilon + (\tfrac34 + 2 \kappa_\downarrow)(\tfrac 23 - \vartheta)<1$.
Proposition \ref{p_main-term}, together with
Corollary~\ref{corlargerange} and Corollary~\ref{corStoD}, now yield
the result.

\medskip

\noindent Proof assuming nonnegativity: We pick $Q=X^{1/8-\varkappa}$,
with $\varkappa$ as in Theorem \ref{thmStoD}. It follows that we have
\begin{equation*}
\sum_{q\geq Q}\sum_{\substack{f\in \overline{\W_q(\Sigma)}}}S(f) \Psi\Bigl (\frac{|\Delta(K)|}{X}\Bigr)=
\sum_{q\geq Q}\sum_{\substack{f\in \overline{\W_q(\Sigma)}}}D(\tfrac12,f) \Psi\Bigl (\frac{|\Delta(K)|}{X}\Bigr)
+O_{\epsilon,\Sigma,\Psi}(X^{1-\varkappa+\epsilon}).
\end{equation*}
Since we are assuming hypothesis (N), Corollary \ref{cornnest} implies
that we have
\begin{equation*}
  \sum_{q\geq Q}
  \sum_{\substack{f\in \overline{\W_q(\Sigma)}\\ |\Delta(f)|\ll X}}
  D(\tfrac12,f)\ll_{\epsilon,\Sigma}
X^{101/112+\epsilon},
\end{equation*}
which is sufficiently small. The result now follows from Proposition
\ref{p_main-term}.
\end{proof}

\section{Proofs of Theorems \ref{thmnv2} and \ref{thmMoment}}\label{sec:proof}

In addition to the quantity $A_\Sigma(X)$, that we defined in
\eqref{eqnvzsieve}, we also define
\index{$MA_\Sigma(X)$, sum of {$\lvert L(\tfrac12,\rho_K) \rvert $} for $K\in \FF_\Sigma(X)$}
\begin{equation*}
\begin{array}{rcl}
MA_\Sigma(X)&:=&\displaystyle
\sum_{\substack{K\in \FF_\Sigma\\ X\le |\Delta(K)|<2X}}
|L(\tfrac 12,\rho_K)|;
\\[.2in]\displaystyle
PA_\Sigma(X)&:=&\displaystyle
\sum_{\substack{K\in \FF_\Sigma\\ X/2 \le |\Delta(K)|< 3X \\L(\tfrac 12,\rho_K)\ge 
0}}
L(\tfrac 12,\rho_K).
\end{array}
\end{equation*}
The letter M stands for \emph{maximal} and the letter P for \emph{positive}. We note that while $A_\Sigma$ is defined with the sum over $K$ weighted by $1/|\Aut(K)|$, the sums over $K$ in $MA_\Sigma$ and $PA_\Sigma$ are unweighted. Since the weight only affects the $O(X^{1/2})$ cyclic cubics, weighted sums and unweighted sums agree up to a negligible error term of $O(X^{3/4})$.
\index{$PA_\Sigma(X)$, sum of $L(\tfrac12,\rho_K)\ge 0$ for $K\in \FF_\Sigma(X)$} 

\begin{proposition}\label{MAlePA}
For every $\epsilon>0$ and $X\ge 1$, we have the asymptotic inequality
\[
 MA_\Sigma(X) \le 2 PA_\Sigma(X) + 
 O_{\epsilon,\Sigma}\bigl(X^{\frac{29-28\delta}{28-16\delta} +\epsilon}\bigr).
\]
\end{proposition}
\begin{proof}
We let $\Psi_1:\R_{>0}\to [0,1]$ be a smooth function compactly supported on the 
interval $[\frac12,3]$ such that $\Psi_1(t)=1$ for $t\in [1,2]$.
We have an inequality followed by a basic identity
\begin{equation}\label{eqbasicidentity}
\begin{aligned}
MA_\Sigma(X) &\le 
\sum_{\substack{K\in \FF_\Sigma}}
|L(\tfrac 12,\rho_K)| \Psi_1\Bigl(\frac{|\Delta(K)|}{X}\Bigr)
\\
&=
2
\sum_{\substack{K\in \FF_\Sigma\\ L(\tfrac 12,\rho_K)\ge 0}}
L(\tfrac 12,\rho_K) \Psi_1\Bigl(\frac{|\Delta(K)|}{X}\Bigr)
-
\sum_{\substack{K\in \FF_\Sigma}}
L(\tfrac 12,\rho_K) \Psi_1\Bigl(\frac{|\Delta(K)|}{X}\Bigr),
\end{aligned}
\end{equation}
which follows from $|x| = 2\max(x,0) - x$ for every $x\in \R$.
The first sum is $\le 2 PA_\Sigma(X)$. (Note that 
in the respective definitions of $MA_\Sigma(X)$ and $PA_\Sigma(X)$, the discriminant range 
has increased from  $X\le |\Delta(K)|<2X$ to $X/2 \le |\Delta(K)|< 3X$ for this purpose).
The second sum is equal to $A_\Sigma(X)$ (up to negligible error) for which we have established the 
estimate~\eqref{eqLfavgupbound}. This concludes the proof.
\end{proof}

We finally arrive at the proof of our main result of this paper. 
In Section~\ref{sec:average}, we have estimated the terms $q<Q$ of the first moment 
$A_\Sigma(X)$. In Section~\ref{sConditional}, we have estimated for the other terms 
$q\ge Q$ the difference $S(f) - D(\tfrac12,f)$. The conclusion of all these results is 
summarized in the 
following which was stated in the introduction as Theorem~\ref{thmMoment}:
\begin{theorem}
There is an absolute constant $\mu>0$ such that the following holds.
For every $0<\nu\le \mu$, $\epsilon>0$, and $X\ge 1$,
\begin{equation}\label{eqASignewerror}
A_\Sigma(X)-C_\Sigma\cdot X \bigl(\log X +  \widetilde \Psi'(1) \bigr)  - C'_\Sigma \cdot X 
\ll_{\epsilon,\nu,\Sigma,\Psi}
X^{1+\epsilon-\nu}
+ 
X^{1/2+\epsilon} \cdot \sum_{ 2^{\mathbb N} \ni Y\le 
X^{3/4+\nu}}\frac{MA_\Sigma(Y)}{Y^{1/2}},
\end{equation}
where the sum over $Y$ is dyadic, namely $Y\in 2^{\mathbb N}$ is constrained to be a 
power 
of $2$.
\end{theorem}
\begin{proof}
The result will follow from Proposition \ref{p_main-term}, Corollary
\ref{corStoD} and \eqref{eqsumK}.
It follows from Proposition \ref{p_main-term} that
\[
 A_{\Sigma}(X) - 
C_\Sigma\cdot X \bigl(\log X +  \widetilde \Psi'(1) \bigr)  - C'_\Sigma \cdot X
\ll_{\epsilon, \Sigma,\Psi}
\frac{X^{1+\epsilon}}{Q}+X^{\frac{11}{12}+\epsilon}
+
 Q^{2+\epsilon} X^{\frac34+\epsilon} +
\sum_{q\ge Q} 
\left|
\sum_{f\in \overline{W_q(\Sigma)}} S(f) \Psi\left(
\frac{|\Delta(f)|}{X}
\right)
\right|.
\]
Let $a>0$ be sufficiently small such that $\Psi(t)=0$ whenever $a^2t\ge 1$.
Choose $Q= a^{-1} X^{1/8-  \nu/ 2}$.
Using that $Q\gg_\Psi X^{1/8-\mu}$, we can apply Corollary~\ref{corStoD} to obtain the bound
\[
 \sum_{q\ge Q} 
\left|
\sum_{f\in \overline{W_q(\Sigma)}} 
\Bigl(
 S(f) -  D(\tfrac 12,f)
\Bigr)
 \Psi\left(
\frac{|\Delta(f)|}{X}
\right)
\right|
\ll_{\Sigma,\Psi} X^{1-\mu} \le X^{1-\nu}.
\]
The estimate \eqref{eqsumK} yields
\[  
\begin{aligned}
 \sum_{q\ge Q} 
\sum_{f\in \overline{W_q(\Sigma)}} 
\left|
D(\tfrac 12,f)
\right|
 \Psi\left(
\frac{|\Delta(f)|}{X}
\right)
&\ll_{\Psi} 
 \sum_{q\ge Q} 
\sum_{\substack{f\in \overline{W_q(\Sigma)}\\ |\Delta(f)|<X/a^2}} 
\left|
D(\tfrac 12,f)
\right|
 \\
&\ll_{\epsilon,\Sigma,\Psi}
X^{\frac12 + \epsilon}
\cdot \sum_{ 2^{\mathbb N} \ni Y\le 
(X/a^2)/Q^2}
\frac{MA_\Sigma(Y)}{Y^{1/2}}.
\end{aligned}
\]
It remains to observe that $(X/a^2)/Q^2 = X^{3/4 + \nu}$ to conclude the proof.
\end{proof}

We are now ready to prove our main Theorem \ref{thmnv2}.
Recall that the qualitative version in Theorem \ref{thmnv1} follows from Theorem 
\ref{thmnv2}.

\begin{proof}[Proof of Theorem \ref{thmnv2}]
Recall that $C_\Sigma>0$ in Proposition~\ref{p_main-term}. 
We distinguish two cases depending on the size of the sum of $MA_\Sigma(Y)$
in the right-hand
side of \eqref{eqASignewerror}.

In the first case, if the right-hand
side of \eqref{eqASignewerror} is $< X$, then we have
$A_\Sigma(X)\sim C_\Sigma \cdot X\cdot \log X$. In combination with
Theorem~\ref{p_Wu}, we obtain that $\gg_{\epsilon,\Sigma}
X^{\frac{3}{4}+\delta-\epsilon}$ cubic fields
$K\in\FF_\Sigma$ with $|\Delta(K)|<X$ satisfy $L(\tfrac 12,\rho_K)>0$.
Hence 
\begin{equation}\label{firstcase}
 \delta_\Sigma(X) \ge \frac{3}{4} + \delta - \epsilon - O_\epsilon\bigl(\frac{1}{\log X}\bigr),
\end{equation}
which is sufficient to imply Theorem \ref{thmnv2} in that case.

\medskip

Assume in the second case that the right-hand side of \eqref{eqASignewerror} is $\ge
X$, namely
\begin{equation*}
\sum_{2^{\mathbb N} \ni Y\le X^{3/4+\nu}}\frac{MA_\Sigma(Y)}{Y^{1/2}}\ge 
X^{1/2-\epsilon}.
\end{equation*}
This implies that there exists $Y\in 2^{\mathbb N}$ with $Y\leq X^{3/4+\nu}$ such that
$MA_\Sigma(Y)\ge X^{1/2-\epsilon}Y^{1/2}$. 
It follows from Proposition~\ref{MAlePA} that
\begin{equation*}
2 PA_\Sigma(Y)\ge X^{1/2-\epsilon}Y^{1/2} 
+O_{\epsilon,\Sigma}\bigl(Y^{\frac{29-28\delta}{28-16\delta} +\epsilon}\bigr).
\end{equation*}
Since $Y \leq X^{3/4+\nu}$, the error term is negligible. (The convexity bound 
$\delta=0$ 
suffices for this). We deduce in the second case:
\begin{equation}\label{eqYineq}
PA_\Sigma(Y)\gg X^{1/2-\epsilon}Y^{1/2}. 
\end{equation}

Theorem~\ref{p_Wu} and \eqref{eqYineq} imply that $\gg_{\epsilon}
X^{1/2-\epsilon}Y^{1/4+\delta-\epsilon}$ cubic fields
$K\in\FF_\Sigma$ with $|\Delta(K)|<Y$ satisfy the inequality
$L(\tfrac 12,\rho_K)>0$. 
Hence
\begin{equation}\label{secondcase}
 \delta_\Sigma(Y) \ge \frac{\log X}{2 \log Y} + \bigl(\frac 14+\delta -\epsilon\bigr) 
 -O_\epsilon\bigl(\frac{1}{\log Y}\bigr).
\end{equation}
Since~\eqref{eqYineq} implies that $Y\to \infty$ we deduce
\begin{equation*}
 \limsup_{X\to \infty} \delta_\Sigma(X) \ge \frac 34 + \delta,
\end{equation*}
because the inequality is satisfied either in the first case by $\delta_\Sigma(X)$ 
in~\eqref{firstcase} or in the 
second case by $\delta_\Sigma(Y)$ in \eqref{secondcase}, since $\frac{\log X}{\log Y} \ge \frac{4}{3+4\nu}\ge 
1$.

To conclude a lower bound on the $\liminf$, we need a lower bound on $Y$ in the second case.
Theorem~\ref{p_Wu} implies $PA_\Sigma(Y) = O_\epsilon(Y^{\frac 54 -\delta+\epsilon})$.
Together with \eqref{eqYineq}, this yields the following lower bound: 
\begin{equation}\label{lowerY}
Y\gg_\epsilon X^{\frac{2}{3-4\delta}-\epsilon}.
\end{equation}
This implies
\begin{equation}\label{deltaSigmage}
 \delta_\Sigma(X) \ge \frac 12 + \bigl(\frac 14 +\delta\bigr)\cdot \frac{2}{3-4\delta} - \epsilon - 
 O_\epsilon\bigl(\frac{1}{\log X}\bigr).
\end{equation}
The first two terms of~\eqref{deltaSigmage} simplify to $\frac{2}{3-4\delta}$, hence
\[  
 \liminf_{X\to \infty} \delta_\Sigma(X) \ge \frac{2}{3-4\delta}.
\]
This concludes the proof of Theorem \ref{thmnv2}.
\end{proof}

The same argument implies an \emph{Omega} result $MA_\Sigma(X) =\Omega_\Sigma(X)$ as 
$X\to 
\infty$. Namely, there is a sequence $X_k\to \infty$ such that $MA_\Sigma(X
_k)/X_k\to \infty$.
Indeed, in the first case of the proof of Theorem \ref{thmnv2}, we have $A_\Sigma(X)\sim C_\Sigma\cdot X\log X$. In the second case, we have 
\begin{equation*}
MA_\Sigma(Y) \ge X^{1/2-\epsilon} Y^{1/2} \ge Y^{7/6 -o(1)},
\end{equation*}
in view of  \(Y \le X^{3/4+\nu}\). Moreover we have seen that~\eqref{eqYineq} implies  \(Y\to \infty\), which enables to extract a sequence  \(X_k=Y\to \infty\) such that  \(MA_\Sigma(Y)/Y\to \infty\).

For completeness, we also record the following lower bound for 
the first moment: 
\begin{proposition}\label{corboundmoment}
For every $\epsilon>0$ and $X\ge 1$,
\[
\sum_{K\in \FF_\Sigma(X)} \bigl|L\bigl(\tfrac12,\rho_K\bigr)\bigr|\gg_{\epsilon,\Sigma} X^{\frac{5-4\delta}{6-8\delta}-\epsilon}.
\]
\end{proposition}

\begin{proof} 
Suppose first that we are in the first case of the proof of Theorem \ref{thmnv2}. Then we have $A_\Sigma(X)\sim C_\Sigma\cdot X\log X$, implying that the left-hand side of the above equation is $\gg_\Sigma X\log X$. Suppose instead that we are in the second case. Then the lower bound~\eqref{lowerY} for $Y$ implies the lower bound in 
Proposition~\ref{corboundmoment} 
as follows:
\[
 \sum_{K\in \FF_\Sigma(X)}
|L(\tfrac 12,\rho_K)| \ge 
 \sum_{K\in \FF_\Sigma(Y)}
|L(\tfrac 12,\rho_K)|
\gg_{\epsilon,\Sigma} X^{\frac12 -\epsilon} Y^{\frac12},
\]
and $\tfrac12 + \tfrac{1}{3-4\delta} = \tfrac{5-4\delta}{6-8\delta}$.
\end{proof}

\newgeometry{left=1cm,right=1cm,tmargin=2.5cm, marginpar=1cm}
{\small  \printindex}

\AtEndDocument{\bigskip{\footnotesize%
  \textit{E-mail address}, \texttt{ashankar@math.toronto.edu}\par
  \textsc{Department of Mathematics, University of Toronto, Toronto, ON, M5S 2E4, 
  Canada}\par
  \addvspace{\medskipamount}
  \textit{E-mail address}, \texttt{andesod@chalmers.se}\par
  \textsc{Department of Mathematical Sciences, Chalmers University of Technology and the University of\\
  \rule[0ex]{0ex}{0ex}\hspace{14pt}Gothenburg, SE-412 96 Gothenburg, Sweden}\par
  \addvspace{\medskipamount}
  \textit{E-mail address}, \texttt{npt27@cornell.edu}\par
  \textsc{Department of Mathematics, Cornell University, Ithaca, NY 14853, USA}
}}

\end{document}